%% file: mainv3.tex
\documentclass[11pt]{amsart}
\input{preamble}

\usepackage[utf8]{inputenc}
\usepackage{dsfont}
\usepackage[english]{babel}
\usepackage{csquotes}
\usepackage{mathrsfs}
\usepackage{comment}

\newcommand{\cM}{\mathcal{M}}

\newcommand{\cA}{\mathcal{A}}
\newcommand{\cB}{\mathcal{B}}
\newcommand{\cC}{\mathcal{C}}
\newcommand{\cD}{\mathcal D}
\newcommand{\cE}{\mathcal{E}}

\newcommand{\cL}{\mathcal{L}}
\newcommand{\cN}{\mathcal{N}}

\newcommand{\cP}{\mathcal{P}}
\newcommand{\cQ}{\mathcal{Q}}

\newcommand{\ra}{\rightarrow}

\theoremstyle{definition} 
\usepackage{hyperref}
\newtheorem*{genericthm*}{\thistheoremname}
\newcommand{\thistheoremname}{???}
\newcounter{genericthm}

\newenvironment{namedthm*}[1]
  {\renewcommand{\thistheoremname}{#1}%
   \refstepcounter{genericthm}%
   \begin{genericthm*}}
  {\end{genericthm*}}

\usepackage[
backend=biber,
style=alphabetic,
bibstyle=ieee-alphabetic,
sorting=nyt,
maxbibnames   = 99,
maxalphanames = 8,
maxcitenames  = 6,
minalphanames = 6,
minnames      = 4,
]{biblatex}
\begin{document}

\title{Rigidity results for group von Neumann algebras with diffuse center} 

\author{Ionu\c{t} Chifan}

\author{Adriana Fernández Quero}

\author{Hui Tan}

\begin{abstract} We introduce the first examples of groups $G$ with infinite center which are completely recognizable from their von Neumann algebras, $\mathcal{L}(G)$. Specifically, assume that $G=A\times W$, where $A$ is an infinite abelian group and $W$ is an ICC wreath-like product group \cite{cios22,amcos23} with property (T) and trivial abelianization.  Then whenever $H$ is an \emph{arbitrary} group such that $\mathcal{L}(G)$ is $\ast$-isomorphic to $\mathcal L(H)$, it must be the case that $H= B \times H_0$ where $B$ is infinite abelian and $H_0$ is isomorphic to $W$. Moreover, we completely describe the $\ast$-isomorphism between $\mathcal L(G)$ and $\mathcal L(H)$. This yields new applications to the classification of group C$^*$-algebras, including examples of non-amenable groups which are recoverable from their reduced C$^*$-algebras but not from their von Neumann algebras.
\end{abstract}

\maketitle


\section{Introduction}

In their pioneering work \cite{mvn43}, Murray and von Neumann associated in a natural way a von Neumann algebra, denoted by $\mathcal L(G)$, to every countable discrete group G. More precisely, $\mathcal L(G)$ is defined as the weak operator closure of the complex group algebra $\mathbb C G$ acting by left convolution on the Hilbert space $\ell^2G$ of square-summable functions on $G$. One can see that the central elements of $\mathcal L(G)$ are supported on the finite conjugacy classes of $G$; thus, $\mathcal L(G)$ is a factor when $G$ is an ICC group. The classification of group von Neumann algebras has since been a central theme in operator algebras driven by the following fundamental question: what aspects of the group $G$ are retained by  $\mathcal L(G)$?

\vspace{2mm}

From the inception of the subject, it became evident that group von Neumann algebras $\mathcal{L}(G)$ typically do not retain the algebraic properties of the underlying groups $G$. In fact, for extensive categories of groups, it was observed a complete lack of memory. For instance, a classical result implies that every infinite abelian group $G$ gives rise to the same group von Neumann algebra $\mathcal{L}(G)\cong L^{\infty}([0,1])$. Here, the isomorphism is implemented, for instance, by the Fourier transform.

\vspace{2mm}

Moreover, this remarkable lack of retention phenomenon persists even within the significantly broader class of so-called amenable groups, \cite{Neumann1929, Dy}. In this direction, a celebrated result of A. Connes \cite{connes76} asserts that for any ICC amenable group $G$ its von Neumann algebra $\mathcal{L}(G)$ is isomorphic to the \emph{hyperfinite} II$_1$ factor $\mathcal{R}$ introduced by Murray and von Neumann in \cite{mvn43}.

\vspace{2mm}

Furthermore, von Neumann algebras of non-ICC amenable groups exhibit a similar behavior. For instance the two aforementioned results already imply that for every $G=A\times B$, where $A$ is infinite abelian and $B$ is ICC amenable, the von Neumann algebra $\cL(G)$ is $\ast$-isomorphic to $L^{\infty}([0,1])\:\overline{\otimes}\:\mathcal{R}$.  Moreover, using \cite[Theorem 5.1]{zellermeier} and Theorem \ref{isomcocycleintegral} one can see this isomorphism still holds for all central extensions $G = A \rtimes_c B$  with infinite center $A$ and ICC amenable central quotient $B$.

\vspace{2mm}

All these results collectively illustrate that, in general, canonical group invariants and constructions such as rank, torsion, generators and relations, semidirect product decompositions, splitting of central extension, among others, do not survive upon transitioning to the realm of von Neumann algebras.

\vspace{2mm}

When $G$ is non-amenable the classification is far more complex. Research in this direction over the last two decades reveals a wide range of striking rigidity phenomena. Specifically, through the powerful deformation/rigidity theory invented by Popa in 2001 (see \cite{popa07}), have been discovered many instances where canonical algebraic properties of groups $G$, as the ones described in the previous paragraphs, can be entirely reconstructed from $\cL(G)$ --- \cite{popabettinumbers01,op03,IPP05,popa06,popaozawa10,popaozawaII,ch08,CP10,peterson09,ioa10,ipv10,berbec,cdss16,kv15,ci18,cdk19,cdad20,CDHK20,cios22,cios22b,cios22c,dr21,,PV21,DV24}, just to enumerate a few. We invite the interested reader to also consult the following surveys \cite{vaes10,io18}. 

\vspace{2mm}

However, this impressive progress has mostly concentrated on group II$_1$ factors; i.e., group von Neumann algebras $\mathcal L(G)$ with $G$ an ICC group. The main reason for this emphasis stems from a classical result of von Neumann, which asserts that every von Neumann algebra can be decomposed as a direct integral of factors over its center \cite{vN49}. Thus, in a certain sense, factors serve as the building blocks of all von Neumann algebras, and understanding their structure is expected to, a priori, provide key insights into the structure of general von Neumann algebras. However, while this philosophy applies to general von Neumann algebras, studying structural aspects for specific classes of examples (e.g., group algebras) reveals a more intricate situation. For instance, while Bekka described in \cite[Theorem B]{bekka21} the integral decomposition of an arbitrary $\mathcal L(G)$ in terms of characters of the FC-center $G^{fc}$ of $G$ and some of its induced representations to $G$, analyzing this data proves to be fairly difficult, and  new developments are required when studying rigidity aspects for these algebras. Unfortunately, to our knowledge, there are only limited results available in the literature in this direction.

\vspace{2mm}

Thus, to expand the rigidity paradigm and spark new technological advancements in the field, we initiate a study into the reconstruction of groups $G$ with infinite center from their von Neumann algebras $\cL(G)$. While the groups we study in this paper are never $W^*$-superrigid in the sense of Popa, our interest lies in exploring the maximum extent to which a reconstruction is possible. In our upcoming work we will introduce property (T) groups with infinite centers that are $W^*$-superrigid.

\vspace{2mm}

\subsection{Statements of the main results}

Our investigation focuses on groups of the form $G= A\times W$, where $A$ is an infinite abelian group and $W$ is a property (T) wreath-like product group as constructed recently in \cite[Theorem 2.6]{cios22} through considerations in geometric group theory. In particular, we will address the following classification problem: Assume that $H$ is an arbitrary group satisfying $\mathcal L(G)\cong \mathcal L(H)$, via a $\ast$-isomorphism 
Our aim is to establish that $H$ possesses analogous characteristics to $G$, such as being a product group of the same nature.


\vspace{2mm}

Our primary objective is to establish that the mystery group $H$ resembles $G$ to the greatest possible extent, implying $H = B \times K$, where $B$ is infinite abelian and $K \cong W$. Achieving this goal involves several steps, including the development of new technology and the interplay between group theory and von Neumann algebras.

\vspace{2mm}

As a first step we show that the FC-center\footnote{$H^{fc} \lhd H$ is the (normal) subgroup of all elements of $H$ that have finite conjugacy class.} $H^{fc}$ of the group $H$ is an infinite subgroup which is close to being the actual center of $H$ (see statement of \ref{theoremA}). The prior theorem establishes that the unknown group $H$ is nearly a central extension. However, to uncover additional structure, such as being a split central extension, requires the development of new ideas. Despite the existence of efficient techniques designed to discern group direct product splitting from $W^*$-equivalence, as outlined in \cite{cdss16, kv15, dhi19}, these methods typically do not apply beyond ICC groups and II$_1$ factors.

\vspace{2mm}

To address this technical challenge, we turn to B. Bekka's work on integral decomposition for group von Neumann algebras, elaborated in \cite{bekka21}. Within this framework, we employ a set of techniques blending group theory and von Neumann algebras, which enables us to establish the splitting of group central extensions from $W^*$-equivalence. This yields the following affirmative resolution of our classification problem, marking the first principal outcome of this paper.

\begin{namedthm*}{Theorem A}\label{theoremB} Let $C$ be a non-trivial free abelian group, let $D$ be a non-trivial ICC subgroup of a hyperbolic group such that the action $D\curvearrowright I$ has amenable stabilizers. Let $W\in\mathcal{WR}(C,D\curvearrowright I)$ be any property (T) group with trivial abelianization. Let $A$ be any infinite abelian group and denote by $G= A\times W$.

\vspace{2mm}

Let $H$ be any countable discrete group and let $\Theta\,:\, \mathcal L(G) \ra \cL(H)$ be any $\ast$-isomorphism.

\vspace{2mm}

Then $H=Z(H)\times H/Z(H)$ where $H/Z(H)\cong W$. Moreover, there is a countable family of projections $\mathcal{P}\subset \cL(Z(H))$ with $\sum_{p\in\mathcal{P}}p=1$, group isomorphisms $\delta_p:W\to H/Z(H) $, for all $p\in\mathcal{P}$, and a unitary $w\in\mathscr{U}(\cL(H))$ such that

\vspace{1mm}

\begin{enumerate}
\item $\Theta(\mathcal L(A))=\mathcal L(Z(H))$, and
   \item  $\Theta(x\otimes u_g)=w\left(\sum_{p\in\mathcal{P}} (\Theta(x) p )\otimes v_{\delta_p(g)}\right)w^*, \quad \text{for all } x\in \cL(A) \text{ and } g\in W$.
\end{enumerate}

\vspace{1mm}

If we assume, in addition, that $\text{Out}(W)=1$, then we can take $\mathcal{P}=\{1\}$.

\vspace{1mm}

Here $(u_g)_{g\in W} \subset \mathcal L(W)$, $(v_h)_{h\in H/Z(H)} \subset \mathcal L(H/Z(H))$ are the canonical unitaries and $\tau_{\mathcal L(H)}$ is the canonical trace on $\mathcal L(H)$.
\end{namedthm*}

The proof of this result is based on several key steps, each of independent interest. First, through a combination of methods in \cite{po05,ioa10,ci10} and integral decomposition techniques which exploits a form of (weak) compactness for the conjugation action of $H$ on $\cL(H^{fc})$, we show the subalgebra $\mathcal L(H^{fc})$ intertwines, in the sense of Popa, into the center of $\mathcal L(H)$. This information is then used to show that the mystery group $H$ is somewhat close to a central extension itself. This portion of the proof can be summarized in the following statement.

\begin{namedthm*}{Theorem B}\label{theoremA} Let $W\in \mathcal{WR}(C,D\curvearrowright I)$ be a property (T) groups where $C$ is a non-trivial abelian, $D$ is a non-trivial ICC subgroup of a hyperbolic group, and $D\curvearrowright I$ has amenable stabilizers. Let $A$ be an infinite abelian group and put $G=W\times A$. Assume that $H$ is an arbitrary group such that $\cL(G)\cong \cL(H)$. 

\vspace{2mm}

If $H^{fc}$ denotes the FC-center of $H$ and $H^{hfc}$ the hyper-FC center\footnote{$H^{hfc}$ is the smallest normal subgroup of $H$ such that $H/H^{hfc}$ is an ICC group (see also Section \ref{section4})}, then $H^{fc}\leqslant H^{hfc}$ is finite index and the commutator $F:=[H^{fc},H^{fc}]$ is finite. Moreover, if we consider the central projection $p= |F|^{-1}\sum_{g\in F}  u_g \in \mathscr Z (\cL(H))$ then one can find a countable family of orthogonal projections $r_i \in \cL(H^{fc})p$ such that $\cL(H^{fc})p = \oplus_n \mathscr Z (\cL(H))r_n$.      


\vspace{2mm}

Moreover, the quotient $H/H^{hfc}$ has property (T) as well.
\end{namedthm*}

\vspace{1mm}



If one considers wreath-like product groups $G$ whose natural $2$-cocycles satisfy a uniform boundedness condition as in \cite{amcos23} (see Definition \ref{bounded2cocycle} in the sequel), in the prior theorem we can actually remove the property (T) assumption on $W$ (see Sections \ref{section4} and \ref{section5} and Theorem \ref{theoremone}). While this more general case is not necessary to derive the main results we believe it is of independent interest as its proof is based on very different arguments.  

\vspace{2mm}

Classes of examples of property (T) wreath-like product groups $W$ satisfying the conditions listed in \ref{theoremB} have been constructed in \cite{cios22, cios22b} using powerful methods in geometric group theory. We invite the interested reader to consult these results beforehand.

\vspace{2mm}

Proving rigidity results for von Neumann algebras of groups with infinite center goes hand-in-hand with establishing rigidity results for twisted group factors, the latter being part of the so-called \emph{fiber analysis}. In fact, this is the second part of the proof of \ref{theoremB}. Specifically, we prove several new W$^*$-rigidity results for embeddings and compressions of twisted group factors of property (T) wreath-like product groups (Theorems \ref{untwistcocycle} and \ref{trivialamplificationembedding}) which generalize some recent findings, \cite[Theorem 7.3]{cios22b}, \cite[Theorem 10.1]{ciosv1}. 

\vspace{2mm}

In this direction, very recently, Donvil and Vaes were able to establish, for the first time in the literature, several beautiful W$^*$-superrigidity results for virtual isomorphisms of twisted II$_1$ factors associated with various classes of left-right wreath product groups \cite{DV24}. Using their height criterion for W$^*$-superrigidity of twisted group factors \cite[Theorem 4.1]{DV24} in combination with some technical outgrowths of the prior methods \cite{cios22} we also obtained a somewhat similar W$^*$-superrigidity result for property (T) wreath-like product groups as in \ref{theoremB}; see the statement of \ref{notrivialcompressions}. We remark that, while this is sufficient to derive \ref{theoremB}, it is not as general as \cite[Theorem A]{DV24} which encompasses twisted group factors on both the target and source along with arbitrary amplifications and can even deal with virtual isomorphisms. 

\begin{namedthm*}{Theorem C}\label{notrivialcompressions} Let $G\in\mathcal{WR}(A,B\curvearrowright I )$ be a property (T) group where $A$ is a non-trivial abelian group, $B$ is an ICC subgroup of a hyperbolic group such that $B\curvearrowright I$ has amenable stabilizers. Let $H$ be any ICC group and let $c:H\times H \to \mathbb T$ be any $2$-cocycle. Let $0<t\leq 1$ be a scalar for which there exists  a $\ast$-isomorphism $\Phi:\cL(G)^t\to \cL_{c}(H)$. 

\vspace{2mm}

Then $t=1$, $c$ is a coboundary and $G\cong H$. Moreover, there exist a group isomorphism $\rho:G \to H$, a map $\xi: G \rightarrow\mathbb T$ with $c(\rho(g),\rho(h)) = \xi_g \xi_h \overline{\xi_{gh}}$, a multiplicative character $\eta: G \rightarrow \mathbb T$ and a unitary $w\in \mathcal L_c(H)$ such that $\Phi(u_g) = \eta_{g} \xi_{g} w v_{\rho(g)}w^*$.
\end{namedthm*}

\vspace{1mm}

We note in passing that if in the statement of \ref{theoremB} one assumes in addition that the outer automorphism group of $W$ is torsion free, then the result can be proved without appealing at all to \ref{notrivialcompressions} (and implicitly the results in \cite{DV24}) but instead just using the W$^*$-strong rigidity results Theorems \ref{untwistcocycle} and \ref{trivialamplificationembedding} along with group theoretic techniques à la J-P. Serre; see Remark \ref{indepenedentproof}.

\vspace{2mm}

Our results also provide brand-new insight into the classification of reduced group $C^*$-algebras, denoted $C^*_r(G)$, for countable groups $G$; a problem that has witnessed notable advancements in recent years.

\vspace{2mm}

A group $G$ is called \emph{abstractly $C^*$-superrigid} if whenever $H$ is an arbitrary group satisfying $C^*_r(G)\cong C^*_r(H)$, via a $\ast$-isomorphism that preserves the canonical traces, then $G\cong H$. A classical result on the classification of homogeneous $C^*$-algebras \cite{scheinberg} entails that every abelian, torsion free group is abstractly $C^*$-superrigid. By leveraging this observation together with other $C^*$-algebraic techniques, it has been demonstrated that the same property holds for various other groups that arise from abelian groups via a canonical construction. This includes certain Bieberbach groups \cite{knubyraumthielwhite17}, some two-step nilpotent groups \cite{er18} and all free nilpotent groups of finite class and rank \cite{omland19}. Notice that all these groups are amenable.

\vspace{2mm}

In the case of non-amenable groups, a significant number of $C^*$-superrigidity results have been obtained over the last years, \cite{ci18,cdad20,cdad21,dr21,cios22c}, by combining deep methods in deformation/rigidity theory of von Neumann algebras \cite{popa07,vaes10,io18} with fundamental prior results on the unique trace property and the projectioneless property of group $C^*$-algebras which emerged from \cite{bkko16} and respectively the extensive prior work on the Baum-Connes conjecture \cite{hk97,hk01,my01,oyono,la12}. However, as in the von Neumann algebraic case, all these results apply only to reduced group $C^*$-algebras of ICC groups. As far as we know, there are no $C^*$-superrigidity results available in the literature for non-amenable groups with non-trivial FC-center.  

\vspace{2mm}

In this paper we make progress on this problem by providing a fairly large family of non-amenable groups $G$ with infinite centers that are abstractly $C^*$-superrigid. Furthermore, for an arbitrary group $H$, we can fully describe \emph{all possible} $\ast$-isomorphisms $\Theta: C^*_r(G) \ra C^*_r(H)$ which preserve the canonical traces. The precise statement is the following.

\vspace{1mm}

\begin{namedthm*}{Corollary D}\label{C*superrigidity} Suppose $C$ is any non-trivial free abelian group and $D$ is any non-trivial ICC subgroup of a hyperbolic group such that the action $D\curvearrowright I$ has amenable stabilizers. Let $W \in \mathcal{WR}(C,D \curvearrowright I)$ be a property (T) group with trivial abelianization. Let $A$ be a torsion free abelian group. Denote by $G= A\times W$. Let $H$ be any countable group and let $\Theta\,:\, C_r^*(G) \ra C_r^*(H)$ be any $\ast$-isomorphism.

\vspace{2mm}

Then $H\cong G$. Moreover, there exist a countable family of projections $\mathcal{P}\subset \cL(Z(H))$ with $\sum_{p\in\mathcal{P}}p=1$, group isomorphisms $\delta_p:W\to H/Z(H)$ for all $p\in\mathcal{P}$, and a unitary $w\in\mathscr{U}(\cL(H))$ such that

\begin{enumerate}
    \item[(a)] $\Theta(C^*_r(A))=C^*_r(Z(H))$, and
    \item[(b)] $\Theta(u_g)=w\left(\sum_{p\in\mathcal{P}}p\otimes v_{\delta_p(g)}\right)w^*$, for all $g\in W$. 
\end{enumerate}

\vspace{1mm}

If we assume, in addition, that $\text{Out}(W)=1$, then we can take $\mathcal{P}=\{1\}$.

\vspace{1mm}

Here $(u_g)_{g\in W} \subset C^*_r(W)$, $(v_h)_{h\in H/Z(H)} \subset C^*_r(H/Z(H))$ are the canonical unitaries. 
\end{namedthm*}

\vspace{2mm}



This result should be compared with \ref{theoremB}, where all information on $G$ could be recovered from $\cL(G)$ with only the exception of its center. Therefore, altogether, \ref{theoremB} and \ref{C*superrigidity} enable us to identify, for the first time, examples of non-amenable groups which, in a certain sense, can be reconstructed from their C$^*$-algebra, but not from their von Neumann algebra. In the amenable realm, such examples have been known for a long time, e.g.\ any abelian free group.

\subsection{Organization of the article} Aside from the introduction, this article comprises eight additional sections. In Section 2 we establish several generic structural results for von Neumann algebras of group extensions using the language of cocycle action crossed products. Section 3 revisits properties of quasi-normalizers and Popa's intertwining techniques. In Section 4, we investigate additional properties of the FC-center $H^{fc}$ of a group $H$. In Section 5, we develop additional deformation/rigidity techniques pertaining to wreath-like product groups, following a similar approach as outlined in \cite{amcos23}, and we use them to provide an independent proof of \ref{theoremA} in the case of wreath-like product groups whose natural 2-cocycles satisfy a uniform boundedness condition. In Section 6, we study new aspects of integral decomposition of group von Neumann algebras and introduce a set of preliminary results that will be used later to recover splitting properties for group centers under von Neumann equivalence. This complex of techniques can be viewed as a von Neumann algebraic counterpart of a classical result of J. P. Serre. In Section 7 we present the proof of \ref{theoremA}. Section 8 establishes several rigidity results for twisted group factors of wreath-like product groups that will be used for the \emph{fiber analysis} in the proof of our main results. These results, also generalize some of the prior work \cite{cios22, cios22b, ciosv1} and are of independent interest. Finally, in Section 9 we present the proofs for our main results \ref{theoremB} and \ref{C*superrigidity}.

\subsection{Acknowledgments} We want to thank Professors Adrian Ioana and Stefaan Vaes for their helpful comments and suggestions. In particular, we are very grateful to Stefaan Vaes for pointing out an error in the first version of this paper. 

\vspace{2mm}

The first author was supported in part by the NSF Grant DMS-2154637. The second author received support from the Erwin and Peggy Kleinfeld Graduate Fellowship, the Graduate College Post-Comprehensive Research Fellowship, and the Graduate College Summer Fellowship 2023. The third author was supported in part by the NSF Grants DMS-1854074 and DMS-2153805. 

\vspace{3mm}

\section{Cocycle actions and cocycle crossed products}\label{section2}

Following \cite[pages 104-105]{br}, we recall the notion of cocycle action and the corresponding cocycle semidirect product of groups. This provides an equivalent description of group extensions. 

\begin{defn}\label{cocycledefinition} A \emph{cocycle action} of a group $B$ on a group $A$ consists of two maps $\alpha:B\to\text{Aut}(A)$ and $c:B\times B\to A$ which satisfy the following: 

\begin{enumerate}
    \item[(1)] $\alpha_g\alpha_h=\text{Ad}(c(g,h))\alpha_{gh}$ for every $g,h\in B$,
    \item[(2)] $c(g,h)c(gh,k)=\alpha_g(c(h,k))c(g,hk)$ for every $g,h,k\in B$, and
    \item[(3)] $c(g,1)=c(1,g)=1$ for every $g\in B$. 
\end{enumerate}

\vspace{1mm}

For brevity, henceforward, a cocycle action will be denoted by $B \curvearrowright^{\alpha,c} A$. The map $c$ is called \emph{a $2$-cocycle associated with $\alpha$}. Moreover, if $A \curvearrowright^{\alpha,c} B$ is a cocycle action, the set $A\times B$ endowed with the unit $1=(1_A,1_B)$ and the multiplication operation $(x,g)\cdot (y,h)=(x\alpha_g(y)c(g,h),gh)$ for all $(x,g),(y,h)\in A\times B$ is a group, denoted by $A\rtimes_{\alpha,c}B$, and called the \emph{cocycle semidirect product} group. We note in passing that any group extension of the form $A$-by-$B$ arises in this way, \cite[Chapter IV.6]{br}. 

\vspace{2mm}

Two cocycle actions $(\alpha, c)$ and $(\alpha',c')$ of $B$ on $A$ are said to be \textit{cohomologous} if there exists a map $\xi:B\to A$ such that $\xi_1 = 1$, $\alpha_g = \text{Ad}(\xi_g) \alpha'_g$, and $c(g,h)\xi_{gh}=\alpha_{g}(\xi_{h})\xi_{g} c'(g,h)$, for all $g,h\in B$. The map $c$ is said to be a \textit{$2$-coboundary for the action $\alpha$} if it is cohomologous to the trivial cocycle. The $2$-cocycle $c$ typically ``measures'' how far away $A\rtimes_{\alpha,c}B$ is from a semidirect product. For example, when $c=1$ we have that $A\rtimes_{\alpha, c}B=A \rtimes_\alpha B$. More generally, when $c$ is a $2$-coboundary, i.e., $c(g,h)= \alpha_{g}(\xi_{h}) \xi_{g} \xi_{gh}^{-1}$, for a map $\xi: B \to A$, the map $\beta_g:=\text{Ad} (\xi_g^{-1})\circ \alpha_g$ defines a genuine action of $B$ on $A$ which gives us $A\rtimes_{\alpha, c}B\cong A \rtimes_{\beta} B$.

\vspace{2mm}

When $A$ is abelian and $\alpha$ is trivial, $A\rtimes_{\alpha,c} B$ amounts to the central extension $A \rtimes_c B$ with the $2$-cocycle $c$.
\end{defn}

Next we record some elementary facts regarding subgroups of cocycle semidirect products that will be needed in the sequel.

\begin{prop}\label{subgroupscocyclecrossedproducts} Let $B \curvearrowright^{\alpha,c} A$ be a cocycle action.

\begin{enumerate}
    \item[(1)] Let $A_0\leqslant A$ be a subgroup satisfying $\alpha_g(A_0)=A_0$ for all $g\in B$, and let $B_0\leqslant B$ be a subgroup satisfying $c(B_0\times B_0)\subseteq A_0$. Then $A_0\rtimes_{\alpha,c}B_0$ is a subgroup of $A\rtimes_{\alpha,c}B$.

    \item[(2)] Assume $\alpha$ is trivial. If $Q$ is a subgroup of $A\rtimes_{c}B$ then one can find a subgroup $B_0\leqslant B$ such that $Q\cong (Q\cap A) \rtimes_{\alpha',c'}B_0$ where $(\alpha',c')$ is a $2$-cocycle action cohomologous to $c$ on $B_0$. In particular, if $B_0=B$ then $A\rtimes_{c}B\cong A\rtimes_{\alpha',c'}B$.
    \item[(3)] Assume $A\rtimes_c B$ is a central extension with property (T). Then, the group generated by $\langle\{c(g,h):g,h\in B\}, B\rangle$ is normal and has finite index in $A\rtimes_cB$. 
\end{enumerate}
\end{prop}

\begin{proof} We prove the second and third statements, as the first one is immediate. 

\vspace{2mm}

Let $B_0$ be the set of all $g\in B$ for which there exists $a \in A$ so that $ag\in Q$. One can see that $B_0\leqslant B$ is a subgroup. Next, for every $g\in B_0$, pick $a_g\in A$ with $a_gg\in Q$ and $a_1 =1$. Note that for any $a\in A$ satisfying $a g\in Q$ we have $a a_g^{-1}=a g g^{-1}a_g^{-1}= ag (a_g g)^{-1}\in Q\cap A$ which implies that $a \in (Q\cap A)a_g $. Altogether these imply that $Q=\{a a_gg: a\in Q\cap A, g\in B_0\}$. Moreover, for $a_{g_1}g_1,a_{g_2}g_2\in Q$ we have

\begin{equation*}
    Q\ni a_{g_1}g_1a_{g_2}g_2=a_{g_1}a_{g_2}c(g_1,g_2)g_1g_2=a_{g_1}a_{g_2}c(g_1,g_2)a_{g_1g_2}^{-1}a_{g_1g_2}g_1g_2.
\end{equation*}

\vspace{2mm}

Let $c'(g_1,g_2):=a_{g_1}a_{g_2}c(g_1,g_2)a_{g_1g_2}^{-1}\in Q\cap A$ for $g_1,g_2\in B_0$. Then, $c'$ is a $2$-cocycle for $\alpha': B_0 \ni b \mapsto \text{Ad}(a_b) \in \text{Aut}(A)$ which is cohomologous to $c$ on $B_0$, and $Q\cong (Q\cap A)\rtimes_{\alpha',c'}B_0$. This finishes the proof of (2).

\vspace{2mm}

To prove (3), denote by $Q:=\langle\{c(g,h):g,h\in B\}, B\rangle$. By (2), since $\alpha$ is trivial, there exists a cocycle $c'$ cohomologous to $c$ on $B$ with $Q\cong (A\cap Q)\rtimes_{c'}B$. Since $(A\rtimes_{c'}B)/Q\cong A/A\cap Q$ is abelian and has property (T), $Q\leqslant A\rtimes_cB$ has finite index.
\end{proof}

These notions admit natural generalizations in the context of von Neumann algebras.

\begin{defn}\label{cocycledefvN} A \emph{cocycle action} of a group $B$ on a tracial von Neumann algebra $(\cM,\tau)$ is a pair $(\alpha,c)$ consisting of two maps $\alpha:B\to {\text{Aut}}(\cM)$ and $c:B\times B\to\mathscr{U}(\cM)$ which satisfy the following:

\begin{enumerate}
    \item[(1)] $\alpha_g\alpha_h={\text {Ad}}(c(g,h))\alpha_{gh}$ for every $g,h\in B$,
    \item[(2)] $c(g,h)c(gh,k)=\alpha_g(c(h,k))c(g,hk)$ for every $g,h,k\in B$, and
    \item[(3)] $c(g,e)=c(e,g)=1$ for every $g\in B$.
\end{enumerate}

\vspace{1mm}

As in the case of groups, a cocycle action as in Definition \ref{cocycledefvN} will be denoted by $B \curvearrowright^{\alpha,c} (\mathcal M,\tau)$. The map $c$ is called a $2$-cocycle for the map $\alpha$. 

\vspace{2mm}

We say $2$-cocycle actions $(\alpha_i,c_i)$ of $B$ on $\cM_i$, $i=1,2$, are \textit{cocycle conjugate} if there exists an isomorphism $\Theta:\cM_1\to\cM_2$ and a map $\xi:B\to\mathscr{U}(\cM_2)$ such that $\Theta\alpha_1(g)\Theta^{-1}=\text{Ad}(\xi_g)\alpha_2(g)$, and $\Theta(c_1(g,h))=\xi_g\alpha_2(g)(\xi_h)c_2(g,h)\xi_{gh}^*$, for all $g,h\in B$. We say the $2$-cocycle $c$ associated to $\alpha$ \textit{vanishes} (or it is a \textit{coboundary}) if it is cocycle conjugate to some action $(\alpha',1)$.
\end{defn}

\begin{defn}\label{cocyclecross} Let  $B \curvearrowright^{\alpha,c} (\mathcal M,\tau)$ be a cocycle action.  The \emph{cocycle crossed product} von Neumann algebra, denoted by $\cM\rtimes_{\alpha,c}B$, is a tracial von Neumann algebra which is generated by a copy of $\cM$ and unitary elements $\{u_g\}_{g\in B}$ such that $u_gxu_g^*=\alpha_g(x)$, $u_gu_h=c(g,h)u_{gh}$, and $\tau(xu_g)=\tau(x)\delta_{g,e}$ for every $g,h\in B$ and $x\in \cM$. Notice that when $c$ is trivial, the classical crossed product von Neumann algebra is recovered.
\end{defn}

Basic instances of cocycle crossed products von Neumann algebra are twisted group von Neumann algebras. Let $G$ be a group, let $\alpha:G\to\C$ be the trivial action, and $c:G\times G\to \mathbb{T}$ be a 2-cocycle; that is, $c(g,h)c(gh,k)=c(h,k)c(g,hk)$, for all $g,h,k\in G$. Then, $G\curvearrowright^{\alpha,c}\C$ is a cocycle action. Moreover, the corresponding von Neumann algebra $\C\rtimes_{\alpha,c}G$ is the so-called \emph{twisted group von Neumann algebra} $\cL_c(G)$. Thus, $\cL_c(G)$ is generated by the elements $u_g$, with $g\in G$, satisfying the relation $u_gu_h=c(g,h)u_{gh}$ for all $g,h\in G$.

\vspace{2mm}

We establish several outcomes related to von Neumann algebras arising from cocycle crossed products. Although these results are likely well known to experts, we were unable to locate their explicit mention within the existing literature. 

\begin{lem}\label{finitedimfromcompatomic} Let $G\curvearrowright^{\alpha,c}(\mathcal{D},\tau)$ be a cocycle action where $\mathcal{D}$ is abelian and completely atomic. If $\mathcal{D}\rtimes_{\alpha,c}G$ is a factor, then $\mathcal{D}$ is finite dimensional.
\end{lem}

\begin{proof} We prove the contrapositive. Assume $\mathcal{D}$ is infinite dimensional and let $\{p_n\}_{n=1}^{\infty}$ be the atoms of $\mathcal{D}$. Since $\mathcal{D}$ is abelian, the measure on each atom eventually decays and there are only finitely many atoms with the same measure. Let $g\in G$ and let $p_n$ be an atom. Observe that $\alpha_g(p_n)$ is an atom in $\mathcal{D}$ that has the same measure as $p_n$. Thus, by the considerations above, the orbit $\mathscr{O}(p_n)=\{p_{n_1},...,p_{n_k}\}$ must be finite. Let $z_{\mathscr{O}(p_n)}=\sum_{p\in\mathscr{O}(p_n)}p$ and notice that for $g\in G$, $gz_{\mathscr{O}(p_n)}g^{-1}=z_{\mathscr{O}(p_n)}$. Hence, $z_{\mathscr{O}(p_n)}$ is a non-trivial element in $\mathscr{Z}(\mathcal{D}\rtimes_{\alpha,c}G)$.
\end{proof}

\begin{prop}\label{amplificationoffinitebyiccvn} Let $G \curvearrowright^{\alpha,c} (\mathcal A,\tau)$ be a cocycle action where $\mathcal A$ is finite dimensional.  
Assume the corresponding cocycle crossed product $\mathcal{A}\rtimes_{\alpha,c}G=\cM$ is a factor. Then, one can find $l,n\in\N$ such that $\mathcal{A}=\mathbb{M}_l(\C)\otimes \mathbb{D}_n$, a transitive action $G\curvearrowright^{\beta}\mathbb{D}_n$ and a 2-cocycle $w:G\times G\to\mathscr{U}(\mathbb{D}_n)$ for $\beta$ so that

\begin{equation*}
    \mathcal{A}\rtimes_{\alpha,c}G\cong \mathbb{M}_l(\C)\otimes(\mathbb{D}_n\rtimes_{\beta,w}G).
\end{equation*}

\vspace{2mm}

Moreover, we can find a subgroup $G_0\leqslant G$  with $[G:G_0]=n$ and a 2-cocycle $\eta:G_0\times G_0\to\mathbb{T}$ satisfying

\begin{equation*}
    \mathcal{A}\rtimes_{\alpha,c}G\cong \mathbb{M}_{nl}(\C)\otimes\cL_{\eta}(G_0)\cong \cL_{\eta}(G_0)^{nl}.
\end{equation*}
\end{prop}

\begin{proof} Since $\mathcal{A}$ is finite dimensional there are minimal projections $z_1,...,z_n\in\mathscr{Z}(\mathcal{A})$ so that $\mathcal{A}=\oplus_{i=1}^n\mathbb{M}_{k_i}(\C)z_i$, where $k_i\in\N$. Fix $g\in G$. As $\alpha_g(\mathscr{Z}(\mathcal{A}))=\mathscr{Z}(\mathcal{A})$, $\alpha_g$ permutes the set  $\{z_1,...,z_n\}$. Furthermore, since $\cM$ is a factor, the action $\alpha$ on $\{z_1,...,z_n\}$ must be transitive. Thus, $\alpha_g(z_i)=z_{g\cdot i}$ for a transitive action $G\curvearrowright\{1,...,n\}$. This implies $\alpha_g(\mathbb{M}_{k_i}(\C)z_i)\subseteq \mathbb{M}_{k_{g\cdot i}}(\C)z_{g\cdot i}$. In particular, $k_i\leq k_{g\cdot i}$ for all $1\leq i\leq n$, and since the action is transitive, one can find an integer $l$ such that $k_i=l$, for all $1\leq i\leq n$. Hence, $\mathcal{A}=\oplus_{i=1}^n\mathbb{M}_l(\C)z_i$, and $\alpha_g(\mathbb{M}_l(\C)z_i)=\mathbb{M}_l(\C)z_{g\cdot i}$. Since all automorphisms of $\mathbb{M}_l(\C)$ are inner, we can find unitaries $V_g^i\in\mathbb{M}_l(\C)$ such that

\begin{equation*}
    \alpha_g(x z_i)=V_g^ix z_{g\cdot i}(V_g^i)^*, \quad\text{ for all } x\in\mathbb{M}_l(\C).
\end{equation*}

\vspace{2mm}

Hence, there is $V_g\in\mathscr{U}(\mathcal A)$ such that $\alpha_g(x)=V_g\beta_g(x)V_g^*$, for all $x\in \mathcal{A}$, where $\beta_g:\mathcal{A}\to \mathcal{A}$ is given by $\beta_g(\oplus \lambda_iz_i)= \oplus \lambda_iz_{g\cdot i}$, for $\lambda_i\in \mathbb C$. Since $\alpha$ is a cocycle action, $\alpha_g\circ\alpha_h=\text{Ad}(c(g,h))\alpha_{gh}$ for all $g,h\in G$. On the other hand, $\alpha_g\circ\alpha_h=\text{Ad}(V_g\beta_g(V_h))\circ\beta_{gh}$ and $\text{Ad}(c(g,h))\alpha_{gh}=\text{Ad}(c(g,h)V_{gh})\circ\beta_{gh}$ for all $g,h\in G$. Altogether, these imply that $\text{Ad}(V_{gh}^*c(g,h)^*V_g\beta_g(V_h))=\text{id}$.

\vspace{2mm}

Since $V_{gh}^*c(g,h)^*V_g\beta_g(V_h)\in \mathcal{A}$ we conclude that $w(g,h):=V_g^*\beta_g(V_h^*)c(g,h)V_{gh}\in\mathscr{Z}(\mathcal{A})$ is also a 2-cocycle. Therefore, we have shown that $\mathcal{A}=\mathbb{M}_l(\C)\otimes(\oplus\C z_i)=\mathbb{M}_l(\C)\otimes \mathscr{Z}(\mathcal{A})=\mathbb{M}_l(\C)\otimes \mathbb{D}_n$, and $\mathcal{A}\rtimes_{\alpha,c}G=\mathbb{M}_l(\C)\otimes (\mathbb{D}_n\rtimes_{\beta,w}G)$. 

\vspace{2mm}

For the last part, fix $z_i\in \mathbb{D}_n$ a minimal projection. Let $G_0\leqslant G$ be the subgroup of all $g\in G$ such that $\beta_g(z_i)=z_i$. Note that $[G:G_0]=n$. Thus, for all $z\in \mathbb{D}_n$ and $g\notin G_0$, $z_izu_gz_i=\lambda_iz_iu_gz_i=\lambda_iz_i\beta_g(z_i)u_g=0$. For $g\in G_0$ we have $z_izu_gz_i=\lambda_i z_iu_gz_i=\lambda_iz_iu_g$. Moreover, for every $g,h\in G_0$ we get 

\begin{equation*}
    z_iu_gz_iu_hz_i=z_iw(g,h) z_i u_{gh}=:\eta(g,h) z_i u_{gh},
\end{equation*}

\vspace{2mm}

where $\eta (g,h) = z_iw(g,h) z_i$. Hence, $z_i(\mathbb{D}_n\rtimes_{\beta,w}G)z_i
\cong \cL_{\eta}(G_0)$. Since the action $\beta$ is transitive, $z_i$ is equivalent to $z_j$ in $\mathbb{D}_n\rtimes_{\beta,w}G$ for any $j \in \{1,..., n\}$. In particular, we have

\begin{align*}
    \mathcal{A}\rtimes_{\alpha,c}G&=\mathbb{M}_l(\C)\otimes (\mathbb{D}_n\rtimes_{\beta,w}G)\cong\mathbb{M}_l(\C)\otimes\mathbb{M}_n(\C)\otimes (z_i(\mathbb{D}_n\rtimes_{\beta,w}G)z_i)\\
    &\cong \mathbb{M}_{nl}(\C)\otimes\cL_{\eta}(G_0)\cong\cL_{\eta}(G_0)^{ln}.
\end{align*}\end{proof}

\begin{prop}\label{decompositionfinitebyabelian} Let $G=A\rtimes_{\alpha,c}B$ be any cocycle semidirect product group, where $A$ is finite and $B$ is abelian. Then, we can find $k\in \mathbb N$ and $n_i, m_i\in\N$ for all $1\leq i\leq k$, such that

\begin{equation*}
    \cL(G)\cong\bigoplus_{i=1}^k\left(\mathbb{M}_{n_i}(\C)\otimes(\mathbb{D}_{m_i}\rtimes_{\beta_i,w_i}B)\right).
\end{equation*}

\vspace{2mm}

Here, for each $1\leq i\leq k$, $\mathbb{D}_{m_i}$ is a finite-dimensional abelian algebra, and $B\curvearrowright^{\beta_i, w_i}\mathbb{D}_{m_i}$ is a cocycle  action, with $\beta_i$ a transitive action. Moreover, $\mathscr{Z}(\cL(G))=\oplus_i\:\cL_{\eta_i}(C_i)$ for $C_i\leqslant B$ a finite index subgroup and $\eta_i$ a symmetric $2$-cocycle cohomologous to $w_i$ on ${C_i\times C_i}$.
\end{prop}

\begin{proof} Observe that $\cL(G)=\cL(A)\rtimes_{\alpha,c}B$ where $\cL(A)$ is a finite dimensional von Neumann algebra, and $B \curvearrowright ^{\alpha,c} \cL(A)$ is the canonical cocycle action implemented from the group case. Moreover, $\mathscr{Z}(\cL(A))$ is also a finite dimensional von Neumann subalgebra, invariant under the action of $\alpha$. Let $\{z_j\}_{j=1}^k$ be the minimal projections of $\mathscr{Z}(\cL(A))$. Then there exists an action $B\curvearrowright\{z_j\}_{j=1}^k$ for which $\alpha_g(z_j)=z_{g\cdot j}$. Let $\{\mathscr{O}_i\}
_i$ be the orbits of this action. Let $z^i=\sum_{z\in\mathscr{O}_i}z$ and denote $\cM_iz^i=\cL(A)z^i$. Note that

\begin{equation}\label{decompL(A)B}
    \cL(A)\rtimes_{\alpha,c}B=\bigoplus_i\cM_iz^i\rtimes_{\alpha_i,c_i}B
\end{equation}

\vspace{2mm}

where $B\curvearrowright^{\alpha_i,c_i}\cM_iz^i$ is the cocycle action given by the restriction of $\alpha$ to $\cM_iz^i$ and $c_i(g,h):=c(g,h)z^i$. Now, decompose $\cM_iz^i=\oplus_{z_j\in\mathscr{O}_i}\mathbb{M}_{n_j^i}(\C)z_j$. Applying the proof of the previous proposition on the fibers $\cM_iz^i\rtimes_{\alpha_i,c_i}B$ shows that $n_j^i=n_i$ for all $j$,  and

\begin{equation*}
    \cM_iz^i\rtimes_{\alpha_i,c_i}B=\mathbb{M}_{n_i}(\C)\otimes(\mathbb{D}_{m_i}\rtimes_{\beta_i,w_i}B)
\end{equation*}

\vspace{2mm}

where $\mathbb{D}_{m_i}$ is abelian, and $B\curvearrowright^{\beta_i,w_i} \mathbb D_{m_i}$ is a cocycle action with $\beta_i$ transitive. 

\vspace{2mm}

Let $B_i:=\{g\in B:\beta_i(g)(x)=x, \:\text{for all }x\in \mathbb{D}_{m_i}\}$, and note this subgroup coincides with the stabilizer of all atoms in $\mathbb{D}_{m_i}$. Indeed, if $g\in B$ is such that $\beta_i(g)(z_j)=z_j$ for some $1\leq j\leq m_i$, then for any other $1\leq l\leq m_i$, since the action $\beta_i$ is transitive, there exists $h\in B$ with $h\cdot l=j$. In that case, $g\cdot l=(gh^{-1})\cdot(h\cdot l)=h^{-1}\cdot j=l$. 

\vspace{2mm}

Next we describe the center of $\cL(G)$. This amounts to describing the center of $\mathbb{D}_{m_i}\rtimes_{\beta_i,w_i}B$. Let $\sum_{g\in B}x_gu_g\in \mathscr{Z}(\mathbb{D}_{m_i}\rtimes_{\beta_i,w_i}B)$. Comparing Fourier coefficients, we have

\begin{equation}\label{fouriercoeff}
    x_g\beta_i(g)(x)=xx_g=x_gx, \quad\text{for all } g\in B\text{ and for all } x\in \mathbb{D}_{m_i}.
\end{equation}

\vspace{2mm}

Next we claim $x_g=0$ for all $g\notin B_i$. First, for $g\notin B_i$ there exists $j$ for which $\beta_i(g)(z_j)z_j=0$. Since the group $B_i$ coincides with the stabilizer of all atoms, $\beta_i(g)(z_j)z_j=0$ for all $j$. Thus, by \eqref{fouriercoeff} we obtain $x_gz_j=x_gz_jz_j=x_g\beta_i(g)(z_j)z_j=0$ for $g\notin B_i$. Since this holds for all atoms of $\mathbb{D}_{m_i}$ we get $x_g=0$.

\vspace{2mm}

Comparing the Fourier coefficients again we have

\begin{equation}\label{fouriercoeff2}
    x_gw_i(g,h)=\beta_i(h)(x_g)w_i(h,g),\quad\text{for all }g,h\in B,
\end{equation}

\vspace{2mm}

as $B$ is abelian. Now, let $x_g=\sum_{j=1}^{m_i}\lambda_g^jz_j$ and $w_i(g,h)=\sum_{j=1}^{m_i}(w_i(g,h))_jz_j$. By \eqref{fouriercoeff2}, 

\begin{equation}\label{new}
    \lambda_g^j(w_i(h,g))_j^{-1}(w_i(g,h))_j=\lambda_g^{h^{-1}\cdot j}, \quad\text{for all } 1\leq j\leq m_i,\text{ and }g,h\in B.
\end{equation}

\vspace{2mm}

Since $(w_i(h,g))_j^{-1},(w_i(g,h))_j$ are scalars of absolute value one, it follows that $|\lambda_g^{h^{-1}\cdot j}|=|\lambda_g^j|$ for all $h$ and $1\leq j\leq m_i$. Thus, for $g\in B_i$, either $x_g=0$ or $x_g$ is a scalar multiple of a unitary in $\mathbb{D}_{m_i}$, say $s_g$. Observe that if $g\in B_i$ is such that $x_g\neq 0$ then, from equation \eqref{fouriercoeff2}, we obtain $w_i(g,h)=w_i(h,g)$ for all $h\in B_i$, meaning that $g\in C_i:=\{g\in B_i: w_i(g,h)=w_i(h,g)$ for all $h\in B_i\}$. Conversely, one can check that if $g \in C_i$, then there exists $x_g \in \mathcal{U}(\mathbb{D}_{m_i})$ satisfying equation \eqref{fouriercoeff2}. Observe that the unitaries $s_g$ satisfy the relation

\begin{equation}\label{x_bu_bincenter}  
    s_gw_i(g,h)=\beta_i(h)(s_g)w_i(h,g),\quad\text{for all }g\in C_i, h\in B.
\end{equation} 

\vspace{2mm}

Therefore, $s_gu_g\in\mathscr{Z}(\mathbb{D}_{m_i}\rtimes_{\beta_i,w_i}B)$ for $g\in C_i$. Let $(y_g)_g$ be any other sequence of unitaries satisfying

\begin{equation}\label{y_bsameequation}
    y_gw_i(g,h)=\beta_i(h)(y_g)w_i(h,g),\quad\text{for all } g\in C_i, h\in B.
\end{equation}

\vspace{2mm}

Then, \eqref{x_bu_bincenter} and \eqref{y_bsameequation} imply

\begin{equation*}
    s_gy_g^*=s_gw_i(g,h)w_i(h,g)^*\beta_i(h)(y_g)^*=\beta_i(h)(s_gy_g^*), \quad\text{for all } h\in B.
\end{equation*}

\vspace{2mm}

The previous equality implies $s_gy_g^*\in \C1$ and so $y_g$ is a scalar multiple of $s_g$. Hence, $\mathscr{Z}(\mathbb{D}_{m_i}\rtimes_{\beta_i,w_i}B)=\{s_gu_g\}_{g\in C_i}''$. Moreover, we can check that for $g_1,g_2\in C_i$ we have that $\mathscr{Z}(\mathbb{D}_{m_i}\rtimes_{\beta_i,w_i}B)\ni s_{g_1}u_{g_1}s_{g_2}u_{g_2}=s_{g_1}s_{g_2}w_i(g_1,g_2)u_{g_1g_2}$, and also $s_{g_1g_2}u_{g_1g_2}\in\mathscr{Z}(\mathbb{D}_{m_i}\rtimes_{\beta_i,w_i}B)$. This implies the existence of $\eta_i(g_1,g_2)\in\C$ such that

\begin{equation*}
    s_{g_1g_2}\eta_i(g_1,g_2)=s_{g_1}s_{g_2}w_i(g_1,g_2) \text{ for all } g_1,g_2\in C_i.
\end{equation*}

\vspace{2mm}

Therefore, $\{s_gu_g\}_{g\in C_i}''\cong \cL_{\eta_i}(C_i)$ and $\mathscr{Z}(\mathbb{D}_{m_i}\rtimes_{\beta_i,w_i}B)=\cL_{\eta_i}(C_i)$, where $\eta_i$ is symmetric and cohomologous to $w_i|_{C_i\times C_i}$.
\end{proof}

We note the following corollary, which will be employed in proving \ref{theoremB}. 

\begin{thm}\label{centerLH} Let $G=F\rtimes_{\alpha,c}(A\rtimes_d B)$ be a cocycle semidirect product group, where $F$ is finite, $A$ is abelian and $B$ is ICC. Then, we can find $k\in\N$, $n_i, m_i\in\N$ for $1\leq i\leq k$, cocycle actions $A\rtimes_d B\curvearrowright^{\alpha_i,w_i}\mathbb{D}_{m_i}$ with $\alpha_i$ transitive such that

\begin{equation*}
    \cL(G)=\bigoplus_{i=1}^k\:\mathbb{M}_{n_i}(\C)\otimes(\mathbb{D}_{m_i}\rtimes_{\alpha_i,w_i}(A\rtimes_d B)).
\end{equation*}

\vspace{2mm}

In addition, we can write $\mathbb{D}_{m_i}=\oplus_{t=1}^{t_i}\mathbb{D}_{s_i}^t$, where $m_i=s_it_i$, in such a way that there exist transitive actions $A_t^i\curvearrowright\mathbb{D}_{s_i}^t$, with $A_t^i$ finite index subgroups of $A$, $x_{t,a}^i\in\mathscr{U}(\mathbb{D}_{s_i}^t)$ and $\eta_t^i:A_t^i\times A_t^i\to\mathbb{T}$ a 2-cocycle such that

\begin{equation*}
    \mathscr{Z}(\cL(F\rtimes_{\alpha,c}A))=\bigoplus_{i=1}^k\bigoplus_{t=1}^{t_i}\:\{x_{t,a}^i:a\in A_t^i\}''\cong\bigoplus_i\bigoplus_t\:\cL_{\eta_t^i}(A_t^i).
\end{equation*}

\vspace{2mm}

Moreover, if we let $A_0:=\cap_{i,t}A_t^i$ and $x_a^i=\sum_{t=1}^{t_i}x_{t,a}^i\in\mathscr{U}(\mathbb{D}_{m_i})$ we have

\begin{equation*}
    \mathscr{Z}(\cL(G))=\left\{\left(\sum_ix_a^i\right)a:a\in A_0\right\}''\subseteq\bigoplus_i\bigoplus_t\cL_{\eta_t^i}(A_t^i).
\end{equation*}
\end{thm}
\vspace{2mm}

\begin{proof} By the first part of Proposition \ref{decompositionfinitebyabelian}, we obtain that

\begin{equation*}
    \cL(G)\cong\bigoplus_{i=1}^k\:(\mathbb{M}_{n_i}(\C)\otimes(\mathbb{D}_{m_i}\rtimes_{\alpha_i,w_i}(A\rtimes_d B))),
\end{equation*}

\vspace{2mm}

for cocycle actions $A\rtimes_d B\curvearrowright^{\alpha_i,w_i}\mathbb{D}_{m_i}$ with $\alpha_i$ transitive.

\vspace{2mm}

Let $\{z_t^i:1\leq t\leq m_i\}$ be the atoms of $\mathbb{D}_{m_i}$ and let $\mathscr{O}_1,...,\mathscr{O}_{t_i}$ be the orbits of the restriction of the action $\alpha_i$ to $A$. Since $A$ is normal in $A\rtimes_d B$, the action $A\rtimes_d B\curvearrowright\mathbb{D}_{m_i}$ will permute the orbit set transitively. Indeed, for $g\in A\rtimes_d B$, if $g\cdot z=z'$ for $z\in\mathscr{O}_i$ and $z'\in\mathscr{O}_j$, then for any other element $r\in\mathscr{O}_i$, there exists $a\in A$ with $a\cdot r=z$ and thus, $g\cdot r=a^{-1}g\cdot z=a^{-1}z'\in\mathscr{O}_j$. In particular, this implies these orbits have the same size, say $|\mathscr{O}_t|=s_i$ for all $1\leq t\leq t_i$, and $s_it_i=m_i$. Similarly, since the copy of $B$ in $A\rtimes_dB$ commutes with $A$, and using a similar argument as earlier, $B$ permutes the orbits transitively as well. 

\vspace{2mm}

Since $B$ is ICC, $\mathscr{Z}(\cL(G))\subseteq\cL(F\rtimes_{\alpha,c}A)$ and using the prior proposition,

\begin{align*}
    \mathscr{Z}(\cL(G))&\subseteq\mathscr{Z}\left(\bigoplus_{i=1}^k\bigoplus_{t=1}^{t_i}\mathbb{D}_{s_i}^t\rtimes_{\alpha_i,w_{i,t}}A\right)=\bigoplus_{i=1}^k\bigoplus_{t=1}^{t_i}\:\{x_{t,a}^ia:a\in A_t^i\}''\cong\bigoplus_{i=1}^k\:\bigoplus_{t=1}^{t_i}\cL_{\eta_t^i}(A_t^i),
\end{align*}

\vspace{2mm}

where $\eta_t^i$ are scalar 2-cocycles of the finite index subgroups $A_t^i$ of $A$ for every $1\leq t\leq t_i$. 

\vspace{2mm}

Denote by $z_t^i=\sum_{z\in\mathscr{O}_t}z$ the identity of $\mathbb{D}_{s_i}^t$ for $1\leq t\leq t_i$. Given the transitive action of $B$ on the orbit set ${\mathscr{O}_1, \ldots, \mathscr{O}_{t_i}}$, there exists a transitive action $B\curvearrowright{1, \ldots, t_i}$ where $\alpha_i(b)(z_t^i)=z_{b\cdot t}^i$.

\vspace{2mm}

Now, take $y\in\mathscr{Z}(\cL(F\rtimes_{\alpha,c}(A\rtimes_d B)))$ and let $y=\sum_i\sum_t\sum_{a\in A_t^i}\lambda_{t,a}^ix_{t,a}^ia$ for $\lambda_{t,a}^i\in\C$.  Fix $b\in B$ and observe that

\begin{align*}
    \sum_{i=1}^k\sum_{t=1}^{t_i}\sum_{a\in A_t^i}\lambda_{t,a}^ix_{t,a}^iw_i(a,b)(ab)&=yb=by=\sum_{i=1}^k\sum_{t=1}^{t_i}\sum_{a\in A_t^i}\lambda_{t,a}^i\alpha_i(b)(z_t^i)\alpha_i(b)(x_{t,a}^i)w_i(b,a)(ab)\\
    &=\sum_{i=1}^k\sum_{t=1}^{t_i}\sum_{a\in A_{b^{-1}\cdot t}^i}\lambda_{b^{-1}\cdot t,a}^iz_t^i\alpha_i(b)(x_{b^{-1}\cdot t,a}^i)w_i(b,a)(ab).
\end{align*}

\vspace{2mm}

Multiplying on both sides by $z_t^i$ and cancelling $b$, we obtain

\begin{align}\label{lambdasum}
    z_t^i\sum_{a\in A_t^i}\lambda_{t,a}^ix_{t,a}^iw_i(a,b)a=\sum_{a\in A_{b^{-1}\cdot t}^i}\lambda_{b^{-1}\cdot t,a}^iz_t^i\alpha_i(b)(x_{b^{-1}\cdot t,a}^i)w_i(b,a)a
\end{align}

\vspace{2mm}

for all $b\in B$, $1\leq t\leq t_i$ and $1\leq i\leq k$. Now, the prior equality can only be true whenever $\lambda_{t,a}^i=0$ for all $a\in (A_t^i\setminus A_{b^{-1}\cdot t}^i)\cup (A_{b^{-1}\cdot t}^i\setminus A_t^i)$ for all $b,t,i$. Since this holds for all elements in $B$ and $B$ acts transitively, $\lambda_{t,a}^i=0$ for all $a\notin\cap_{t,i}A_t^i=:A_0$. Moreover, from \eqref{lambdasum}

\begin{equation}\label{lambda2}
    \lambda_{t,a}^ix_{t,a}^iw_i(a,b)=\lambda_{b^{-1}\cdot t,a}^i\alpha_i(b)(x_{b^{-1}\cdot t,a}^i)w_i(b,a),
\end{equation}

\vspace{2mm}

for all $i,t,b$. Taking norms gives us $|\lambda_{t,a}^i|=|\lambda_{b^{-1}\cdot t,a}^i|$. Since $B$ acts transitively, $|\lambda_{t,a}^i|=:\mu_a^i$. By writing $\lambda_{t,a}^i=\frac{\lambda_{t,a}^i}{\mu_a^i}\mu_a^i$, for $\lambda_{t,a}^i\neq 0$, and by changing $x_{t,a}^i$ to $\frac{\lambda_{t,a}^i}{\mu_a^i}x_{t,a}^i$, equation \eqref{lambda2} becomes

\begin{equation}\label{cancelambda}
    x_{t,a}^iw_i(a,b)=\alpha_i(b)(x_{b^{-1}\cdot t,a}^i)w_i(b,a)\quad\text{for all }i,t,a,b.
\end{equation}

\vspace{2mm}

Let $u\in\mathscr{Z}(\cL(G))$ with $u_{t,a}^i\in\C$ as in the case of $y$ with $\lambda_{t,a}^i$. Observe that by repeating the same process with $u$, we obtain that the constants $u_{t,a}^i$ satisfy equation \eqref{lambda2}. From this and combining it with equation \eqref{cancelambda} we obtain that $u_{t,a}^i=u_{b^{-1}\cdot t,a}^i$ for all $a,b,t$. Thus, by letting $u_a=u_{t,a}^i$ we obtain

\begin{equation*}
    u=\sum_{i=1}^k\sum_{t=1}^{t_i}\sum_{a\in A_t^i}u_{t,a}^ix_{t,a}^ia=\sum_{a\in A_0}u_a\left(\sum_{i,t}x_{t,a}^i\right)a\in\left\{\left(\sum_{i,t}x_{t,a}^i\right)a:a\in A_0\right\}''.
\end{equation*}
\end{proof}

\section{Intertwining of von Neumann algebras and control of their quasi-normalizers}

\subsection{Popa's Intertwining Techniques.} We recall from \cite[Theorem 2.1, Corollary 2.3]{popa06} Popa’s intertwining-by-bimodules theory.

\begin{thm}(\cite{popa06})\label{intertwining} Let $(\mathcal{M},\tau)$ be a separable tracial von Neumann algebra with $\mathcal{N}_1,\mathcal{N}_2\subseteq \cM$ von Neumann subalgebras, where the inclusion $\mathcal{N}_1 \subseteq \cM$ is not necessarily unital. Then, the following are equivalent:

\begin{enumerate}
    \item[(1)] There exists non-zero projections $p\in \mathcal{N}_1$, $q\in \mathcal{N}_2$, a $*$-homomorphism $\Theta:p\mathcal{N}_1p\to q\mathcal{N}_2q$ and a non-zero partial isometry $v\in q\mathcal{M} p$ such that $\Theta(x)v=vx$ for all $x\in p\mathcal{N}_1p$.
    \item[(2)] For every subgroup $\mathcal G\subseteq \mathscr U(\mathcal N_1)$ with $\mathcal G''=\mathcal N_1$  there is no sequence $(u_n)_n\subset\mathcal G$ satisfying $\|\mathbb{E}_{\cN_2}(xu_ny)\|_2\to 0$ for all $x,y\in\cM$.
\end{enumerate}
\end{thm}

If one of these equivalent conditions hold, we say that a corner of $\cN_1$ embeds into $\cN_2$ inside $\cM$. We denote it as $\cN_1\prec_\cM\cN_2$. If $\mathcal{N}_1p'\prec_{\cM}\mathcal{N}_2$ for any nonzero projection $p'\in\mathcal{N}_1'\cap \cM$, we write $\mathcal{N}_1\prec_{\cM}^s\mathcal{N}_2$.

\begin{prop}\label{intlower}Let $\cM$ be a finite von Neumann algebra and let $\cB, \cD \subseteq \cC\subseteq \cM$ be von Neumann subalgebras such that $\mathscr N_{\cM}(\cD)''=\cM$. If $\cB\prec_\cM \cD$ then  $\cB\prec_\cC \cD$. \end{prop}

\begin{proof} Using Theorem \ref{intertwining} and  $\cB\prec_\cM \cD$, we can find $x_1,...,x_n,y_1,...,y_n\in \cM$ and $d>0$ for which

\begin{equation*}
    \sum_{i=1}^n\|\mathbb{E}_{\cD}(x_iby_i)\|_2^2\geq d,\quad\text{for all } b\in\mathscr{U}(\cB).
\end{equation*}

\vspace{2mm}

Since $\cD$ is regular in $\cM$, using basic  $\|\cdot\|_2$-approximations of $x_i$ and $y_j$, we can find $u_1,...,u_m,v_1,...,v_m\in \mathscr{N}_\cM(\cD)$ and $d'>0$ such that

\begin{equation*}
    \sum_{i=1}^m\|\mathbb{E}_{\cD}(u_ibv_i)\|_2^2\geq d',\text{ for all }b\in\mathscr{U}(\cB).
\end{equation*}

\vspace{2mm}

Observe that $\mathbb{E}_{\cD}(u_ibv_i)=\mathbb{E}_{u_i\cD u_i^*}(u_ibv_i)=u_i\mathbb{E}_\cD(bv_iu_i)u_i^*$. Hence, for all $b\in \mathscr{U}(\cB)$, we have

\begin{align*}
    d'&\leq \sum_{i=1}^m\|\mathbb{E}_{\cD}(u_ibv_i)\|_2^2=\sum_{i=1}^n\|u_i\mathbb{E}_{\cD}(bv_iu_i) u_i^*\|_2^2=\sum_{i=1}^m\|\mathbb{E}_{\cD}(\mathbb{E}_\cC(bv_iu_i))\|_2^2=\sum_{i=1}^m\|\mathbb{E}_{\cD}(b\mathbb{E}_\cC(v_iu_i))\|_2^2.
\end{align*}

\vspace{2mm}

By Theorem \ref{intertwining}, $\cB\prec_\cC \cD$.
\end{proof}

\begin{cor}\label{regularcenter} Let $\cM$ be a finite von Neumann algebra and let $\mathscr Z(\cM)\subseteq \cB\subseteq \cM$ be a von Neumann subalgebra. If $\cB\prec_\cM\mathscr Z(\cM)$ then $\cB\prec_\cB\mathscr Z(\cM)$.\end{cor}

\begin{proof} This follows from Proposition \ref{intlower} by letting $\cD= \mathscr Z(\cM)$ and $\cC= \cB$. 
\vspace{2mm}
\end{proof}

\begin{thm} \label{tinyinclusion} Let $\mathscr Z(\mathcal M)\subseteq \mathcal A \subseteq\mathcal M$ be finite separable von Neumann algebras where $\mathcal A$ is regular and abelian. If we have that $\mathcal A \prec^s \mathscr Z(\mathcal M)$ then one can find a countable set of orthogonal projections  $\{r_n\}_{n\in \mathbb N}\subset \mathcal A $ such that $\sum_n r_n =1$ and $\mathcal A =\oplus_n \mathscr Z(\mathcal M)r_n$.
\end{thm}

\begin{proof} First note that by \cite[Lemma 2.4]{dhi19} the condition $\mathcal A \prec^s \mathscr Z(\mathcal M)$ is equivalent to $\mathcal A z \prec \mathscr Z(\mathcal M)$ for all projections $z\in \mathscr Z(\mathcal M)$. By Corollary \ref{regularcenter} this is further equivalent to $\mathcal A z \prec_{\mathcal A} \mathscr Z(\mathcal M)$ for all projections $z\in \mathscr Z(\mathcal M)$.

\vspace{2mm}

Next we show for every projection $0\neq z\in \mathscr Z(\mathcal M)$ there exist a countable family of orthogonal projections  $0\neq r_i\in \mathcal A z$ such that $s:=\sum_n r_n \in \mathscr Z(\mathcal M)z$ and   

\begin{equation}\label{equalcoreners100}
    \mathcal A s=\oplus_n \mathscr Z(\mathcal M)r_n.
\end{equation} 

\vspace{2mm}

Since $\mathcal A z \prec \mathscr Z(\mathcal M)$ there are nonzero projections  $p\in \mathcal A z$, $q\in\mathscr {Z}(\mathcal M)$, a partial isometry $v\in \mathcal A$ and a $\ast$-isomorphism onto its image $\Psi:\mathcal A p\to\mathscr{Z}(\mathcal M)q$ such that $\Psi(x)v=vx$ for all $x\in \mathcal A p$. Since $\mathcal{A}$ is abelian, multiplying by $v$  we further get $\Psi(x)vv^*=xvv^*$ for every $x\in \mathcal A p$. Letting $r=vv^*\in \mathcal A$, this further shows that $\mathcal Ar \subseteq\mathscr{Z}(\mathcal M )r$. Since $\mathscr Z(\mathcal M)\subseteq \mathcal A$ we get that 

\begin{equation}\label{equalcorners101}
    \mathcal Ar =\mathscr{Z}(\mathcal M )r.
\end{equation}

\vspace{2mm}

By relation \eqref{equalcorners101}, for any unitary $u\in \mathscr N_{\mathcal M}(\mathcal A)$ we have that $\mathcal A (uru^*)=\mathscr Z(\mathcal M) (uru^*)$. Since $\mathcal A$ is abelian we have $r \vee (uru^*) = (r- r (uru^*)) +( uru^*- r(uru^*) )+ r(uru^*)$, where $(r- r (uvu^*))$, $uru^*- r(uru^*)$,  $r(uru^*)$ are orthogonal projections in $\mathcal A$. Using this inductively one can find  a countable set of projections $\{r_n\}_n \subset \mathcal A$ such that the central support $s:=z(r)$ of $r$ satisfies $s=\vee_{u\in \mathscr N_{\mathcal M}(\mathcal A)} uru^*= \sum_n r_n$ and $\mathcal A r_n = \mathscr Z(\mathcal M)r_n$  for all $n$. This clearly implies \eqref{equalcoreners100}. 

\vspace{2mm}

The conclusion follows from \eqref{equalcoreners100}, via a standard maximality argument.
\end{proof}

\subsection{Quasi-normalizers}\label{qn} Let $\cN\subset \cM$ be an inclusion of von Neumann algebras. The set $\mathscr{QN}_{\cM}(\cN)\subset \cM$ is defined as the collection of elements $x\in \cM$ for which there exists $x_1,...,x_n\in \cM$ satisfying the following two relations (see \cite[Definition 4.8]{popa99}):

\begin{equation*}
    x\cN\subset \sum_{i=1}^n\cN x_i\quad\text{and}\quad \cN x\subset \sum_{i=1}^nx_i\cN.
\end{equation*}

\vspace{2mm}

Observe that $\mathscr{QN}_{\cM}(\cN)$ is a $*$-subalgebra of $\cM$ containing $\cN$. Its weak operator closure is called the \textit{quasi-normalizer of $\cN$ inside $\cM$}. By construction both $\cN$ and $\cN'\cap \cM$ are subalgebras of $\mathscr{QN}_{\cM}(\cN)''$.

\vspace{2mm}

For further use, we record some formulae for quasi-normalizers of compressions of von Neumann algebras.

\begin{lem}\label{qnrem} \cite{popa06,fgs10} For any $\cQ\subset \cM$ tracial von Neumann algebras the following hold:

\begin{enumerate}
    \item[(1)] For any projection $q\in \cQ'\cap \cM$, $\mathscr{QN}_{q\cM q}(\cQ q )''=q\mathscr{QN}_{\cM}(\cQ)''q$. 
    \item[(2)] For any projection $q\in \cQ$, $\mathscr{QN}_{q\cM q}(q\cQ q)''=q\mathscr{QN}_{\cM}(\cQ)''q$. 
\end{enumerate}
\vspace{2mm}
\end{lem}

\subsection{Wreath-like product groups} Let $A, B$ be arbitrary groups, $I$ an abstract set, and $B\curvearrowright I$ a (left) action of $B$ on $I$. A group $G$ is called a \textit{wreath-like product} of $A$ and $B$ corresponding to the action $B\curvearrowright I$ if it is an extension of the form $1\to \oplus_{i\in I}A_i\to G\xrightarrow{\epsilon}B\to 1$ such that $A_i\cong A$ and $gA_ig^{-1}=A_{\epsilon(g)i}$ for every $i\in I$ and $g\in G$. The set of all wreath-like products of $A$ and $B$ is denoted by $\mathcal{WR}(A,B\curvearrowright I)$. When $I=B$ and $B$ acts on itself by left translation, we denote all wreath-like products by $\mathcal{WR}(A,B)$. As remarked in \cite{cdi22}, $G\in\mathcal{WR}(A,B\curvearrowright I)$ if and only if there exists a cocycle action $B\curvearrowright^{\alpha,c}A^{(I)}$ such that $\alpha_b(A_i) = A_{b\cdot i}$ and $G\cong A^{(I)}\rtimes_{\alpha,c}B$. 

\vspace{2mm}

Moreover, if $(\cM,\tau)$ is a tracial von Neumann algebra and $B$ is a group, the von Neumann algebra $\cN$ is said to be a \textit{wreath-like product of $\cM$ corresponding to an action $B\curvearrowright I$} if it is isomorphic to $\cM^{(I)}\rtimes_{\beta,c}B$ for a cocycle action $B\curvearrowright^{\beta,c}\cM^{(I)}$ with $\beta_b(\cM^i)=\cM^{b\cdot i}$ for $b\in B$, $i\in I$. In \cite[Example 3.2]{cdi22} one can see that for $G\in\mathcal{WR}(A,B\curvearrowright I)$ we have $\cL(G)\cong \cL(A)^{(I)}\rtimes_{\alpha,c}B$. 

\vspace{2mm}

\begin{defn}\label{bounded2cocycle}Following \cite{amcos23} we say that a group $W\in\mathcal{W}\mathcal{R}_b(C,D\curvearrowright I)$  if $W\in\mathcal{W}\mathcal{R}(C,D\curvearrowright I)$ and the natural $2$-cocycle $c: D\times D \rightarrow C^{(I)}$ associated with the wreath-like short exact sequence satisfies the following uniform bounded cardinality condition on supports:  

\begin{equation*}
\sup_{x,y\in D} |\text{supp}(c(x,y))|<\infty.
\end{equation*} 
\end{defn}
\vspace{2mm}

We continue with a few results on controlling quasi-normalizers of certain subalgebras of the core of a wreath-like product von Neumann algebra. These results are in the spirit of \cite[Theorem 3.1]{popa06} or \cite[Lemma 4.1]{vaes08} and for the sake of completeness we include a few proofs along these lines.

\begin{thm}\label{controlincore} Let $G\in \mathcal{WR}(A,B\curvearrowright I)$ for nontrivial groups $A, B$ with associated cocycle action $B\curvearrowright^{\alpha,c}A^{(I)}$, and let $\cN$ be a tracial von Neumann algebra. Assume that $G\curvearrowright^{\sigma} (\cN, \tau) $ is a trace-preserving action and denote by $\cM=\cN\rtimes_{\sigma} G$ the corresponding crossed product von Neumann algebra. Let $F\subset I$ be a finite set and let $p\in \cN\rtimes A^F$ be a projection. Let $a_n \in \mathscr{U}(p (\cN\rtimes A^F) p)$ be a sequence of unitaries such that for every $x_1,x_2 \in \cN\rtimes  A^{(I)}$ and every subset $K\subsetneq F$ we have  

\begin{equation}\label{notintertwiningNrtimesAk}
    \lim_{n\rightarrow \infty}\| \mathbb{E}_{\cN\rtimes A^K}(x_1 a_n x_2)\|_2 = 0.
\end{equation}

\vspace{2mm}

Then for every $y_1,y_2\in \cM$ such that $\mathbb{E}_{\cN\rtimes A^{(I)} \text{Norm}(F) }(y_1)=\mathbb{E}_{\cN\rtimes A^{(I)} \text{Norm}(F)}(y_2)=0$, we have that 

\begin{equation}\label{notintertwiningNrtimesAF}
    \lim_{n\rightarrow \infty}\| \mathbb{E}_{\cN\rtimes A^F}(y_1 a_n y_2)\|_2 = 0.
\end{equation}

\vspace{2mm}

Here, for every subset $F\subseteq I$ we have denoted by ${\text{Norm}}(F)= \{g \in B \,:\, gF =F\}$.
\end{thm}

\begin{proof} Let $a_n=\sum_{h\in A^F}m_h^nu_h$ be the Fourier expansion in $\cN\rtimes A^F$ where $m_h^n\in \cN$ for all $h\in A^{(I)}$. Relation \eqref{notintertwiningNrtimesAk} above implies that for every subset $K\subsetneq F$ and every finite set $S\subseteq A^{(I\setminus K)}$ we have that

\begin{equation}\label{limitcoefficientunitaries}
    \lim_{n\to\infty}\sum_{h\in A^KS}\|m_h^n\|_2^2=0.
\end{equation}

\vspace{2mm}

Using basic $\|\cdot\|_2$-approximations it suffices to show \eqref{notintertwiningNrtimesAF} when $y_i=n_iu_{k_ig_i}$ where $n_i\in \cN$, $k_i\in A^{(I)}\text{Norm}(F)$, and $g_i$ is a representative of a nontrivial coset for $I/\text{Norm}(F)$. Assume $k_i=b_it_i$ where $b_i\in A^{(I)}$ and $t_i\in\text{Norm}(F)$. Here the $t_i$'s are representatives of $A^{(I)}\text{Norm}(F)/A^{(I)}$. Using this, we see that

\begin{align}\label{ENrtimesA^Fys}
    \|\mathbb{E}_{\cN\rtimes A^F}(y_1a_ny_2)\|_2^2&=\|\mathbb{E}_{\cN\rtimes A^F}(n_1u_{k_1g_1}a_nn_2u_{k_2g_2})\|_2^2\nonumber\\
    &=\|n_1\mathbb{E}_{\cN\rtimes A^F}(u_{k_1g_1}a_nu_{k_2g_2})\sigma_{(k_2g_2)^{-1}}(n_2)\|_2^2\\
    &\leq\|n_1\|_{\infty}^2\cdot\|n_2\|_{\infty}^2\cdot\|\mathbb{E}_{\cN\rtimes A^F}(u_{k_1g_1}a_nu_{k_2g_2})\|_2^2.\nonumber
\end{align}

\vspace{2mm}

Using $a_n=\sum_hm_h^nu_h$ we have that

\begin{align}\label{pluginan}
    \|\mathbb{E}_{\cN\rtimes A^F}(u_{k_1g_1}a_nu_{k_2g_2})\|_2^2&=\left\|\sum_{h\in A^F}\mathbb{E}_{\cN\rtimes A^F}(u_{k_1g_1}m_h^nu_hu_{k_2g_2})\right\|_2^2\nonumber\\
    &=\left\|\sum_{h\in A^F}\sigma_{k_1g_1}(m_h^n)\mathbb{E}_{\cN\rtimes A^F}(u_{k_1g_1\cdot h\cdot k_2g_2})\right\|_2^2.
\end{align}

\vspace{2mm}

Observe that $k_1g_1\cdot h\cdot k_2g_2=b_1t_1g_1\cdot h\cdot b_2t_2g_2=b_1\alpha_{t_1g_1}(hb_2)c(t_1g_1,t_2g_2)t_1g_1t_2g_2$. Therefore, $k_1g_1\cdot h\cdot k_2g_2\in A^F$ only if $t_1g_1t_2g_2\in A^{(I)}$. Hence, there exists $d_{1,2}\in A^{(I)}$ such that $\mathbb{E}_{\cN\rtimes A^F}(u_{k_1g_1.h.k_2g_2})=u_{b_1\alpha_{t_1g_1}(hb_2)c(t_1g_1,t_2g_2)d_{1,2}}$. Continuing from \eqref{pluginan}, we have

\begin{align}\label{sumwithsubscript}
    \|\mathbb{E}_{\cN\rtimes A^F}(u_{k_1g_1}a_nu_{k_2g_2})\|_2^2&=\sum_{\substack{h\in A^F,\\b_1\alpha_{t_1g_1}(hb_2)c(t_1g_1,t_2g_2)d_{1,2}\in A^F}}\|\sigma_{k_1g_1}(m_h^n)\|_2^2.
\end{align}

\vspace{2mm}

If $b_1\alpha_{t_1g_1}(hb_2)c(t_1g_1,t_2g_2)d_{1,2}\in A^F$, using that $\alpha_{t_1g_1}$ is an automorphism, we obtain

\begin{align*}
    \alpha_{t_1g_1}(h)&\in b_1^{-1}A^Fd_{1,2}^{-1}c(t_1g_1,t_2g_2)^{-1}\alpha_{t_1g_1}(b_2^{-1})=A^F\underbrace{b_1^{-1}d_{1,2}^{-1}c(t_1g_1,t_2g_2)^{-1}\alpha_{t_1g_1}(b_2^{-1})}_{=:\gamma(t_1,g_1,t_2,g_2)\in A^I}.
\end{align*}

\vspace{2mm}

Hence, $\alpha_{t_1g_1}(h)\in A^F\gamma(t_1,g_1,t_2,g_2)$. Applying $\alpha_{t_1^{-1}}$ we get

\begin{align}\label{compositionrhot1}
    \alpha_{t_1^{-1}}\circ \alpha_{t_1g_1}(h)&\in \alpha_{t_1^{-1}}(A^F\gamma(t_1,g_1,t_2,g_2))=A^F\alpha_{t_1^{-1}}(\gamma(t_1,g_1,t_2,g_2))
\end{align}

\vspace{2mm}

as $t_1^{-1}\in\text{Norm}(F)$. On the other hand, using the 2-cocycle action and \eqref{compositionrhot1}, we have

\begin{equation}
    \alpha_{g_1}(h)\in c(t_1^{-1},t_1g_1)^{-1}A^F\alpha_{t_1^{-1}}(\gamma(t_1,g_1,t_2,g_2))c(t_1^{-1},t_1g_1).
\end{equation}

\vspace{2mm}

Applying $\alpha_{g_1^{-1}}$ we obtain

\begin{align}\label{h}
    h&\in \alpha_{g_1^{-1}}(A^F)\underbrace{c(g_1^{-1},g_1)^{-1}\alpha_{g_1^{-1}}(c(t_1^{-1},t_1g_1)^{-1}\alpha_{t_1^{-1}}(\gamma(t_1,g_1,t_2,g_2))c(t_1^{-1},t_1g_1))c(g_1^{-1},g_1)}_{=:\delta(t_1,g_1,t_2,g_2)}\nonumber\\
    &=A^{g_1^{-1}F}\delta(t_1,g_1,t_2,g_2)
\end{align}

\vspace{2mm}

By \eqref{h}, we have that the subscript in the summation of equation \eqref{sumwithsubscript} implies that $h\in A^F\cap A^{g_1^{-1}F}\delta(t_1,g_1,t_2,g_2)$. We continue from \eqref{sumwithsubscript} to obtain

\begin{align}\label{ENrtimesA^Funitaries}
    \|\mathbb{E}_{\cN\rtimes A^F}(u_{k_1g_1}a_nu_{k_2g_2})\|_2^2&=\sum_{h\in A^F\cap A^{g_1^{-1}F}\delta(t_1,g_1,t_2,g_2)}\|\sigma_{k_1g_1}(m_h^n)\|_2^2\nonumber\\
    &\leq \sum_{h\in A^K\delta(t_1,g_1,t_2,g_2)}\|m_h^n\|_2^2
\end{align}

\vspace{2mm}

By \eqref{limitcoefficientunitaries}, we get that \eqref{ENrtimesA^Funitaries} goes to zero. Using \eqref{ENrtimesA^Fys} we obtain \eqref{notintertwiningNrtimesAF}, as wanted.
\end{proof}\vspace{0.5mm}

\begin{cor}\label{controlquasinorm3}(\cite[Lemma 4.1]{vaes08}) Let $G\in \mathcal{WR}(A,B\curvearrowright I)$ and let $\mathcal N$ be a tracial von Neumann algebra. Let $G\curvearrowright \mathcal N$ be a trace-preserving action and denote by $\cM=\mathcal N\rtimes G$ be the corresponding crossed-product von Neumann algebra. Let $F \subset I$ be a finite subset and let $q\in \mathcal N\rtimes A^F$ be a non-zero projection.  Assume $\cQ \subseteq q(\cN\rtimes A^F)q$ is a von Neumann subalgebra satisfying $\cQ \nprec_{\mathcal N\rtimes A^F} \mathcal N \rtimes A^K$, for all $K\subsetneq F$. Then we have that $\mathscr {Q N}_{q \cM q} (\cQ)\subseteq q(\cN\rtimes (A^{(I)}{\text{Norm}}(F)))q$.
\end{cor}

\begin{proof} If $\cQ\not\prec_{\cN\rtimes A^F}\cN\rtimes A^K$ for any subset $K\subsetneq F$, by Popa's intertwining techniques, there exists a sequence of unitaries $(a_n)_n\subset \cQ$ such that for all $x_1,x_2\in \cN\rtimes A^F$ we have condition \eqref{notintertwiningNrtimesAk} from the previous theorem is satisfied. Let $x\in\mathscr{QN}_{q\cM q}(\cQ)$ and denote $y=x-\mathbb{E}_{\cN\rtimes A^{(I)}\text{Norm}(F)}(x)$. We claim $y=0$. Since $x$ is in the quasi-normalizer of $\cQ$, so is $y$ and hence, $yq=\varphi(q)y$ for all $q\in \cQ$. It follows that for every $z\in \mathcal{M}$ with $\mathbb{E}_{\cN\rtimes A^{(I)}\text{Norm}(F)}(z)=0$, since $\mathbb{E}_{\cN\rtimes A^{(I)}\text{Norm}(F)}(y)=0$, we have

\begin{equation*}
    \|\mathbb{E}_{\cN\rtimes A^F}(zy^*)\|_2=\|\mathbb{E}_{\cN\rtimes A^F}(zy^*)\varphi(a_n)\|_2=\|\mathbb{E}_{\cN\rtimes A^F}(za_ny^*)\|_2\to 0
\end{equation*}

\vspace{2mm}

by Theorem \ref{controlincore}. This means that $\mathbb{E}_{\cN\rtimes A^F}(zy^*)=0$ for all $z\in \mathcal{M}$ with $\mathbb{E}_{\cN\rtimes A^{(I)}\text{Norm}(F)}(z)=0$. By letting $z=y$, we have $\mathbb{E}_{\cN\rtimes A^F}(yy^*)=0$ or $y=0$, as claimed.
\end{proof}\vspace{0.5mm}

\begin{cor}\label{quasinormcontrol2} Let $G\in \mathcal{WR}(A,B\curvearrowright I)$ and let $\mathcal N$ be a tracial von Neumann algebra. Let $G\curvearrowright\mathcal N$ be a trace-preserving action and denote by $\cM=\mathcal N\rtimes G$ the corresponding crossed-product von Neumann algebra. Let $F \subset I$ be a finite subset and let $q\in \mathcal N\rtimes A^F$ be a non-zero projection. Assume $\cQ \subseteq q \cM q$ is a von Neumann subalgebra satisfying  $\cQ \prec_{\cM} \mathcal N\rtimes A^F$ and $\cQ \nprec \mathcal N \rtimes A^K$, for all $K\subsetneq F$. Then, $\mathscr {Q N}_{q \cM q} (\cQ)''\prec \mathcal N\rtimes (A^{(I)} {\text{Norm}}(F))$.
\end{cor}

\begin{proof} If $\cQ\prec_{\cM}\cN\rtimes A^F$ there are projections $f\in \cQ \text{ and } s\in \cN\rtimes A^F$ and a $*$-isomorphism onto its image $\Theta:f\cQ f\to  \Theta(f\cQ f)=:\cP\subseteq s(\cN\rtimes A^F)s$ that is implemented by a nonzero partial isometry  $v\in s \cM f$, i.e. 

\begin{equation}\label{intertwiningQ}
    \Theta(x)v=vx.
\end{equation} 

\vspace{2mm}

Notice that $r=vv^*\in \cP'\cap s\cM s$, and using the previous relation we obtain $\cP vv^*=v\cQ v^*$. 

\vspace{2mm}

We can assume without loss of generality that $\text{supp}(\mathbb{E}_{\cN\rtimes A^F}(r))=s$. Next we observe that $\cP\nprec_{\mathcal N\rtimes A^F} \mathcal N \rtimes A^K$ for any proper subset $K\subset F$; otherwise, composing it with the intertwining \eqref{intertwiningQ} we would get that $\cQ \prec \mathcal N \rtimes A^K$ for some set $K\subsetneq F$, a contradiction to the assumption. Using this in combination with Corollary \ref{controlquasinorm3} and properties of quasi-normalizers we can see that, if $v=wt$ for $w\in\mathscr{U}(\cM)$ and $t=v^*v$,

\begin{align*}
r(\cN\rtimes (A^{(I)}\text{Norm}(F)))r&\supseteq \mathscr{QN}_{r\cM r}(\cP r)''&&&\text{(by Lemma \ref{qnrem})}\\
    &=\mathscr{QN}_{r\cM r}(vt \cQ  tv^*)''=\mathscr{QN}_{v\cM v^*}(v\cQ v^*)''\\
    &=\mathscr{QN}_{wt\cM tw^*}(wt \cQ tw^*)''\\
    &=w\mathscr{QN}_{t\cM t}(t\cQ t)''w^*=w t\mathscr{QN}_{p\cM p}(\cQ)''tw^* &&&\text{(by Lemma \ref{qnrem})}\\
    &=v\mathscr{QN}_{p\cM p}(\cQ)''v^*.
\end{align*}

\vspace{2mm}

This shows that $\mathscr{QN}_{\cM}(\cQ)''\prec \cN\rtimes (A^{(I)}\text{Norm}(F))$, as desired. 
\end{proof}

\vspace{0.5mm}

\subsection{A remark on splitting central extensions} In this subsection, we leverage certain results from Serre's work to establish a dichotomy regarding whether a group undergoes splitting under the assumptions outlined in \ref{theoremA}. We will initiate by revisiting the results obtained by Serre.

\begin{thm}\label{infiniteabelianization}(\cite{serre}) Let $G$ be a countable group, and $Z\leqslant Z(G)$. Assume $G/[G,G]$ is finite. If $G/Z$ has property (T), then $G$ has property (T).
\end{thm}

\begin{cor}\label{nonsplit}(\cite{serre}) Let $G$ be a finitely generated group. Assume $G/Z(G)$ has property (T). If the short exact sequence

\begin{equation*}
    1\longrightarrow Z(G)\longrightarrow G\longrightarrow G/Z(G)\longrightarrow 1
\end{equation*}

\vspace{2mm}

does not split, then $G$ has property (T).
\end{cor}

We proceed to establish the following lemma, applying Serre's results together with Proposition \ref{subgroupscocyclecrossedproducts}.

\vspace{1mm}

\begin{lem}\label{splittinglemma} Assume $B, G, H$ are countable discrete groups such that $B$ is abelian, $G$ is ICC and property (T), and $\cL(H)\cong \cL(B\times G)$. If $H/Z(H)$ has property (T) and trivial abelianization, then the short exact sequence 

\begin{equation*}
    1\longrightarrow Z(H)\longrightarrow H\longrightarrow H/Z(H)\longrightarrow 1.
\end{equation*}

\vspace{2mm}

splits up to a finite central subgroup.
\end{lem}

\begin{proof} Let $Z_0=Z(H)$ and $B_0=H/Z(H)$. Note that $H\cong Z_0\rtimes_{c_0}B_0$ where $B_0$ has property (T). Consider the subgroup $H^{(1)}=[H,H]$. By Proposition \ref{subgroupscocyclecrossedproducts}, we have that $H^{(1)}=Z_1\rtimes_{c_1}B_0$ where $c_1$ is cohomologous to $c_0$, and $Z_1=H^{(1)}\cap Z_0$. If $H^{(1)}$ has property (T), then $\cL(H^{(1)})\prec\cL(G)$, as $\cL(B)$ has Haagerup property. If that is the case, then $\cL(Z_1)\prec \cL(G)$ which implies $Z_1$ is finite. Since $\text{Im}\:c_1\subset Z_1$, we obtain that the short exact sequence splits up to the finite normal subgroup $Z_1$. Otherwise, if $H^{(1)}$ doesn't have property (T), by Serre's Theorem \ref{infiniteabelianization}, $H^{(1)}/H^{(2)}$ is infinite. By Proposition \ref{subgroupscocyclecrossedproducts}, let $H^{(2)}=Z_2\rtimes_{c_2}B_0$ where $Z_2\leqslant Z_0$ and $c_2$ is cohomologous to $c_1$, and hence to $c_0$. Since $B_0$ has trivial abelianization, we have that $H/H^{(2)}\cong Z_0/Z_2$ is abelian. This means $H^{(1)}\subseteq H^{(2)}$, which contradicts the fact that $H^{(1)}/H^{(2)}$ is infinite. Hence, the short exact sequence must split up to a finite central subgroup. 
\end{proof}

\section{The FC-center of a countable group}\label{section4}

In this section we revise several properties of the finite conjugacy radical (FC-center) and the hyper-FC center of a group that will be used in the sequel. 

\vspace{2mm}

Let $H$ be a countable discrete group. The \emph{FC-center} of $H$, denoted by $H^{fc}$, consists of all $g\in H$ whose conjugacy class $\mathscr O(g)=\{ hgh^{-1}\,:\, h\in H\}$ is finite. One can easily see that $H^{fc} \leqslant H$ is a normal subgroup. The upper FC-series of $H$ is the series

\begin{equation*}
    \{1\}=H^h_0\leqslant H^h_1\leqslant ...\leqslant H^h_n\leqslant...
\end{equation*}

\vspace{2mm}

where $H^h_{n+1}/H^h_n$ is the set of all FC-elements of $H/H^h_n$. Note that $H_1^h=H^{fc}$. This series stabilizes at some ordinal. We call \textit{hyper-FC center} of $H$, and denote it by $H^{hfc}$, such group. Note that $H^{hfc}$ is an amenable and normal subgroup of $H$, and by \cite[Proposition 2.2]{hfc}, $H/H^{hfc}$ is ICC.

\vspace{2mm}
In the remaining part we establish several properties of the FC-center of a group which are needed when reconstructing a group's center from W$^*$ - equivalence. 

\vspace{2mm}

For a countable group $H$ we consider $\mathscr O(g_1), \mathscr O(g_2),\ldots$ (with $g_1=1$) an enumeration of all finite conjugacy orbits of $H$. Notice that $H^{fc}=\cup_i \mathscr O(g_i)$. Consider the subgroups $H_n:=\langle \mathscr O(g_1),...,\mathscr O(g_n)\rangle \leqslant H^{fc}$. Under these notations we have the following properties.

\begin{lem}\label{towergroups}\begin{enumerate}
    \item $H_n \ngroup H$ is a center-by-finite normal subgroup for all $n\in \mathbb N$. 
    \item We have $H_1\leqslant H_2\leqslant \cdots \leqslant H_n \leqslant \cdots \leqslant H^{fc}$ and $\cup_n H_n = H^{fc}$. 
\end{enumerate}
\end{lem} 

\begin{proof} We will show that $H_n$ is center-by-finite.  If we consider the map $H\to Sym(\mathscr{O}(g_1),...,\mathscr{O}(g_n))$ given by $g\mapsto (t\mapsto gtg^{-1})$ we obtain that $C_{H}(H_n)$ is of finite index in $H$. By the second isomorphism theorem, it follows that $[H_n:Z(H_n)]=[C_{H}(H_n)H_n:C_{H}(H_n)]\leqslant[H:C_{H}(H_n)]<\infty$, showing that $H_n$ is center-by-finite. This shows that $H^{fc}$ is a tower of virtually abelian subgroups.
\end{proof}

Now we introduce a concept which expands the notion of BFC-group, \cite[Section 14.5]{robinson},  to the case of inclusions of groups.

\begin{defn}
 Let $H\leqslant G$ be groups. We say that \emph{$H$ is BFC inside $G$} if 
 \begin{equation*}
     \sup_{h\in H}|\mathscr O_G(h)|<\infty.
 \end{equation*}   
\end{defn}

When $H=G$ this amounts to $H$ being a  BFC-group in the classical sense.

\begin{ex} From the prior lemma, since the centralizer $C_H(H_n)$ is finite index in $H$, $H_n$ is BFC inside $H$.
\end{ex}

\begin{lem}\label{bfceqint}Let $H\leqslant K\leqslant  G$ be groups such that $[K:H]<\infty$. If $H$ is BFC inside $G$ then $\cL(K)\prec_{\cL(G)} \mathscr Z(\cL(G))$.   
\end{lem}

\begin{proof} Assume $H$ is BFC inside $G$. Thus, for all $g\in H$ we have

\begin{align*}
    \|\mathbb{E}_{\mathscr{Z}(\cL(G))}(u_g)\|_2^2&=\left\|\frac{1}{|\mathscr{O}_G(g)|}\sum_{h\in\mathscr{O}_G(g)}u_h\right\|_2^2=\frac{1}{|\mathscr{O}_G(g)|}\geq\frac{1}{\sup_{h\in H}|\mathscr{O}_G(h)|}>0.
\end{align*}

\vspace{2mm}

By Theorem \ref{intertwining}, $\cL(H)\prec_{\cL(G)}\mathscr{Z}(\cL(G))$. As $[K:H]<\infty$ we further get  $\cL(K)\prec_{\cL(G)}\mathscr{Z}(\cL(G))$.\end{proof}

The following theorem gives us the converse implication.

\begin{thm}\label{relativecommutatorgroup} Let $H$ be a countable group. Then the following assertions are equivalent: 
\begin{enumerate}
    \item[(1)] The  FC-center $H^{fc}$ is BFC inside $H$ and $[H^{hfc}: H^{fc}]<\infty$.
    \item[(2)]  $\cL(H^{hfc})\prec_{\cL(H)} \mathscr Z(\cL (H))$.
\end{enumerate}
\end{thm}

\begin{proof} (1) $\Rightarrow$ (2) follows directly from Lemma \ref{bfceqint}. For the converse, assume $\cL(H^{hfc})\prec_{\cL(H)} \mathscr Z(\cL (H))$. Using Corollary \ref{regularcenter}, we have that $\cL(H^{hfc})\prec_{\cL(H^{hfc})} \mathscr{Z}(\cM)$ where $\cM = \cL (H)$. By Popa's intertwining techniques, we can find $x_1,...,x_n\in\cL(H^{hfc})$ and $c>0$ for which

\begin{equation*}
    \sum_{i=1}^n\|\mathbb{E}_{\mathscr{Z}(\cM)}(hx_i)\|_2^2\geq c,\quad\text{ for all }h\in\mathscr{U}(\cL(H^{hfc})).
\end{equation*}

\vspace{2mm}

Using basic $\|\cdot\|_2$-approximations, decreasing $c>0$ and increasing $n$, if necessary, we may assume $x_i=u_{g_i}$ for some $g_i\in H^{hfc}$. From the prior equation, we get for all $h\in H^{hfc}$

\begin{equation*}
    0<c\leq \sum_{i=1}^n\|\mathbb{E}_{\mathscr{Z}(\cM)}(u_{hg_i})\|_2^2=\sum_{i=1}^n\left\|\frac{1}{|\mathscr{O}_H(hg_i)}\sum_{l\in \mathscr{O}_H(hg_i)}u_l\right\|_2^2=\sum_{i=1}^n\frac{1}{|\mathscr{O}_H(hg_i)|}.
\end{equation*}

\vspace{2mm}

Hence, for every $h\in H^{hfc}$ there is $g_i\in H^{hfc}$ with 

\begin{equation}\label{orbitbound}|\mathscr{O}_H(hg_i)|\leq n/c.\end{equation} 

\vspace{2mm}

Since $h= (hg_i)g_i^{-1}$, \eqref{orbitbound} further implies  $H^{hfc}\subseteq \cup^n_{i=1} H^{fc} g_i^{-1}$ and hence $[H^{hfc}:H^{fc}]\leq n$. 

\vspace{2mm}

To see the other assertion, fix $h\in H^{fc}$ together with the corresponding $g_i\in H^{hfc}$ satisfying \eqref{orbitbound}. Since $h^{-1}\in H^{fc}$ and $\mathscr O_H(g_i)\subseteq \mathscr O_H(h^{-1})\mathscr O_H(hg_i)$ we get $g_i \in H^{fc}$. Finally, as $\mathscr{O}_H(h)\subset \mathscr{O}_H(hg_i)\mathscr{O}_H(g_i)$, \eqref{orbitbound} entails $|\mathscr{O}_H(h)|\leq n/c\cdot\max_{i}|\mathscr{O}_H(g_i)|$. This shows $H^{fc}$ is BFC inside $H$.
\end{proof}

We remark the following consequence of the proof of the prior theorem.

\begin{cor} Let $H$ be a countable group. Then, the FC-center $H^{fc}$ is BFC inside $H$ if and only if $\cL(H^{fc})\prec_{\cL(H)}\mathscr{Z}(\cL(H))$.
\end{cor}

Observe that if $|[H^{fc},H]|<\infty$ then $H^{fc}$ is BFC inside of $H$. Indeed, if $h\in H^{fc}$ and $h_1,...,h_n$ are distinct elements in $\mathscr{O}_H(h)$ then $h_i=g_ihg_i^{-1}$ with $g_i\in H$. Note that $1,h_1^{-1}h_2,...,h_1^{-1}h_n$ are distinct commutators as $h_1^{-1}h_i=[h^{-1},g_1]^{-1}[h^{-1},g_i]\in [H^{fc},H]$. However, as we will see shortly the converse implication is not true. We are grateful to Stefaan Vaes for providing us with the following example.

\begin{ex} Consider the group $\Z/2\Z$ and let $\sigma\in\Z/2\Z$ be its non-trivial element. Let $\Z/2\Z$ act on $\Z^2$ by $\sigma\cdot(a,b)=(b,a)$ for all $a,b\in\Z$. Let $H=\Z^2\rtimes\Z/2\Z$ be the semidirect product generated by $(a,b)\in\Z^2$ and the element $\sigma$ of order 2 with $\sigma(a,b)\sigma^{-1}=(b,a)$. Then, $H^{fc}=\Z^2$. For $(a,b)\in\Z^2$ we have $\mathscr{O}_H((a,b))=\{(a,b),(b,a)\}$, showing that $H^{fc}$ is BFC inside $H$. Moreover, one can check that $\cL(H)\cong\mathbb{M}_2(\C)\:\overline{\otimes}\:L^{\infty}([0,1])$ and therefore, $\cL(H)\prec\mathscr{Z}(\cL(H))$. On the other hand, since $(a,0)\sigma(a,0)^{-1}\sigma^{-1}=(a,-a)$ for all $a\in \Z$, we have that $[H^{fc},H]$ is infinite.
\end{ex}

\vspace{2mm}

For further use we record the following result which is a consequence of Theorem \ref{relativecommutatorgroup} and a classical result of B. H. Neumann for BFC-groups \cite{bhneumann}.

\begin{cor} Let $H$ be any countable group and let $\mathcal M= \cL(H)$. If $\cL(H^{fc})\prec_{\mathcal M} \mathscr Z(\cL (H))$ then the commutator $[H^{fc},H^{fc}]$ is finite.
\end{cor}
\begin{proof} As $\cL(H^{fc})\prec_{\mathcal M} \mathscr Z(\cL (H))$, Theorem \ref{relativecommutatorgroup} implies that $H^{fc}$ is a BFC in $H$. Thus $H^{fc}$ is a BFC group itself and using \cite[Theorem 3.1]{bhneumann} we get that $[H^{fc}, H^{fc}]$ is finite.
\end{proof}

\subsection{Weakly compact actions}

Let $G\curvearrowright ^\sigma \cP$ be a trace-preserving action of a group $G$ on a tracial von Neumann algebra $(\cP, \tau)$. We call $\sigma$ \textit{compact} if $\sigma(G)\subset Aut(\cP)$ is pre-compact in the point ultra-weak topology. Ozawa and Popa introduced in \cite{popaozawa10} a generalization of this notion which turned out fundamental to the structural study of von Neumann algebras; $\sigma$ is \textit{weakly compact} if one can find a net $(\eta_n)_n\subset L^2(\cP\otimes \cP)_+$ of unit vectors satisfying the following relations: 

\begin{enumerate}
    \item[(i)]$\|\eta_n-(v\otimes \overline{v})\eta_n\|_2\to 0$, for every $v\in\mathscr{U}(\cP)$;
    \item[(ii)] $\|\eta_n-(\sigma_g\otimes\overline{\sigma}_g)(\eta_n)\|_2\to 0$, for every $g\in G$;
    \item[(iii)] $\langle(x\otimes 1)\eta_n,\eta_n\rangle=\tau(x)=\langle\eta_n,(1\otimes \overline{x})\eta_n\rangle$, for every $x\in \cP$ and every $n$.
\end{enumerate}

\vspace{2mm}

Further, an embedding of finite von Neumann algebras $\cP\subset \cM$ is called \emph{weakly compact} if the natural action by conjugation of $\mathscr{N}_{\cM}(\cP)$ on $\cP$ is weakly compact.

\vspace{2mm}

If $H \curvearrowright^\sigma \cL(H^{fc})$ is the action by conjugation, one can easily see the sequence $\eta_n=\frac{1}{|\mathscr{O}_n|^{1/2}}\sum_{g\in\mathscr{O}_n}u_g\otimes u_g$, where $|\mathscr{O}_n|<\infty$, satisfies conditions (ii) and (iii) above. However, we do not necessary have condition (i) since there is no obvious connection between the action by conjugation and the translation action. For this reason, we modify the previous definition to incorporate our needs while also preserving a form of weak compactness for $H\curvearrowright^\sigma \cL(H^{fc})$. 

\vspace{1mm}

\begin{defn}\label{weakerweakcompactness}Let $G\curvearrowright ^{\sigma,c} \cP$ be a cocycle action of a group $G$ on a finite von Neumann algebra $(\cP,\tau)$. Let $\mathcal G \subset \mathscr U(\mathcal P)$ be a subgroup satisfying $\mathcal G''=\mathcal P$. We say the action is \textit{weakly compact relative to $\mathcal G$ with respect to the trace $\tau$} if there is a net $(\eta_n)_n\subset L^2(\cP\otimes \cP)_+$  of unit vectors  satisfying the following relations:

\begin{enumerate}
    \item[(j)] $\|(v\otimes \overline{v})\eta_n-\eta_n\|_2\to 0$ as $n\to\infty$ for all $v\in \mathcal{G}$;
    \item[(jj)] $\|(\sigma_g\otimes\overline{\sigma}_g)\eta_n-\eta_n\|_2\to 0$ as $n\to\infty$ for all $g\in G$; and
    \item[(jjj)] $\tau(x)=\langle (x\otimes 1)\eta_n,\eta_n\rangle$ for all $x\in \mathcal{P}$,
\end{enumerate}

where the $\|\cdot\|_2$-norm is the induced norm from the trace $\tau$.

\vspace{2mm}

Similarly,  an embedding of finite von Neumann algebras $\cP\subset \cM$ is \emph{weakly compact relative to $\mathcal{G}$ with respect to the trace $\tau$} if the conjugation action $\mathscr{N}_{\cM}(\cP)\curvearrowright\cP$ is weakly compact relative to $\mathcal{G}$ with respect to $\tau$.
\end{defn}

\begin{lem}\label{folnerextension} Let $G=A\rtimes_{\alpha,c}F$ be a cocycle semiproduct where $A$ is amenable and $F$ is finite. Then, one can find a sequence $(A_n)_n\subset A$ that is F\o lner for $A$ and also satisfies 

\begin{equation*}
    \lim_{n\to\infty}\frac{|fA_nf^{-1}\:\triangle\: A_n|}{|A_n|}=0,\text{ for all }f\in F.
\end{equation*}
\end{lem}

\begin{proof} By \cite[Proposition 5]{tao}, there is a F\o lner sequence $(A_n)_n\subset A$ such that $(A_nF)_n\subset G$ is a F\o lner sequence for $G$; that is,

\begin{equation}\label{folnerseqforG}
    \lim_{n\to\infty}\frac{|sA_nF\:\triangle\:A_nF|}{|A_nF|}=0,\text{ for all }s\in G.
\end{equation}

\vspace{2mm}

Fix $s\in F$. Then, $sA_nF=A_n^ssF=\sqcup_{t\in F}A_n^sc(s,s^{-1}t)t$. Hence, \eqref{folnerseqforG} is equivalent to 

\begin{equation}\label{conjugatefolnerseq}
    \lim_{n\to\infty}\sum_{t\in F}\frac{|A_n^sc(s,s^{-1}t)\:\triangle\:A_n|}{|A_nF|}=\lim_{n\to\infty}\frac{|sA_nF\:\triangle\:A_nF|}{|A_nF|}=0.
\end{equation}

\vspace{2mm}

Next, observe that

\begin{align}\label{conjugateforF}
    \frac{|A_n^s\:\triangle\:A_n|}{|A_n|}&=\frac{\sum_{t\in F}|A_n^s\:\triangle\:A_n|}{|A_nF|}\leq\frac{\sum_{t\in F}(|A_n^s\:\triangle\:A_nc(s,s^{-1}t)^{-1}|+|A_nc(s,s^{-1}t)^{-1}\:\triangle\:A_n|)}{|A_nF|}\nonumber\\
    &=\frac{\sum_{t\in F}(|A_n^sc(s,s^{-1}t)\:\triangle\:A_n|+|A_n\:\triangle\:A_nc(s,s^{-1}t)|)}{|A_nF|}\\
    &=\frac{\sum_{t\in F}|A_n^sc(s,s^{-1}t)\:\triangle\:A_n|}{|A_nF|}+\frac{\sum_{t\in F}|A_n\:\triangle\:A_nc(s,s^{-1}t)|}{|A_nF|}\nonumber.
\end{align}

\vspace{2mm}

Since $\{c(s,s^{-1}t):s,t\in F\}\subset A$ is a finite set and $(A_n)_n$ is a F\o lner sequence for $A$, 

\begin{equation}\label{finitesumfolner}
    \lim_{n\to\infty}\frac{\sum_{t\in F}|A_n\:\triangle\:A_nc(s,s^{-1}t)|}{|A_n|}=0.
\end{equation}

\vspace{2mm}

Combining \eqref{conjugateforF} with the limits \eqref{conjugatefolnerseq} and \eqref{finitesumfolner}, we obtain $\lim_{n\to\infty}\frac{|A_n^s\:\triangle\:A_n|}{|A_n|}=0$, for all $s\in F$.
\end{proof}

\vspace{1mm}

\begin{prop}\label{folnerinvariantconj} Let $B\lhd H$ be a normal subgroup so that $[H:C_H(B)]<\infty$. Then $[B: Z(B)]<\infty$. Moreover, there is a set $C$ of representatives for $B/Z(B)$ and a sequence of sets $B_n :=A_nC\subseteq B$, where $A_n\subset Z(B)$, so that $B_n$ is a F\o lner sequence of $B$ and 

\begin{equation}\label{weakcompactcondition}
    \lim_{n\to\infty} \frac{|B^g_n\:\triangle\: B_n|}{|B_n|}=0,\text{ for all }g\in H.
\end{equation}
\end{prop}

\begin{proof} Clearly $[B:Z(B)]<\infty$. As $B$ is normal in $H$, it follows that $C_H(B)$ is normal in $H$. Moreover, $Z(B)$ is normal in $H$.

\vspace{2mm}

Let $K=H/C_H(B)$ and consider a cocycle semidirect product decomposition $H = C_H(B)\rtimes_{\alpha, \beta} K$ for an action $K\curvearrowright^\alpha C_H(B)$ and a 2-cocycle $\beta:K\times K\rightarrow C_H(B)$ for $\alpha$. As $\alpha$ invaries $Z(B)$, we define an action $\hat{\alpha}:K\curvearrowright \text{Aut}(Z(B))$ by $\hat{\alpha}_g(x)=\alpha_g(x)$, for $x\in Z(B)$. Notice that

\begin{equation*}
    (\hat{\alpha}_g\circ\hat{\alpha}_h)(x)=(\alpha_g\circ\alpha_h)(x)=\text{Ad}(\beta(g,h))\alpha_{gh}(x)=\alpha_{gh}(x)=\hat{\alpha}_{gh}(x),
\end{equation*}

\vspace{2mm}

for all $x\in Z(B)$. Above we have used that $\beta(g,h)\in C_H(B)$. Thus, we can consider the semidirect product $Z(B)\rtimes_{\hat{\alpha}}K\supset Z(B)\rtimes_{\hat{\alpha}}C$, where $C$ is a set of coset representatives for $B/Z(B)$. By Lemma \ref{folnerextension}, since $Z(B)$ is abelian and $K$ finite, we can find a sequence $(A_n)_n\subset Z(B)$ such that

\begin{enumerate}
    \item[(1)] $\lim_{n}\frac{|bA_n\:\triangle\:A_n|}{|A_n|}=0$, for all $b\in Z(B)$; and
    \item[(2)] $\lim_{n}\frac{|A_n^k\:\triangle\:A_n|}{|A_n|}=0$, for all $k\in K$. 
\end{enumerate}

\vspace{1mm}

Let $B_n:=A_nC\subset B$ and notice (1) implies that for all $b\in Z(B)$,

\begin{equation*}
    \lim_{n\to\infty}\frac{|bB_n\:\triangle\:B_n|}{|B_n|}=\lim_{n\to\infty}\frac{\sum_{c\in C}|bA_n\:\triangle\:A_n|}{|A_nC|}=\lim_{n\to\infty}\frac{|bA_n\:\triangle\:A_n|}{|A_n|}=0.
\end{equation*}

\vspace{2mm}

For $d\in C$ we have

\begin{align*}
    \lim_{n\to\infty}\frac{|dB_n\:\triangle\:B_n|}{|B_n|}&=\lim_{n\to\infty}\frac{|dA_nC\:\triangle\:A_nC|}{|A_nC|}=\lim_{n\to\infty}\frac{|A_n^ddC\:\triangle\:A_nC|}{|A_n|\cdot|C|}\\
    &=\lim_{n\to\infty}\frac{\sum_{c\in C}|A_n^d\beta(d,d^{-1}c)\:\triangle\:A_n|}{|A_n|\cdot|C|}=\lim_{n\to\infty}\frac{1}{|C|}\sum_{c\in C}\frac{|A_n^d\beta(d,d^{-1}c)\:\triangle\:A_n|}{|A_n|}\\
    &\leq\lim_{n\to\infty}\frac{1}{|C|}\sum_{c\in C}\frac{|A_n^d\:\triangle\:A_n|+|A_n\beta(d,d^{-1}c)\:\triangle\:A_n|}{|A_n|}.
\end{align*}

\vspace{2mm}

Since $C$ is finite and seeing $C\subset K$, by (1) and (2), the last limit above equals zero. Thus $(B_n)_n\subset B$ is a F\o lner sequence for $B$. 

\vspace{2mm}

Now let $h\in H$ and write $h=k\hat{h}$ where $\hat{h}\in H/C_H(B)=K$. Then,

\begin{align*}
    \lim_{n\to\infty}\frac{|B_n^h\:\triangle\:B_n|}{|B_n|}&=\lim_{n\to\infty}\frac{|A_n^hC^h\:\triangle\:A_nC|}{|A_n|\cdot|C|}=\lim_{n\to\infty}\frac{|A_n^{\hat{h}}C^{\hat{h}}\:\triangle\:A_nC|}{|A_n|\cdot|C|}.
\end{align*}

\vspace{2mm}

Note that since the action by conjugation of $K$ invaries the subgroup $C$ then $C^{\hat{h}}=\{\gamma(\hat{h},c) c :c\in C\}$ for finitely many $\gamma(\hat{h},c)\in Z(B)$. Continuing from the prior equation, we have

\begin{align*}
    \lim_{n\to\infty}\frac{|B_n^h\:\triangle\:B_n|}{|B_n|}&=\lim_{n\to\infty}\sum_{c\in C}\frac{|A_n^{\hat{h}}\gamma(\hat{h},c)\:\triangle\:A_n|}{|A_n|\cdot|C|}\leq\lim_{n\to\infty}\sum_{c\in C}\frac{|A_n^{\hat{h}}\:\triangle\:A_n|+|A_n\gamma(\hat{h},c)\:\triangle\:A_n|}{|A_n|\cdot|C|}=0,
\end{align*}

\vspace{2mm}

by (1) and (2).
\end{proof}

\begin{prop}\label{weakweakaction} The action $H\curvearrowright\mathcal{L}(H^{fc})$ by conjugation is weakly compact relative to $H^{fc}$ with respect to the canonical trace in $\cL(H^{fc})$.
\end{prop}

\begin{proof} Using Lemma \ref{towergroups}, there is a sequence (possibly finite) of center-by-finite normal subgroups of $H$ such that $H_1\leqslant H_2\leqslant\cdots\leqslant H^{fc}$ with $\cup_nH_n=H^{fc}$. 

\vspace{2mm}

For each $n\in\N$, consider the center $Z(H_n)\leqslant H_n$ and note it is a normal in $H$ with $[H_n:Z(H_n)]<\infty$. Fix $C_n\subset H_n$ a (finite) set of representatives for the cosets $H_n/Z(H_n)$. 

\vspace{2mm}

Since $[H:C_H(H_n)]<\infty$, by Proposition \ref{folnerinvariantconj} one can find a sequence of finite subsets $F^n_l= A^n_l C_n\subseteq H_n$, with $A^n_l \subseteq Z(H_n)$, satisfying 

\vspace{1mm}

\begin{itemize}
    \item $\lim_l \frac{|(F_l^n)^g\triangle F^n_l|}{|F^n_l|}= 0$, for all $g\in H$ and 
    \item  $\lim_l \frac{|h(F_l^n)\triangle F^l_n|}{|F^n_l|}=0$, for all $h\in H_n$.
\end{itemize}

\vspace{1mm}

As $\cup_n H_n = H^{fc}$, the prior relations together with a standard diagonalization argument yield a sequence of subsets $E_n:=F^n_{k_n}= A^n_{k_n} C_n$, with $A^n_{k_n}\subseteq Z(H_n)$, such that \begin{enumerate}
    \item[(l)] $\lim_n \frac{|(E_n)^g\:\triangle\: E_n|}{|E_n|}= 0$, for all $g\in H$, and 
    \item[(ll)]  $\lim_n \frac{|h E_n\:\triangle\: E_n|}{|E_n|}=0$, for all $h\in H^{fc}$.
\end{enumerate}

\vspace{2mm}
 
Define

\begin{equation*}
    \eta_n=\frac{1}{|E_n|^{1/2}}\sum_{a\in E_n}u_a\otimes u_a\in\ell^2(H_n)\otimes\ell^2(H_n).
\end{equation*}

\vspace{2mm}

Next we show the sequence $\eta_n$ satisfies the conditions (j), (jj) and (jjj) for weak compactness as in Definition \ref{weakerweakcompactness}. For (j), let $g\in H$. Then, by (l)

\begin{equation*}
    \|(\sigma_g\otimes\sigma_g)(\eta_n)-\eta_n\|_2^2=\frac{1}{|E_n|}\left\|\sum_{a\in E_n}u_{gag^{-1}}\otimes u_{gag^{-1}}-\sum_{a\in E_n}u_a\otimes u_a\right\|_2^2=\frac{|(E_n)^{g}\:\triangle\: E_n|}{|E_n|}\to 0.
\end{equation*}

\vspace{2mm}

To show (jj), let $h\in H^{fc}$. Then,

\begin{equation*}
    \|(u_h\otimes u_h)\eta_n-\eta_n\|_2^2=\frac{1}{|E_n|}\left\|\sum_{a\in E_n}u_{ha}\otimes u_{ha}-\sum_{a\in E_n}u_a\otimes u_a\right\|_2^2=\frac{|hE_n\:\triangle\: E_n|}{|E_n|}\to 0,
\end{equation*}

\vspace{2mm}

by (ll). Finally, for $x\in\cL(H^{fc})$, we have

\begin{equation*}
    \langle(x\otimes 1)\eta_n,\eta_n\rangle=\frac{1}{|E_n|}\sum_{a,b\in E_n}\langle(x\otimes 1)(u_a\otimes u_a),u_b\otimes u_b\rangle=\frac{1}{|E_n|}\sum_{a\in E_n}\langle xu_a,u_a\rangle=\tau(x).
\end{equation*}

\vspace{2mm}

Thus, $(\eta_n)_n$ witnesses the weak compactness for the conjugation action $H\curvearrowright \cL(H^{fc})$ relative to $H^{fc}$.
\end{proof}

\begin{rem}\label{Hfcfolnersetsvirtab} In the prior proposition, if one assumes $H^{fc}$ is virtually abelian with abelian finite index subgroup $A$, then the sets $\{E_n\}_n$ can be chosen in such a way that $E_n=K_nC$ with $K_n\subset A$.
\end{rem}

We finish with the main theorem of this section. We combine Propositions \ref{folnerinvariantconj} and \ref{weakweakaction} to show the action of $H$ on the group von Neumann algebra associated to $H^{hfc}$ is weakly compact in the sense of Definition \ref{weakerweakcompactness}.

\begin{thm} The action $H\curvearrowright \cL(H^{hfc})$ by conjugation is weakly compact relative to $H^{hfc}$ with respect to the canonical trace in $\cL(H^{hfc})$.
\end{thm}

\begin{proof}  


Assume, by induction, the action by conjugation $H\curvearrowright H_m^h$ admits a Fölner sequence satisfying \eqref{weakcompactcondition}. We show the same holds for $H\curvearrowright H_{m+1}^h$. Observe that $H_{m+1}^h=H_m^h\rtimes_{\alpha,c}H_{m+1}^h/H_m^h$ and $H=H_m^h\rtimes_{\beta,d}H/H_m^h$.

\vspace{2mm}

From Proposition \ref{folnerinvariantconj} and the first part of the proof of Proposition \ref{weakweakaction}, we know the action $H/H_{m}^h\curvearrowright H_{m+1}^h/H_m^h$ also admits a Fölner sequence satisfying \eqref{weakcompactcondition}. That is, there exists $\{E_n\}_n\subset H_m^h$, $\{F_n\}_n\subset H_{m+1}^h/H_m^h$ satisfying 

\begin{enumerate}
    \item[(l)] $\lim_{n\to\infty}\frac{|aE_n\triangle E_n|}{|E_n|}=0$, for all $a\in H_m^h$,
    \item[(ll)] $\lim_{n\to\infty}\frac{|gE_ng^{-1}\triangle E_n|}{|E_n|}=0$, for all $g\in H$; and
    \item[(lll)] $\lim_{n\to\infty}\frac{|aF_n\triangle F_n|}{|F_n|}=0$, for all $a\in H_{m+1}^h/H_m^h$,
    \item[(lv)] $\lim_{n\to\infty}\frac{|gF_ng^{-1}\triangle F_n|}{|F_n|}=0$, for all $g\in H/H_m$. 
\end{enumerate}

\vspace{2mm}

Let $x=ab\in H_{m+1}^h$ where $a\in H_m^h$ and $b\in H_{m+1}^h/H_m^h$. Then,

\begin{align*}
    |xE_{k_n}F_n&\:\triangle\: E_{k_n}F_n|=|a(bE_{k_n}b^{-1})(b\cdot F_n)\:\triangle\: E_{k_n}F_n|\\
    &\leq |a(bE_{k_n}b^{-1})(b\cdot F_n)\:\triangle\: aE_{k_n}(b\cdot F_n)|+|aE_{k_n}(b\cdot F_n)\:\triangle\: E_{k_n}(b\cdot F_n)|\\
    &\quad\quad\quad+|E_{k_n}(b\cdot F_n)\:\triangle\: E_{k_n}F_n|\\
    &\leq |bE_{k_n}b^{-1}\:\triangle\: E_{k_n}|\cdot|F_n|+|aE_{k_n}\:\triangle\: E_{k_n}|\cdot |F_n|+|E_{k_n}(b\cdot F_n)\:\triangle\: E_{k_n}F_n|.
\end{align*}

\vspace{2mm}

Consider the third term of the prior equation. Notice that $b\cdot F_n=\{c(b,h)(bh):h\in F_n\}$. Then,

\begin{align*}
    |E_{k_n}(b\cdot F_n)\:\triangle\:E_{k_n}F_n|&=\left|\left(\cupm_{h\in F_n}E_{k_n}c(b,h)(bh)\right)\:\triangle\:E_{k_n}F_n\right|\\
    &\leq \left|\left(\cupm_{h\in F_n}E_{k_n}c(b,h)(bh)\right)\:\triangle\:E_{k_n}(bF_n)\right|+|E_{k_n}(bF_n)\:\triangle\:E_{k_n}F_n|\\
    &\leq \sum_{h\in F_n}|E_{k_n}c(b,h)\:\triangle\: E_{k_n}|+|E_{k_n}|\cdot|(bF_n)\:\triangle\:F_n|.
\end{align*}

\vspace{2mm}

Combining the previous displayed inequalities

\begin{align*}
    \frac{|xE_{k_n}F_n\:\triangle\:E_{k_n}F_n|}{|E_{k_n}F_n|}&\leq \frac{|bE_{k_n}b^{-1}\:\triangle\:E_{k_n}|}{|E_{k_n}|}+\frac{|aE_{k_n}\:\triangle\:E_{k_n}|}{|E_{k_n}|}+\sum_{h\in F_n}\frac{|E_{k_n}c(b,h)\:\triangle\: E_{k_n}|}{|E_{k_n}|\cdot |F_n|}+\frac{|(bF_n)\:\triangle\: F_n|}{|F_n|}.
\end{align*}

\vspace{2mm}

Take $(k_n)_n$ a growing sequence such that by (l) we have for every $\epsilon>0$ and for every $b\in H_{m+1}^h/H_m^h$, $h\in F_n$ there exists $k_n$ large enough with $\frac{|E_{k_n}c(b,h)\:\triangle\:E_{k_n}|}{|E_{k_n}|}<\epsilon$. Therefore, 

\begin{equation*}
    \lim_{n\to\infty}\frac{|xE_{k_n}F_n\:\triangle\:E_{k_n}F_n|}{|E_{k_n}F_n|}=0,\text{ for all }x\in H_{m+1}^h.
\end{equation*}

\vspace{2mm}

Now, let $g\in H$. Then,

\begin{align*}
    |gE_{k_n}F_ng^{-1}\:\triangle\:E_{k_n}F_n|&\leq |(gE_{k_n}g^{-1})(g\cdot F_n\cdot g^{-1})\:\triangle\: E_{k_n}(g\cdot F_n\cdot g^{-1})|\\
    &\quad\quad+|E_{k_n}(g\cdot F_n\cdot g^{-1})\:\triangle\:E_{k_n}F_n|\\
    &=|gE_{k_n}g^{-1}\:\triangle\: E_{k_n}|\cdot|F_n|+|E_{k_n}(g\cdot F_n\cdot g^{-1})\:\triangle\:E_{k_n}F_n|.
\end{align*}

\vspace{2mm}

We consider the second term. Let $g=rl$ where $r\in H_m^h$ and $l\in H/H_m^h$. Then, since $F_n\subset H_{m+1}^h/H_m^h\subset H/H_m^h$ we have $g\cdot F_n\cdot g^{-1}=\{rd(l,h)d(lh,l^{-1})\beta_{lhl^{-1}}(r^{-1})(lhl^{-1}):h\in F_n\}=:\{a_{g,h}(lhl^{-1}):h\in F_n\}$, where $a_{g,h}\in H_m^h$. In this case,

\begin{align*}
    |E_{k_n}(g\cdot F_n\cdot g^{-1})\:\triangle\:&E_{k_n}F_n|=\left|\left(\cupm_{h\in F_n}E_{k_n}a_{g,h}(lhl^{-1})\right)\:\triangle\:E_{k_n}F_n\right|\\
    &\leq \left|\left(\cupm_{h\in F_n}E_{k_n}a_{g,h}(lhl^{-1})\right)\:\triangle\:E_{k_n}(lF_nl^{-1})\right|+|E_{k_n}(lF_nl^{-1})\:\triangle\:E_{k_n}F_n|\\
    &\leq\sum_{h\in F_n}|E_{k_n}a_{g,h}\:\triangle\:E_{k_n}|+|E_{k_n}|\cdot|lF_nl^{-1}\:\triangle\:F_n|.
\end{align*}

\vspace{2mm}

Combining the prior two inequalities and similarly as earlier, we obtain

\begin{equation*}
    \lim_{n\to\infty}\frac{|gE_{k_n}F_ng^{-1}\:\triangle\:E_{k_n}F_n|}{|E_{k_n}F_n|}=0,\text{ for all }g\in H.
\end{equation*}

\vspace{2mm}

This means $\{E_{k_n}F_n\}_n\subset H_{m+1}^h$ is the desired sequence. 

\vspace{2mm}

By induction, this is true for any $H\curvearrowright H_k^h$, $k\in\N$. The Fölner sequence $\{K_n\}_n$ for $H\curvearrowright H^{hfc}$ is the limit of the product of the Fölner sequences at each step. Following Proposition \ref{weakweakaction}, we see that the sequence

\begin{equation*}
    \eta_n=\frac{1}{|K_n|^{1/2}}\sum_{a\in K_n}u_a\otimes u_a,
\end{equation*}

\vspace{2mm}

witnesses the weak compactness of $H\curvearrowright \cL(H^{hfc})$ in our sense.
\end{proof}

\begin{rem}\label{Hhfcfolnersetsvirab} Assume $H^{hfc}$ is virtually abelian with $H^{hfc}=A\rtimes F$, where $A$ is the finite index abelian subgroup. In that case, $H_n^h$ is virtually abelian for all $n$ with $A_n:=A\cap H_n^h$ as the finite index abelian subgroup. Note that, $H_{n+1}^h/H_n^h=A_{n+1}/A_n\rtimes F_{n+1}$ is also virtually abelian with finite index abelian subgroup $A_{n+1}/A_n$. Now, since $F\subset H^{hfc}$ is finite, there exists $k\in\N$ such that $F\subset H_k^h$. So, by the prior proposition and Remark \ref{Hfcfolnersetsvirtab}, we can choose the Fölner sets for $H_k^h$ of the form $(E_{l_n^k}F_1)(E_{l_n^{k-1}}F_2)...(E_{l_n^2}F_{k-1})(E_nF)$, or $R_nF$ with $R_n$ a finite subset of $A$ after moving the $F_i$'s to the right by conjugation. For $s>k$, we may choose the sets of the form $R_{l_n^s}FE_{l_n^{s-k}}...E_n$. Similarly, the Fölner sets for $H\curvearrowright H^{hfc}$ are the limit of the product of the Fölner sets at each step.
\end{rem}

\section{Techniques for reconstructing the group center under W\texorpdfstring{$^*$}{*}-equivalence}\label{section5}

In this section we prove a version of \ref{theoremA} concerning wreath-like product groups whose natural 2-cocycles have uniformly bounded support length. Our techniques build on some of the prior methods \cite{popa06,i07,popaozawa10,si11,ct13,amcos23}. To properly introduce our proof we need some preliminaries.

\subsection{Deformations associated with arrays on groups.} Arrays were introduced  in \cite{ct13} as a way to put together length functions and 1-cocycles. In practice arrays can be used either to strengthen the concept of a length function by introducing a representation or to introduce some ``geometric'' flexibility to the concept of a 1-cocycle. We recall the construction of a deformation from an array based on \cite{si11}.

\begin{defn} Let $G$ be a countable discrete group and $\pi:G\to\mathscr{O}(\mathcal{H}_{\pi})$ an orthogonal representation on a real Hilbert space $\mathcal{H}_{\pi}$. We say $G$ admits an \textit{array} into $\mathcal{H}_{\pi}$ if there exists a map $q:G\to \mathcal{H}_{\pi}$ such that for every finite subset $F$ of $G$ there exists $K\geq 0$ with $\|\pi(g)(q(h))-q(gh)\|\leq K$, for all $g\in F$, $h\in G$. 
\end{defn}

Let $\pi:G\to\mathcal{O}(\mathcal{H}_{\pi})$ be an orthogonal representation. The Gaussian construction described in \cite{ps12} provides an abelian von Neumann algebra $(\cE,\tau)$ generated by a family of unitaries $\omega(\xi)$ where the following relations hold:

\begin{equation*}
    \omega(0)=1,\quad \omega(\xi_1)\omega(\xi_2)=\omega(\xi_1+\xi_2) \quad\omega(-\xi)=\omega(\xi)^*,\quad \tau(\omega(\xi))=\exp{(-\|\xi\|^2)},
\end{equation*}

\vspace{2mm}

for all $\xi,\xi_1,\xi_2\in \mathcal{H}_{\pi}$. The \textit{Gaussian action} of $G$ on $(\cE,\tau)$ is defined by $\tilde{\pi}(g)(\omega(\xi))=\omega(\pi(g)\xi)$ for $g\in G$ and $\xi\in \mathcal{H}_{\pi}$. Let $G\curvearrowright^{\sigma}(\cN,\tau)$ be a trace-preserving action of $G$ on a finite von Neumann algebra $\mathcal{N}$ and let $\cM=\cN\rtimes_{\sigma}G$ be the corresponding cross product von Neumann algebra. The \textit{Gaussian dilation} associated to $\cM$ is the von Neumann algebra $\tilde{\cM}=(\mathcal{N}\:\overline{\otimes}\:\cE)\rtimes_{\sigma\otimes\tilde{\pi}}G$. 

\vspace{2mm}

Let $q:G\to \mathcal{H}_{\pi}$ be an array for the representation $G$ as above. The deformation is constructed as follows. For each $t\in\R$, define a unitary $V_t\in\mathscr{U}(L^2(\cN)\otimes L^2(\cE)\otimes\ell^2(G))$ by

\begin{equation*}
    V_t(n\otimes e\otimes \delta_g):=n\otimes\omega(tq(g))e\otimes\delta_g,
\end{equation*}

\vspace{2mm}

for $n\in L^2(\cN)$, $e\in L^2(\cE)$, $g\in G$. In \cite{ct13} it was proved that $V_t$ is a strongly continuous one-parameter group of unitaries. Next we specialize this to the case of von Neumann algebras associated with direct sums of infinitely many groups.

\vspace{2mm}

Let $W=C^{(I)}\rtimes D$ be a wreath-like product and let $\pi:W\to B(\ell^2(I))$ be the orthogonal representation given by $\pi(g)\delta_d=\delta_{\epsilon(g)d}$, where $\ell^2(I)$ is seen as a real Hilbert space. Just as in the definition, let $(\cE,\tau)$ be the abelian von Neumann algebra generated by unitaries $\omega(\xi)$ with $\xi\in\ell^2(D)$ and $\tilde{\pi}$ the Gaussian action of $W$ on $\cE$. The map $q:W\to \ell^2(I)$ defined such that for $g=(c,\tilde{d})\in W$ we have

\begin{equation*}
    q(g)d=\begin{cases}
    1&\text{if }d\in supp(c)\\
    0&\text{otherwise,}
    \end{cases}
\end{equation*}

\vspace{2mm}

is an array by \cite{amcos23}. Let $W$ act trivially on a finite von Neumann algebra $\cL(A)$ and consider the cross product $\cM=\cL(A)\rtimes W$. Let $\tilde{\cM}=(\cL(A)\otimes \cE)\rtimes_{\tilde{\pi}}W$ be the Gaussian dilation associated to $\cM$. In this case, the deformation $V_t^q\in \mathscr{U}(\ell^2(A)\otimes L^2(\cE)\otimes \ell^2(W))$ is given by $V^q_t(\delta_a\otimes e\otimes  \delta_g)=\delta_a\otimes \omega(tq(g))e\otimes \delta_g$. 

\vspace{2mm}

Using \cite[Theorem 3.2]{ct13} (lines 18-48) one can show that $V_t^q$ converges uniformly to the identity on $\cL(H_n)$. Otherwise, taking $\cN:=\cL(H_n)'\cap \cM$, we would obtain that $\cN$ is amenable, but that would imply that $C_{H}(H^{fc})$ is amenable, as $\cL(C_{H}(H_n))\subset \cN$. This is a contradiction because $C_{H}(H_n)$ is a finite index subgroup of $H$. Using the transversality property shown in Lemma 2.8 of the same paper, we obtain that $e_{\cM}^{\perp}V_t^q\to 0$ uniformly on $(\cL(H_n))_1$, where $e_{\cM}$ denotes the orthogonal projection of $L^2(\tilde{M})$ onto $L^2(M)$.

\vspace{2mm}

\subsection{Deformations for infinite tensor products of von Neumann algebras.} The deformation described below was first used by Ioana in \cite{i07} to obtain conjugacy results for rigid subalgebras of wreath products.

\vspace{2mm}

Given a finite von Neumann algebra $(\mathcal{B},\tau)$, let $\tilde{\mathcal{B}}=\mathcal{B}\ast \cL(\Z)$. If $u\in\mathscr{U}(\cL(\Z))$ is a generating Haar unitary, choose a self-adjoint $h\in\cL(\Z)$ such that $u=\exp(\pi ih)$, and let $u^t=\exp(\pi ith)$. Define the deformation $\alpha:\R\to \text{Aut}(\tilde{\mathcal{B}},\tilde{\tau})$ by $\alpha_t=\text{Ad}(u^t)$. Similarly, if we have an infinite family of tracial von Neumann algebras $(\cN_n,\tau_n)_{n\in\N}$ we can denote by $\cN=\overline{\otimes}_{n\in\N}\cN_n$ and define $\tilde{\cN}_n=\cN_n\ast\cL(\Z)$ with respect to the natural traces. Taking Ioana's deformation $\alpha_t^n$ associated to each $\cN_n$, we can consider the path of automorphisms $\alpha_t=\otimes_{n\in\N}\alpha_t^n\in \text{Aut}(\tilde{\mathcal{N}})$. 

\vspace{2mm}

For von Neumann algebras associated with direct sums of infinitely many groups, let $W=C^{(I)}\rtimes D$ be a wreath-like product group and $A$ an abelian group. Define $\mathcal{N}=\mathcal{L}(A\times C^{(I)})$ and $\tilde{\cN}=\mathcal{L}(A\times (C\ast \Z)^{(I)})=\cL(A)\overline{\otimes}\cL((C\ast\Z)^{(I)})$ and consider Ioana's deformation $\alpha_t:\mathcal{L}(A\times C^{(I)})\to\mathcal{L}(A\times(C\ast\Z)^{(I)})$ as described above. 

\vspace{2mm}

Although the array deformation and Ioana's deformation seem, a priori, quite different, there is a way to relate them on the intersection of their domains. We recall the following subordination property established in \cite{amcos23}.

\begin{thm}\label{juanfelipe}\cite[Theorem 5.6]{amcos23} Let $\cM= \cL(A\times W)$. Then for every $\xi\in \ell^2(A\times C^{(I)})$ we have the following inequalities:

\begin{equation}
    \|e_{\cM}^{\perp}\circ V^q_{\sqrt{\frac{2}{3}} t}(\xi)\|_2\geq \| e_{\cM}^{\perp}\circ \alpha_t(\xi)\|_2\geq\|e_{\cM}^{\perp}\circ V^q_{\sqrt{\frac{1}{3}} t}(\xi)\|_2.
\end{equation}
\end{thm}

\vspace{1mm}

\subsection{Reconstruction of the group center} In this section we show that, up to a finite normal subgroup, the FC-center $H^{fc}$ of a group $H$ satisfying $\cL(H)=\cL(A\times W)$ is abelian.

\vspace{2mm}

The following result can be found in the proof of \cite[Theorem C]{cik16} but we include it here for completeness.

\begin{lem} Let $G$ and $D$ be groups, with $D$ non-elementary hyperbolic, and let $\pi:G\to D$ be an epimorphism. Let $G\curvearrowright (X,\mu)$ be a free ergodic p.m.p. action and let $\cM=L^{\infty}(X)\rtimes G$. If $\cA\subset \cM$ is any regular amenable von Neumann subalgebra, then $\cA\prec L^{\infty}(X)\rtimes \ker(\pi)$. In particular, for any $\cA\subset\cL(G)$ amenable regular von Neumann subalgebra, we have $\cA\prec\cL(\ker\pi)$.
\end{lem}

\begin{proof}\label{regularamenable} Let $\{u_g\}_{g\in G}\subset \cL(G)$ and $\{v_h\}_{h\in D}\subset \cL(D)$ be the canonical unitaries. Consider the map $\Theta:\cM\to \cM\overline{\otimes} \cL(D)$ given by $\Theta(nu_g)=nu_g\otimes v_{\pi(g)}$, where $g\in G$, $n\in L^{\infty}(X)$. We identify $\cM\otimes \cL(D)=\cM\rtimes D$, where $D$ acts trivially on $\cM$. By \cite[Theorem 1.1]{popavaes12}, $\Theta(\cM)\subset\mathcal{N}_{\cM\otimes \cL(D)}(\Theta(\cA))''$. By \cite[Theorem 1.6]{popavaes11}, it follows that either $\Theta(\cA)\prec_{\cM\otimes \cL(D)}\cM\otimes 1$ or, $\Theta(\cM)$ is amenable relative to $\cM\otimes 1\subset \cM\otimes \cL(D)$. If the second holds, \cite[Proposition 3.5]{cik16} implies that $\pi(G)=D$ is amenable, contradicting our assumptions. Hence, $\Theta(\cA)\prec_{\cM\otimes \cL(D)}\cM\otimes 1$ and using \cite[Proposition 3.4]{cik16}  we get $\cA\prec_{\cM}L^{\infty}(X)\rtimes\pi^{-1}(\{e_D\})=L^{\infty}(X)\rtimes \ker(\pi)$. 

\vspace{2mm}

The last part follows from the first, by considering the trivial action $G\curvearrowright \{x\}$, on a singleton. 
\end{proof}

\begin{thm}\label{theoremone} Let $W\in \mathcal{WR}_b(C,D\curvearrowright I)$ where $C$ is abelian, $D$ is an ICC subgroup of a hyperbolic group, and the action $D\curvearrowright I$ has amenable stabilizers. Let $A$ be an infinite abelian group and denote by $G=W\times A$. Assume that $H$ is an arbitrary group such that $\Psi:\cL(G)\to \cL(H)$ is a $*$-isomorphism that preserves the canonical trace on $\cL(G)$ and the canonical trace on $\cL(A)$. Then, the hyper-FC center $H^{hfc}$ is virtually abelian and $\Psi (\cL(H^{hfc}))\prec^s \mathscr Z(\cL(G))$. 
\end{thm}

\begin{proof} Denote $\cM=\cL(G)$ and identify $\cM$ with $\cL(H)$ under the image of the $*$-isomorphism $\Psi$. Fix $z\in\mathscr{Z}(\cM)$. From definitions it follows that $\mathscr Z(\cM)=\mathscr Z(\cL(H))\subseteq \cL(H^{hfc})$. Moreover, since $W$ is ICC and $A$ is abelian one can easily see that $\mathscr Z(\cM)=\cL(A)$. Next we prove the following:

\begin{claim}\label{step1} $\cL(H^{hfc})\prec^s_\cM \cL(A\times C^{(I)})$.\end{claim}

\begin{subproof}[Proof of Claim \ref{step1}] Since $W$ surjects into $D$, so does $G= W\times A$. Using the normality of the group $H^{hfc}$ in $H$, we have that the normalizer satisfies $\mathscr N_{\cM}(\cL(H^{hfc}))''= \cM$, and since $H^{hfc}$ is amenable by \cite{le68}, so is $\cL(H^{hfc})$. Furthermore, since $D$ is an ICC subgroup of  hyperbolic group, Lemma \ref{regularamenable} implies that 

\begin{equation}\label{intabelian}
    \cL(H^{hfc})\prec_\cM \cL(A\times C^{(I)}).
\end{equation} 

\vspace{2mm}

For a central projection $z\in\mathscr{Z}(\cM)$ we have that $\mathscr{N}_{\cM z}(\cL(H^{hfc})z)=\cM z$ and by the prior argument we obtain $\cL(H^{hfc})z\prec_\cM \cL(A\times C^{(I)})$; i.e. $\cL(H^{hfc})\prec_{\cM}^s\cL(A\times C^{(I)})$.
\end{subproof}

Since $\cL(H^{hfc})$ is finite it admits unique direct sum decomposition 

\begin{equation*}
    \cL(H^{hfc})=\cL(H^{hfc})z_1\oplus \cL(H^{hfc})z_2
\end{equation*}

\vspace{2mm}

where $\cL(H^{hfc})z_1$ is type I,  $\cL(H^{hfc})z_2$ is type II, and $z_1,z_2\in \mathscr Z(\cL(H^{hfc}))$ are projections with $z_1+z_2=1$. Since each $h\in H$ normalizes $\cL(H^{hfc})$, conjugation by $u_h$ induces an automorphism of $\cL(H^{hfc})$. As automorphisms preserve component type, it follows that $H$ normalizes $\cL(H^{hfc})z_2$. Next we show $z_2=0$. Assume by contradiction otherwise. The prior claim shows that $\cL(H^{hfc})z_2\prec_\cM \cL(A\times C^{(I)})$. Since both $A$ and $C$ are abelian, $\cL(A\times C^{(I)})$ is abelian. Using \eqref{intabelian} and Theorem \ref{intertwining}, one can find a nontrivial corner of $\cL(H^{hfc})z_2$ that is abelian, contradicting that $z_2$ is a continuous projection. Thus $z_2=0$, and hence $\cL(H^{hfc})=\cL(H^{hfc})z_1$ is a type I von Neumann algebra. Therefore \cite{ka69} implies that $H^{hfc}$ is virtually abelian. That is, $H^{hfc}=H_0\rtimes F$, where $H_0\leqslant H^{hfc}$ is a finite index abelian subgroup.

\vspace{3mm}

By Proposition \ref{weakweakaction} the vectors $\eta_n:=|E_n|^{-1/2}\sum_{a\in E_n}u_{a}\:\otimes\: u_{a}\in \ell^2(H^{hfc})\:\otimes\:\ell^2(H^{hfc})$ satisfy the weaker conditions for weak compactness defined in Definition \ref{weakerweakcompactness}. From Remark \ref{Hhfcfolnersetsvirab} we can choose the sets $E_n=R_{n}F$ where $R_{n}\subset H_0$. Denote by $\xi_{n,t}:=(e_{\cM}^{\perp}\otimes 1)((V_t^q\otimes 1)\eta_n(z\otimes z))$.

\vspace{1mm}

\begin{claim}\label{above} $\lim\limits_{n\to\infty}\|\xi_{n,t}\|_2=0$.
\end{claim}

\begin{subproof}[Proof of Claim \ref{above}] It follows since if $\lim_{n\to\infty}\|\xi_{n,t}\|_2\geq c>0$ for some $c>0$, \cite[Lemmas 4.4 and 4.5]{ct13} (which only use (j) and (jj)) apply and \cite[Theorem B]{popaozawaII} together with Haagerup's criterion gives us that $\mathscr{N}_{\cM}(\cL(H^{hfc}))''=\cM$ is amenable, a contradiction.
\end{subproof}

Using Claim \ref{above} and property (j) we have that

\begin{equation*}
    \lim_{n\to\infty}\|(e_{\cM}^{\perp}\otimes 1)(V_t^q\otimes 1)((u_h\otimes u_h)\eta_n(z\otimes z))\|_2=0\quad\text{for all }h\in H^{hfc}.
\end{equation*}

\vspace{2mm}
Note that the previous limit is equivalent to 
\begin{equation}\label{limit1}
    \lim_{n\to\infty}\frac{1}{|E_n|}\sum_{a}\|(e_{\cM}^{\perp}\circ V_t^q)(u_{ha}z)\|_2^2=0\quad\text{for all }h\in H^{hfc}.
\end{equation}

By Claim \ref{step1} there exist non-zero projections $p\in \cL(H^{hfc})z$, $q\in \cL(A\times C^{(I)})$, a non-zero partial isometry $v\in q\cM p$ and a $*$-isomorphism onto its image $\Theta:p\cL(H^{hfc})p\to q\cL(A\times C^{(I)})q$ with $\Theta(x)vv^*=vxv^*$ for all $x\in p\cL(H^{hfc})p$. By the Kaplanski density theorem, there exists $v_{\epsilon}=\sum_F\beta_gw_g$ for a finite subset $F\subset G$ with $\|v_{\epsilon}\|_{\infty}\leq 1$ and $\|v-v_{\epsilon}\|_2<\frac{\epsilon^{1/2}}{4}$. Therefore, for all $h\in H^{hfc}$, 

\begin{align}\label{conjugationconvergence}
    \|(e_{\cM}^{\perp}\circ V_t^q)(vu_{ha}zv^*)\|_2&\leq \frac{\epsilon^{1/2}}{2}+\|(e_{\cM}^{\perp}\circ V_t^q)(v_{\epsilon}u_{ha}zv_{\epsilon}^*)\|_2\nonumber\\
    &\leq \frac{\epsilon^{1/2}}{2}+\sum_{l,g\in F}|\beta_g|\cdot|\beta_{l^{-1}}|\cdot\|(e_{\cM}^{\perp}\circ V_t^q)(w_gu_{ha}zw_l)\|_2\\
    &\leq \frac{\epsilon^{1/2}}{2}+\sum_{l,g\in F}\|(e_{\cM}^{\perp}\circ V_t^q)(w_gu_{ha}zw_l)\|_2\nonumber
\end{align}

\vspace{2mm}

because $|\beta_g|\leq 1$ for all $g\in F$. Using the bimodularity condition \cite[Proposition 1.10]{csu16} there exists $t_{\epsilon}>0$ for which $\|V_t^q(w_gu_{ha}zw_l)-w_gV_t^q(u_{ha}z)w_l\|_2<\frac{\epsilon^{1/2}}{4|F|^2}$ for all $t<t_{\epsilon}$. Thus, for all $h\in H^{hfc}$ and $t<t_{\epsilon}$, and by the bimodularity of $e_{\cM}^{\perp}$, we have

\begin{align*}
    \|(e_{\cM}^{\perp}\circ V_t^q)(w_gu_{ha}zw_l)\|_2&\leq \|(e_{\cM}^{\perp}\circ V_t^q)(w_gu_{ha}zw_l)-w_g(e_{\cM}^{\perp}\circ V_t^q)(u_{ha}z)w_l\|_2 +\|w_g(e_{\cM}^{\perp}\circ V_t^q)(u_{ha}z)w_l\|_2\\
    &\leq \|V_t^q(w_gu_{ha}zw_l)-w_gV_t^q(u_{ha}z)w_l\|_2+\|(e_{\cM}^{\perp}\circ V_t^q)(u_{ha}z)\|_2\\
    &\leq \frac{\epsilon^{1/2}}{4|F|^2}+\|(e_{\cM}^{\perp}\circ V_t^q)(u_{ha}z)\|_2.
\end{align*}

\vspace{2mm}

By \eqref{conjugationconvergence} and the previous estimate, we have $\|(e_{\cM}^{\perp}\circ V_t^q)(vu_{ha}zv^*)\|_2\leq \epsilon^{1/2}+|F|^2\|(e_{\cM}^{\perp}\circ V_t^q)(u_{ha}z)\|_2$ and hence,

\begin{equation*}
    \lim_{n\to\infty}\frac{1}{|E_n|}\sum_{a}\|(e^{\perp}_{\cM}\circ V_t^q)(vu_{ha}zv^*)\|_2^2\leq \epsilon
\end{equation*}

\vspace{2mm}

for all $t<t_{\epsilon}$ and all $h\in H^{hfc}$. Using the fact that $vxv^*=\Theta(x)vv^*$ for all $x\in p\cL(H^{hfc})p$, we obtain

\begin{equation*}
    \lim_{n\to\infty}\frac{1}{|E_n|}\sum_{a}\|(e_{\cM}^{\perp}\circ V_t^q)(\Theta(pu_{ha}p)vv^*)\|_2^2\leq \epsilon
\end{equation*}

\vspace{2mm}

for all $t<t_{\epsilon}$ and all $h\in H^{hfc}$. Working with $H_0$ instead of $H^{hfc}$, we may assume $p\in\cL(H_0)$ so that $u_hp=pu_h$ for all $h\in H_0$. Then,


\begin{equation*}
    \lim_{n\to\infty}\frac{1}{|E_n|}\sum_{a}\|(e_{\cM}^{\perp}\circ V_t^q)(\Theta(u_hpu_ap)vv^*)\|_2^2\leq \epsilon
\end{equation*}

\vspace{2mm}

for all $t<t_{\epsilon}$ and all $h\in H_0$. Let $r=vv^*\leq q$ and approximate $r$ with $r_{\epsilon}=\sum_{s\in K}x_sw_s$ for $K\subset D$ finite and $x_s\in \cL(A\times C^{(I)})$ such that $\|r-r_{\epsilon}\|_2<\epsilon\|vv^*\|_2$. Then,

\begin{align*}
    \lim_{n\to\infty}\frac{1}{|E_n|}\sum_{a}\|(e_{\cM}^{\perp}\circ V_t^q)(\Theta(u_{h}pu_{a}p)r_{\epsilon})\|_2^2\leq \epsilon^2(\|vv^*\|_2+1)^2
\end{align*}

\vspace{2mm}

by making $\epsilon$ smaller if necessary. By the bimodularity of $V_t^q$ again and by modifying $t,\epsilon$ accordingly,

\begin{align*}
    \lim_{n\to\infty}\frac{1}{|E_n|}\sum_{a}\left\|\sum_{s\in K}(e_{\cM}^{\perp}\circ V_t^q)(\Theta(u_hpu_{a}p)x_s)w_s\right\|_2^2\leq \epsilon^2(\|vv^*\|_2+1)^2.
\end{align*}

\vspace{2mm}

From the definition of $V_t^q$, since $\Theta(y)x_s\in \cL(A\times C^{(I)})$, we have that $\{(e_{\cM}^{\perp}\circ V_t^q)(\Theta(y)x_s)w_s\}_{s\in K}$ are orthogonal with respect to $\langle\cdot,\cdot\rangle_2$. Hence,

\begin{align*}
    \lim_{n\to\infty}\frac{1}{|E_n|}\sum_{s\in K}\sum_{a}\|(e_{\cM}^{\perp}\circ V_t^q)(\Theta(u_{h}pu_ap)x_s)\|_2^2\leq \epsilon^2(\|vv^*\|_2+1)^2.
\end{align*}

\vspace{2mm}

Using Theorem \ref{juanfelipe} and by making $t$ smaller if necessary, we have

\begin{equation*}
    \lim_{n\to\infty}\frac{1}{|E_n|}\sum_{s\in K}\sum_{a}\|(e_{\cM}^{\perp}\circ \alpha_t)(\Theta(u_hpu_{a}p)x_s)\|_2^2\leq \epsilon^2(\|vv^*\|_2+1)^2
\end{equation*}

\vspace{2mm}

for all $h\in H_0$ and all $t$ small enough. By the transversality property in \cite[Lemma 4]{ci10} and the fact that $\Theta$ is a $\ast$-homomorphism, we obtain

\begin{equation}\label{limit3}
    \lim_{n\to\infty}\frac{1}{|E_n|}\sum_{s,a}\|\Theta(u_h p)\Theta(pu_ap)x_s-\alpha_t(\Theta(u_h p)\Theta(pu_ap)x_s)\|_2^2\leq \epsilon^2(\|vv^*\|_2+1)^2
\end{equation}

\vspace{2mm}

for all $h\in H_0$. By letting $h=e$ on \eqref{limit3} and since for any $h$, $\Theta(u_hp)$ is a unitary in $q\cL(A\times C^{(I)})q$, we have

\begin{align}\label{limit4}
    \lim_{n\to\infty}\frac{1}{|E_n|}&\sum_{a,s}\|\Theta(u_h p)\Theta(pu_a p)x_s-\Theta(u_hp)\alpha_t(\Theta(pu_ap)x_s)\|_2^2\nonumber\\
    &=\lim_{n\to\infty}\frac{1}{|E_n|}\sum_{a,s}\|\Theta(pu_ap)x_s-\alpha_t(\Theta(pu_ap)x_s)\|_2^2\leq \epsilon^2(\|vv^*\|_2+1)^2.
\end{align}

\vspace{2mm}

Using \eqref{limit3} and \eqref{limit4} we get

\begin{align*}
   \lim_{n\to\infty}\frac{1}{|E_n|}&\sum_{a,s}\|(\Theta(u_hp)-\alpha_t(\Theta(u_hp)))\alpha_t(\Theta(pu_ap)x_s)\|_2^2\leq 2\epsilon^2(\|vv^*\|_2+1)^2
\end{align*}

\vspace{2mm}

for all $h\in H_0$. Since $\Theta(pu_ap)vv^*=vv^*\Theta(pu_ap)$, we have $\sum_{s\in K}\Theta(pu_ap)x_sw_s=\sum_{s\in K}x_sw_s\Theta(pu_ap)$ and so $\Theta(pu_ap)x_s=x_s\sigma_s(\Theta(pu_ap))$, by the uniqueness of the Fourier coefficients. Taking into account the previous inequality, for all $h\in H_0$,

\begin{equation}\label{neweq}
    \lim_{n\to\infty}\frac{1}{|E_n|}\sum_{a,s}\|(\Theta(u_hp)-\alpha_t(\Theta(u_hp)))\alpha_t(x_s\sigma_s(\Theta(pu_ap)))\|_2^2\leq 2\epsilon^2(\|vv^*\|_2+1)^2.
\end{equation}

\vspace{2mm}

Now, consider the left-hand side of the prior equation. For every $a\in E_n=R_nF$, since $p\in\cL(H_0)$, there exists $b\in F$ such that $u_apu_{a^{-1}}=u_bpu_{b^{-1}}$. Notice that for each $a\in E_n$ there are exactly $|R_n|$ elements that will give the same $b\in F$. Therefore, $|E_n|^{-1}\sum_{a\in E_n}u_apu_{a^{-1}}=|F|^{-1}\sum_{b\in F}u_bpu_{b^{-1}}$, and so

\begin{align*}
    \frac{1}{|E_n|}&\sum_{a\in E_n}\|(\Theta(u_hp)-\alpha_t(\Theta(u_hp)))\alpha_t(x_s\sigma_s(\Theta(pu_ap)))\|_2^2\\
    &=|\tau((\Theta(u_hp)-\alpha_t(\Theta(u_hp)))^*(\Theta(u_hp)-\alpha_t(\Theta(u_hp)))\alpha_t(x_s\sigma_s(\Theta(p\frac{1}{|F|}\sum_{b\in F}(u_bpu_{b^{-1}})p))x_s^*))|\\
    &\geq \frac{1}{|F|}|\tau((\Theta(u_hp)-\alpha_t(\Theta(u_hp)))^*(\Theta(u_hp)-\alpha_t(\Theta(u_hp)))\alpha_t(x_sx_s^*))|\\
    &=\frac{1}{|F|}\|(\Theta(u_hp)-\alpha_t(\Theta(u_hp)))\alpha_t(x_s)\|_2^2.
\end{align*}

\vspace{2mm}

Combining the prior inequality with \eqref{neweq}, we obtain

\begin{align*}
    2\epsilon^2(\|vv^*\|_2+1)^2&\geq
    \frac{1}{|F|}\sum_{s\in K}\|(\Theta(u_hp)-\alpha_t(\Theta(u_hp)))\alpha_t(x_s)\|_2^2\\
    &=\frac{1}{|F|}\sum_{s\in K}\left(2\|x_s\|_2^2-2\mathfrak{R}\mathfrak{e}\tau(\Theta(u_hp)\alpha_t(x_s)\alpha_t(x_s)^*\alpha_t(\Theta(u_hp))^*\right)\\
    &=\frac{1}{|F|}\left( 2\|vv^*\|_2^2-2\mathfrak{R}\mathfrak{e}\tau\left(\Theta(u_hp)\alpha_t\left(\sum_{s\in K}x_sx_s^*\right)\alpha_t(\Theta(u_hp))^*\right)\right).
\end{align*}

\vspace{2mm}

Thus, $2\|vv^*\|_2^2-2\epsilon^2|F|(\|vv^*\|_2+1)^2\leq 2\mathfrak{R}\mathfrak{e}\tau(\alpha_t(\Theta(u_hp))^*\Theta(u_hp)\alpha_t(\sum_{s\in K}x_sx_s^*))$ for all $h\in\Sigma$. Consider the set $\mathcal{K}=\overline{co}^w\{\alpha_t(\Theta(u_hp))^*\Theta(u_hp):h\in H_0\}$ and let $w_0^t\in\mathcal{K}$ be  the unique element of minimal $\|\cdot\|_2$-norm. Observe that $w_0^t\neq 0$ because $0<2\mathfrak{R}\mathfrak{e}\tau(w_0^t\alpha_t(\sum_sx_sx_s^*))$ by choosing $\epsilon$ accordingly. Since $\alpha_t(\Theta(u_hp))^*\mathcal{K}\Theta(u_hp)=\mathcal{K}$ and $\|\alpha_t(\Theta(u_hp))^*w_0^t\Theta(u_hp)\|_2=\|w_0^t\|_2$, by the uniqueness of the element, we have $\alpha_t(\Theta(u_hp)^*)w_0^t=w_0^t\Theta(u_hp)^*$ for all $h\in H_0$. This means that $\alpha_t(y^*)w_0^t=w_0^ty^*$ for all $y\in\Theta(pL(H_0)p)$. By making $\epsilon>0$ smaller, we have a sequence $w_0^t\to q$ as $t\to 0$. Because $\alpha_t\to id$ as $t\to 0$ we have $\lim_{t\to 0}\alpha_t(v^*)w_0^tv=v^*qv=v^*v\neq 0$. In particular, there exists $t>0$ for which $\alpha_t(v^*)w_0^tv\neq 0$, and $\alpha_t(x)(\alpha_t(v^*)w_0^tv)=(\alpha_t(v^*)w_0^tv)x$ for all $x\in \Theta(\cL(H_0)p)$. Hence, $\alpha_t$ is implemented by a partial isometry and hence, by \cite[Lemma 2.4]{i07}, $\Theta(\cL(H_0)p)\prec \cL(A\times C^{(I_0)})$ or $\cL(H_0)z\prec \cL(A\times C^{(I_0)})$, for some finite set $I_0\subset D$. Using \cite[Lemma 3.9]{vaes08}, since $H_0$ has finite index in $H^{hfc}$, 

\begin{equation}\label{intertwiningfinitesubset}
    \cL(H^{hfc})z\prec \cL(A\times C^{(I_0)}).
\end{equation}

\vspace{2mm}

Let $\cQ\subset q\cL(A\times C^{(I_0)})$ be the von Neumann algebra isomorphic to $l\cL(H^{hfc})l$, for some projection $l\leq z$, that we obtain from the intertwining \eqref{intertwiningfinitesubset}. If $\cQ\nprec \cL(A)\rtimes C^{(K)}$ for any subset $K\subsetneq I_0$, then by Corollary \ref{quasinormcontrol2} we obtain that $\mathscr{QN}_{q\cM q}(\cQ)''\prec \cL(A)\rtimes (C^{(I)}\text{Norm}(I_0))$ where $\text{Norm}(I_0)=:S$ is finite. Using the same computations with quasi-normalizers as in the proof of Corollary \ref{quasinormcontrol2} we obtain that $\cL(H)\prec \sum_{s\in S}\cL(A\times C^{(I)})w_s$, a contradiction because $S$ is finite. This means that there exists $K\subsetneq I_0$ for which $\cQ\prec_{\cL(A\times C^{(I)})}\cL(A)\rtimes C^{(K)}$. By repeating the argument with $K$ instead of $I_0$, and continuing by induction, we obtain $\cQ\prec_{\cL(A\times C^{(I)})}\cL(A)$. Therefore, $\cL(H^{hfc})z\prec_{\cL(A\times C^{(I)})}\cL(A)$. By \cite[Lemma 2.4]{dhi19} we obtain $\cL(H^{hfc})\prec^s\cL(A)$. 
\end{proof}

\vspace{1mm}

\section{Integral decomposition of von Neumann algebras}

In this section we recall several definitions along with main properties of integral decompositions of von Neumann algebras and establish some results that will be applied in subsequent sections.

\vspace{2mm}

Let $\mu$ be a Radon measure on the $\sigma$-compact locally compact Hausdorff space $X$ and let $\{\mathcal{H}_x:x\in X\}$ be a measurable field of Hilbert spaces over $X$. The latter is a collection $\{\mathcal{H}_x:x\in X\}$ of Hilbert spaces together with a linear subspace $\mathcal{S}\subset \prod_x\mathcal{H}_x$ whose elements are measurable sections satisfying:

\begin{enumerate}
    \item[(a)] if $\eta\in\prod_x\mathcal{H}_x$, then $\eta\in \mathcal{S}$ if and only if $x\mapsto\langle\zeta(x),\eta(x)\rangle$ is measurable for all $\zeta\in \mathcal{S}$; and
    \item[(b)] there exists $(\zeta_n)_n\subset \mathcal{S}$ such that  $\overline{\text{span}\{\zeta_n(x)\,:\, n\in\N\}}=\mathcal{H}_x$, for $\mu$-a.e.\ $x\in X$.
\end{enumerate}

The \textit{direct integral of $\mathcal{H}_x$}, also denoted by $\int_X^{\oplus}\mathcal{H}_xd\mu(x)$, is defined as the set of all $\zeta\in \mathcal{S}$ (up to measure zero sets) such that $\int_X\|\zeta(x)\|_{\mathcal{H}_x}^2d\mu(x)<\infty$. This space is endowed with the inner product

\begin{equation*}
    \langle\zeta,\eta\rangle=\int_X\langle\zeta(x),\eta(x)\rangle_{\mathcal{H}_x}d\mu(x),
\end{equation*}

\vspace{2mm}

which turns $\int_X^{\oplus}\mathcal{H}_xd\mu(x)$ into a Hilbert space.

\vspace{2mm}

By a \emph{measurable field $\{T_x:x\in X\}$ of operators} we mean $(T_x)_{x\in X}\in\prod_{x\in X}\mathfrak B(\mathcal{H}_x)$ satisfying

\begin{enumerate}
    \item[(c)] $x\mapsto \langle T_x\xi(x),\eta(x)\rangle$ is measurable for all $\xi,\eta\in\mathcal{S}$; and
    \item[(d)] the essential supremum of $\|T_x\|$ is finite. 
\end{enumerate}

An operator $T$ on $\int^{\oplus}_X\mathcal{H}_xd\mu(x)$ is said to be \emph{decomposable} if there exists a measurable field $(T_x)_x$ of operators such that $T=\int_X^{\oplus}T_xd\mu(x)$. Direct computations show the following formulas hold:

\begin{align}\label{propertiesintegraldecomp}
    \left(\int_X^{\oplus}T_xd\mu(x)\right)\left(\int_X^{\oplus}S_xd\mu(x)\right)=&\int_X^{\oplus}T_xS_xd\mu(x),\quad
    \left(\int_X^{\oplus}T_xd\mu(x)\right)^*=\int_X^{\oplus}T_x^*d\mu(x),\\
    \int_X^{\oplus}T_xd\mu(x)+\int_X^{\oplus}&S_xd\mu(x)=\int_X^{\oplus}(T_x+S_x)d\mu(x).\nonumber
\end{align}

\vspace{2mm}

A collection of von Neumann algebras $\{\mathcal M_x:x\in X\}$ with ${\text 1}_{H_x}\in \mathcal M_x\subset \mathfrak B(\mathcal{H}_x)$ for $\mu$-almost every $x\in X$ is said to be a \emph{measurable field of von Neumann algebras} if there are measurable fields of operators $\{a_n(x):x\in X\}$ for all $n\in \mathbb N$  such that $\cM_x=W^*(\{a_n(x):n\in\N\})$ for $\mu$-a.e.\ $x\in X$. Let $\cM$ be the set of all decomposable operators $\int_X^{\oplus}a_xd\mu(x)$ with $a_x\in \cM_x$ for $\mu$-a.e.\ $x\in X$. Then $\cM$ is called the \emph{direct integral of $\mathcal M_x$} and is denoted by $\mathcal M=\int_X^{\oplus}\cM_xd\mu(x)$. 

\vspace{2mm}

\begin{rem}\label{uniquenessofdecompoverhilbertspace} By \cite[Proposition II.3.1]{dixmier} it follows that the decomposition of a von Neumann algebra represented on a fixed Hilbert space is unique. More precisely, if $\cM=\int_X^{\oplus}\cM_xd\mu(x)$ is a von Neumann algebra on $\mathcal{H}$ and there exists another measurable field of von Neumann algebras $x\mapsto \mathcal{N}_x$, which defines the same von Neumann algebra $\mathcal{M}$ in $\mathcal{H}$, then $\cM_x=\cN_x$ almost everywhere.
\end{rem}

Let $\cM=\int_X\cM_xd\mu(x)$ be a von Neumann algebra. If $\phi_x\in(\cM_x)_*$ is such that $x\mapsto \phi_x(a_x)$ is measurable for all $a\in \cM$ and $\int_X\|\phi_x\|d\mu(x)<\infty$ we have that $\phi(a)=\int_X\phi_x(a_x)d\mu(x)$ is in $\cM_*$. In this case we write $\phi=\int_X\phi_xd\mu_x$. Moreover, if $\phi_x\geq 0$, then so is $\phi$; and if $\phi_x\geq 0$ is faithful, then so is $\phi$. We refer to \cite{dixmier} for the following application of the integral decomposition of normal functionals in a von Neumann algebra.

\begin{thm}\label{decompositiontraces} Let $\mu$ be a Radon measure on a $\sigma$-compact locally compact Hausdorff space $X$, and let $(\cM_x,H_x)_x$ be a measurable field of von Neumann algebras over $X$. Set $\cM=\int_X^{\oplus}\cM_xd\mu(x)$. If $\tau$ is a normal faithful trace on $\cM$, then $\tau=\int_X^{\oplus}\tau_xd\mu(x)$ for normal faithful traces $\tau_x$ on $\cM_x$, for almost every $x\in X$.
\end{thm}

\vspace{1mm}

Let $\cM=\int_X\cM_xd\mu(x)$ and $\cN=\int_X\cN_xd\mu(x)$ be von Neumann algebras. For every $x\in X$, let $\Phi_x$ be a homomorphism of $\cM_x$ into $\cN_x$. The mapping $x\mapsto \Phi_x$ is called a field of homomorphisms. This field is said to be \textit{measurable} if, for any measurable field of operators $x\mapsto T_x\in \cM_x$, the field $x\mapsto\Phi_x(T_x)\in \cN_x$ is measurable. For $T=\int_X^{\oplus}T_xd\mu(x)\in\cM$, let $\Phi(T)=\int_X^{\oplus}\Phi_x(T_x)d\mu(x)\in\cN$. Then, $\Phi$ is a $*$-homomorphism from $\cM$ into $\cN$, and it is normal if the $\Phi_x$ are normal for almost every $x\in X$. A normal homomorphism $\Phi$ of $\cM=\int_X\cM_xd\mu(x)$ into $\cN=\int_X\cN_xd\mu(x)$ is said to be \textit{decomposable} if it is defined by a measurable field $x\mapsto \Phi_x$ of normal homomorphisms of $\cM_x$ into $\cN_x$. In this case we write $\Phi=\int_X^{\oplus}\Phi_xd\mu(x)$. Dixmier \cite[Proposition II.3.11]{dixmier} established that, under specific conditions, a normal homomorphism $\Phi:\mathcal{M}\to\cN$ is indeed decomposable. Later on, Spaas established that, under more restrictions, the integral decomposition of the initial von Neumann algebras is preserved.

\begin{thm}\label{decompositioniso}(\cite[Theorem 2.1.14]{spaas}) Suppose $\cM=\int_X^{\oplus}\cM_x d\mu(x)$ and $\cN=\int_Y^{\oplus}\cN_xd\nu(x)$ are direct integrals of tracial von Neumann algebras and $\Phi:\cM\to\cN$ is a trace preserving isomorphism such that $\Phi(L^{\infty}(X))=L^{\infty}(Y)$. Then, there exists full measure sets $X'\subset X$ and $Y'\subset Y$, a Borel isomorphism $f:Y'\to X'$ with $f(\nu)$ equivalent to $\mu$, and a measurable field $x\mapsto \Phi_x$ of tracial isomorphisms of $\cM_x$ onto $\cN_{f^{-1}(x)}$ such that

\begin{equation*}
    \Phi=\int_X^{\oplus}\Phi_xd\mu(x).
\end{equation*}
\end{thm}

\vspace{0.5mm}

\subsection{Direct integral decomposition of cocycle actions} We will now establish a connection between two notions: integral decomposition and cocycle actions.

\begin{defn}\label{defncocycleactions} Let $\mu$ be a Radon measure on a $\sigma$-compact locally compact Hausdorff space $X$, and let $(\cM_x,\mathcal{H}_x)$ be a measurable field of von Neumann algebras over $X$. Let $B$ be a countable group. A \emph{measurable field of cocycle actions} $\{\alpha_x, c_x:x\in X\}$ of $B$ is a collection of cocycle actions $B\curvearrowright^{\alpha_x,c_x}\cM_x$ for $x\in X$ such that for each $g\in B$, the map $x\mapsto \alpha_x(g)$ is a measurable field of automorphisms, and the map $x\mapsto\langle c_x(g,h)\xi(x),\eta(x)\rangle$ is measurable for all $\xi,\eta\in\mathcal{S}$ and all $g,h\in B$. Set $\cM=\int_X^{\oplus}\cM_xd\mu(x)$. Define $\alpha(g)(a)=\int_X^{\oplus}\alpha_x(g)(a_x)d\mu(x)$ for $a\in\cM$, and $c:B\times B\to \mathscr{U}(\cM)$ by $c(g,h)=\int_X^{\oplus}c_x(g,h)d\mu(x)$ for $g,h\in B$. Then, $B\curvearrowright^{\alpha,c}\cM$ is a cocycle action that we refer to as the \emph{cocycle direct integral of $\alpha_x,c_x$}.
\end{defn}

\begin{rem} If $B\curvearrowright^{\alpha_x,c_x}(\cM_x,\mathcal{H}_x)$ is a measurable field of cocycle actions as in the prior definition we obtain that $\{\cM_x\rtimes_{\alpha_x,c_x}B:x\in X\}$ is a measurable field of von Neumann algebras. By \cite[Proposition II.8.10]{dixmier} the field of Hilbert spaces $\{\mathcal{H}_x\otimes\ell^2(B):x\in X\}$ is $\mu$-measurable and 
$\{a_n(x)u_b:x\in X,b\in B\}$ is a measurable field of operators for all $n\in\N$ such that $\mathcal{M}_x\rtimes_{\alpha_x,c_x}B=W^*(\{a_n(x)u_b:n\in\N,b\in B\})$ for $\mu$-a.e. $x\in X$. Hence, we obtain the von Neumann algebra $\int_X^{\oplus}(\cM_x\rtimes_{\alpha_x,c_x}B)\:d\mu(x)$. 
\end{rem}

\begin{prop}\label{integraldecompcocycles} Let $\mu$ be a Radon measure on a $\sigma$-compact locally compact Hausdorff space $X$, and let $(\cM_x,\mathcal{H}_x)_x$ be a measurable field of tracial von Neumann algebras over $X$. Let $\cM=\int_X^{\oplus}\cM_xd\mu(x)$ be tracial. Let $B$ be a countable group and suppose $B\curvearrowright^{\alpha,c}\cM$ is a cocycle action. Assume $\alpha_g$ acts trivially on $L^{\infty}(X,\mu)$ for every $g\in B$. Then, there exists a measurable field $\{\alpha_{\cdot,x},c_x:x\in Y\}$, where $Y$ is a conull subset of $X$, of cocycle actions such that

\begin{equation*}
    \alpha_g=\int_{Y}^{\oplus}\alpha_{g,x}d\mu(x),\quad\text{for all }g\in B
\end{equation*}

\vspace{2mm}

and $c_x:B\times B\to\mathscr{U}(\cM_x)$ is a cocycle associated to $\alpha_{\cdot,x}$ with

\begin{equation*}
    c=\int_{Y}^{\oplus}c_xd\mu(x).
\end{equation*}
\end{prop}

\begin{proof} First note that for each $g\in B$, $\alpha_g\in\text{Aut}(\cM)$ is a $*$-isomorphism that preserves the trace of $\cM$. Thus, its decomposition follows by Theorem \ref{decompositioniso} and the assumption that $\alpha_g$ preserves the fibers. That is, after removing a countable union of measure zero sets, for each $g\in B$ there exists a measurable field of trace preserving $*$-automorphisms $x\mapsto \alpha_{g,x}\in\text{Aut}(\cM_x)$ such that

\begin{equation*}
    \alpha_g=\int_{X}\alpha_{g,x}d\mu(x).
\end{equation*}

\vspace{2mm}

For $g,h\in B$ there exists operators $(c(g,h)_x)_{x\in X}$ with $c(g,h)_x\in M_x$ for $\mu$-a.e.\ $x\in X$ and 

\begin{equation*}
    c(g,h)=\int_X^{\oplus}c(g,h)_xd\mu(x).
\end{equation*}

\vspace{2mm}

Let $c_x(g,h):=c(g,h)_x$ for $g,h\in B$. By the definition of a measurable field of operators we have that $x\mapsto \langle a(x)\xi(x),\eta(x)\rangle$ is measurable for all $a\in\cM$. In particular, we have that $x\mapsto \langle c_x(g,h)\xi(x),\eta(x)\rangle$ is $\mu$-measurable. 

\vspace{2mm}

Finally, we show $B\curvearrowright^{\alpha_{\cdot,x},c_x}\cM_x$ is a cocycle action for almost every $x\in X$. Let $a_n(x)$ be a weak$^*$ dense sequence of measurable fields of operators in $\cM$. Then, for all $g,h\in B$, by using equations \eqref{propertiesintegraldecomp} we obtain

\begin{align*}
    \int_X\text{Ad}(c_x(g,h))\alpha_{gh,x}(a_n(x))d\mu(x)&=\text{Ad}(c(g,h))\alpha_{gh}(a_n)=(\alpha_g\circ\alpha_h)(a_n)\\
    &=\int_X(\alpha_{g,x}\circ\alpha_{h,x})(a_n(x))d\mu(x).
\end{align*}

\vspace{2mm}

Hence, $\text{Ad}(c_x(g,h))\alpha_{gh,x}(a_n(x))=(\alpha_{g,x}\circ\alpha_{h,x})(a_n(x))$ holds for $\mu$-a.e. $x\in X$. Taking a countable union of null sets we may assume that $\text{Ad}(c_x(g,h))\alpha_{gh,x}(a_n(x))=(\alpha_{g,x}\circ\alpha_{h,x})(a_n(x))$ for every $x$, for every $n$ and for every $g,h\in B$. By the weak$^*$ density of $a_n(x)$, we get that $\text{Ad}(c_x(g,h))\alpha_{gh,x}=\alpha_{g,x}\circ\alpha_{h,x}$ for every $x\in X$ and every $g,h\in B$. 

\vspace{2mm}

Moreover, by \eqref{propertiesintegraldecomp}

\begin{equation*}
    \int_Xc_x(g,h)c_x(gh,k)d\mu(x)=c(g,h)c(gh,k)=\alpha_g(c(h,k))c(g,hk)=\int_X\alpha_{g,x}(c_x(h,k))c_x(g,hk)d\mu(x),
\end{equation*}

\vspace{2mm}

for $g,h,k\in B$. The previous equality implies that $c_x(g,h)c_x(gh,k)=\alpha_{g,x}(c_x(h,k))c_x(g,hk)$ holds for almost every $x\in X$. Using \eqref{propertiesintegraldecomp} once again, we obtain $c_x(g,1_B)=c_x(1_B,g)=1_{\cM_x}$ for $\mu$-a.e. $x\in X$. Thus, $B\curvearrowright^{\alpha_{\cdot,x},c_x}\cM_x$ is a cocycle action for almost every $x\in X$, as claimed.
\end{proof}

The following theorem is well-known, but we present it here because it will be frequently referenced.

\begin{thm}\label{extensionisom} Let $(\mathcal A,\tau), (\mathcal B,\rho)$ be tracial von Neumann algebras, and let $\mathcal A_0\subset \mathcal A$ and $\mathcal B_0\subset \mathcal B$ be dense $\ast$-subalgebras. Let $\Phi:\mathcal A_0\to \mathcal B_0$ be a trace preserving $\ast$-homomorphism. Then, $\Phi$ extends to a $\ast$-isomorphism between $\mathcal A$ and $\overline{\Phi(\mathcal A_0)}^{\rm WOT}$.
\end{thm}

\begin{proof} Since the map $\Phi$ is trace preserving, it is injective. Moreover, as $\Phi$ is a $\ast$-homomorphism, and the traces are normal, $\Phi$ is continuous in the $\sigma$-WOT topology. Since $\overline{\mathcal A_0}^{\sigma-\text{WOT}}=\mathcal A$, and using a standard extension argument, we see that $\Phi$ uniquely extends to a well-defined injective and normal $\ast$-homomorphism $\mathcal A\to \mathcal B$.
\end{proof}

\begin{thm}\label{isomcocycleintegral} Let $\mu$ be a Radon measure on a $\sigma$-compact locally compact Hausdorff space $X$, and let $(\cM_x,\mathcal{H}_x)_x$ be a measurable field of tracial von Neumann algebras over $X$. Set $\cM=\int_X^{\oplus}\cM_xd\mu(x)$ and let $\tau$ be a trace on $\cM$. Let $B$ be a countable group and $B\curvearrowright^{\alpha,c}(\cM,\tau)$ a cocycle action. Assume $\alpha_g$ acts trivially on $L^{\infty}(X,\mu)$ for every $g\in B$. Then, the cocycle crossed product von Neumann algebra $\cM\rtimes_{\alpha,c}B$ can be decomposed as the direct integral of cocycle crossed products $\cM_x\rtimes_{\alpha_{\cdot,x},c_x}B$. In other words,

\begin{equation*}
    \cM\rtimes_{\alpha,c}B\cong\int_X\cM_x\rtimes_{\alpha_{\cdot,x}c_x}B\:d\mu(x),
\end{equation*}

\vspace{2mm}

where the cocycle action integral decomposition is given as in Proposition \ref{integraldecompcocycles}.
\end{thm} 

\begin{proof} Define the map $\Phi:\cM\rtimes_{\alpha,c}B\to \int_X\cM_x\rtimes_{\alpha_{\cdot,x}c_x}B\:d\mu(x)$ on the generators by $\Phi(\int_Xa_xd\mu(x)u_b)= \int_Xa_xu_bd\mu(x)$ and extend by linearity. This map is trace preserving. To see this, use Theorem \ref{decompositiontraces} to decompose $\tau$ as $\int_X\tau_x\mu(x)$ with $\tau_x$ a trace for $\cM_x$, for almost every $x$. Then,

\begin{align*}
    \tau\left(\int_Xa_xd\mu(x)u_b\right)&=\delta_{b,e}\:\tau\left(\int_Xa_xd\mu(x)\right)=\delta_{b,e}\int_X\tau_x(a_x)d\mu(x)\\
    &=\int_X\tau_x(a_x)\delta_{b,e}d\mu(x)=\int_X\tau_x(a_xu_b)d\mu(x).
\end{align*}

\vspace{2mm}

The map $\Phi$ is a $\ast$-homomorphism. Indeed, by \eqref{propertiesintegraldecomp}, we have 

\begin{align*}
    \int_Xa_xd\mu(x)u_{b_1}\int_Xb_xd\mu(x)u_{b_2}&=\int_Xa_xd\mu(x)\:\alpha_{b_1}\left(\int_Xb_xd\mu(x)\right)c(b_1,b_2)u_{b_1b_2}\\
    &=\int_Xa_x\alpha_{b_1,x}(b_x)c_x(b_1,b_2)d\mu(x)u_{b_1b_2}.
\end{align*}

\vspace{2mm}

Hence,

\begin{align*}
    \Phi\left(\int_Xa_xd\mu(x)u_{b_1}\int_Xb_xd\mu(x)u_{b_2}\right)&=\int_Xa_x\alpha_{b_1,x}(b_x)c_x(b_1,b_2)u_{b_1b_2}d\mu(x)=\int_Xa_xu_{b_1}b_xu_{b_2}d\mu(x)\\
    &=\Phi\left(\int_Xa_xd\mu(x)u_{b_1}\right)\Phi\left(\int_Xb_xd\mu(x)u_{b_2}\right), \quad\text{and}\\
    \Phi\left(\left(\int_Xa_xd\mu(x)u_{b}\right)^*\right)&=\Phi\left(\int_X\alpha_{b^{-1},x}(a_x^*)d\mu(x)u_{b^{-1}}\right)=\int_X\alpha_{b^{-1},x}(a_x^*)u_{b^{-1}}d\mu(x)\\
    &=\int_Xu_b^*a_x^*d\mu(x)=\Phi\left(\int_Xa_xd\mu(x)u_b\right)^*
\end{align*}

\vspace{2mm}

by \eqref{propertiesintegraldecomp} again. Using Theorem \ref{extensionisom}, $\Phi$ lifts to a von Neumann algebra $\ast$-isomorphism.
\end{proof}

We end this subsection by showing that if the cocycles $c_x$, corresponding to a group von Neumann algebra cocycle $c$, are 2-coboundaries then $c$ is a 2-coboundary as well.

\begin{prop}\label{integ2coboundaries} Suppose $B$ is a countable group, and consider a cocycle crossed product action $B\curvearrowright^c L^{\infty}(X)$ (resp. $B\curvearrowright^d L^{\infty}(Y)$). Define $\cM:=L^{\infty}(X)\rtimes_cB$ (resp. $\cN:=L^{\infty}(Y)\rtimes_dB$), and let $c=\int_X^{\oplus}c_xd\mu(x)$ (resp. $d=\int_Y^{\oplus}d_yd\mu(y)$) denote the decomposition of $c$ (resp. $d$) into 2-cocycles $c_x:B\times B\to\C$ (resp. $d_y:B\times B\to\C$), as given in Proposition \ref{integraldecompcocycles}. Assume $\Phi:\cM\to\cN$ is a tracial isomorphisms such that $\Phi(L^{\infty}(X))=L^{\infty}(Y)$. 

\vspace{2mm}

Assume there exists a measurable measure-preserving map $f:X'\to Y'$, where $X'\subset X$, $Y'\subset Y$ are co-null sets, such that 
$c_{f^{-1}(y)}$ is cohomologous to $d_y$ for every $y\in Y'$. Then $c$ is cohomologous to $d$.

\vspace{2mm}

In particular, if the cocycles $c_x:B\times B\to\C$ are 2-coboundaries for almost every $x\in X$, then $c:B\times B\to L^{\infty}(X)$ is a 2-coboundary.
\end{prop}

\begin{proof} By considering $\Phi\circ c$, we may assume $c$ is a 2-cocycle from $B$ into $L^{\infty}(Y)$. By the prior proposition $\cM=\int_X^{\oplus}\cL_{c_x}(B)d\mu(x)$ and $\cN=\int_Y^{\oplus}\cL_{d_y}(B)d\mu(y)$. Let $\mathcal{H}$ be the Hilbert space for which $\cL_{d_y}(B)\subset\mathscr{B}(\mathcal{H})$ for almost every $y\in Y$.

\vspace{2mm}

Consider the set 

\begin{align*}
    A&=\{(y,(\xi_g)_{g\in B})\in Y\times \mathscr{B}(\ell^2(B,\mathcal{H})):\xi_g\in\mathbb{T}_y\text{ and }c_{f^{-1}(y)}(g,h)=\xi_g\xi_h\xi_{gh}^{-1}d_y(g,h)\text{ for all }g,h\in B\}\\
    &=\{(y,(\xi_g)_g):(\xi_g)_g\in\mathscr{U}(\ell^{\infty}(B,\mathbb{T}_y))\}\cap\{(y,(\xi_g)_g):c_{f^{-1}(y)}(g,h)=\xi_g\xi_h\xi_{gh}^{-1}d_y(g,h)\text{ for all }g,h\in B\}.
 \end{align*}

\vspace{2mm}

The set $\{(y,(\xi_g)_g):(\xi_g)_g\in\mathscr{U}(\ell^{\infty}(B,\mathbb{T}_y))\}$ is Borel by \cite[Corollary 8.5]{takesaki}. Thus, to show $A$ is measurable it suffices to show $\{(y,(\xi_g)_g):c_{f^{-1}(y)}(g,h)=\xi_g\xi_h\xi_{gh}^{-1}d_y(g,h)\text{ for all }g,h\in B\}$ is measurable. However, this follows because the functions $(y,(\xi_g)_g)\mapsto c_{f^{-1}(y)}(g,h),d_y(g,h)$ and $(y,(\xi_g)_g)\mapsto \xi_g,\xi_h,\xi_{gh}^{-1}$ are Borel when we consider the strong$^*$ topology on $\mathscr{B}(\ell^2(B,\mathcal{H}))$.

\vspace{2mm}

Let $\pi_1:A\to Y$ be the projection onto the first coordinate. By assumption, we have that $\pi_1(A)=Y'$. By \cite[Theorem A.16]{takesaki} there exists a measurable section $s^g:Y'\to\mathscr{B}(\ell^2(B,\mathcal{H}))$, for every $g\in B$, satisfying $(y,(s^g(y))_g)\in A$ for every $y\in Y'$; that is, $s^g(y)\in\mathbb{T}_y$ and $c_{f^{-1}(y)}(g,h)=s^g(y)s^h(y)s^{gh}(y)^{-1}d_y(g,h)$ for all $g,h\in B$ and every $y\in Y'$. By letting $s^l=\int_Y^{\oplus}s^l(y)d\mu(y)\in L^{\infty}(Y)$, with $l\in B$, we see that $c(g,h)=s^gs^h(s^{gh})^{-1}d(g,h)$ for every $g,h\in B$. This shows $c$ is cohomologous to $d$.
\end{proof}

\vspace{0.5mm}

\subsection{Finite index inclusions} In this subsection, we establish that finite index inclusion of von Neumann algebras is a property preserved under the integral decomposition of von Neumann algebras. While we will reference these findings in the proof of \ref{theoremB}, the complete results are not explicitly required for our purposes. Nevertheless, we anticipate that these results may find application in other contexts.

\vspace{2mm}

First, recall that $\mathcal{Q}\subset \cM$ is an inclusion of finite index von Neumann algebras if and only if the inclusion has finite Pimsner-Popa basis \cite{pimsnerpopa}; that is, there exists a finite family $(m_i)_{i=1}^k\subset \cM$ with $\mathbb{E}_{\mathcal{Q}}(m_im_j^*)=\delta_{i,j}p_i$, where $0\neq p_i\in\mathscr{P}(\mathcal{Q})$, such that for all $y\in\cM$ we have $y=\sum_{i=1}^k\mathbb{E}_{\mathcal{Q}}(y m_i^*)m_i$.

\vspace{2mm}

\begin{thm}\label{decompcentralalg}(\cite[Theorem IV.8.21]{takesaki}) If $(\cM,\mathcal{H})$ is a von Neumann algebra on a separable Hilbert space, then for any von Neumann subalgebra $\mathcal{A}$ of the center $\mathscr{Z}(\cM)$, there exists a measurable field $\{(\cM_y,\mathcal{H}_y):y\in Y\}$ of von Neumann algebras on a $\sigma$-finite measure space $(Y,\nu)$ such that $(\cM,\mathcal{H})\cong\int_Y^{\oplus}(\cM_y,\mathcal{H}_y)\:d\mu(y)$, and $\mathcal{A}$ is the diagonal algebra.
\end{thm}

Let $\cM=\int_X^{\oplus}\cM_xd\mu(x)$ be the integral decomposition of a tracial von Neumann algebra over its center, and let $\mathcal{P}\subset\cM$ be a von Neumann subalgebra containing the center $L^{\infty}(X,\mu)$ of $\cM$. By letting $\mathcal{A}=L^{\infty}(X,\mu)$ in Theorem \ref{decompcentralalg}, there exists a measurable field of von Neumann algebras $X\ni x\mapsto \mathcal{P}_x$ for which $\mathcal{P}\cong\int_X^{\oplus}\mathcal{P}_xd\mu(x)$. Henceforth, we identify $\mathcal{P}$ with its integral decomposition $\int_X^{\oplus}\mathcal{P}_xd\mu(x)$ so that $\mathcal{P}_x\subset \cM_x$ for almost every $x\in X$ \cite[Proposition II.2.1]{dixmier}.

\vspace{2mm}

\begin{prop}\label{surjective*homdecomp} Suppose $(\cM,\tau)=\int_X^{\oplus}(\cM_x,\tau_x)\:d\mu(x)$ is a direct integral decomposition of a tracial von Neumann algebra over its center, and let $\mathcal{P}=\int_X^{\oplus}\mathcal{P}_x\:d\mu(x)\subset\cM$ be the integral decomposition of a von Neumann subalgebra containing the center. Denote by $\mathbb{E}_{\mathcal{P}}:\cM\to\mathcal{P}$ the conditional expectation from $\cM$ onto $\mathcal{P}$. Then, there exists a set $X'\subset X$ of full measure for which $\mathbb{E}_{\mathcal{P}}$ decomposes as the direct integral of conditional expectations $\{\mathbb{E}_x:\cM_x\to\mathcal{P}_x\}_{x\in X'}$.
\end{prop}

\begin{proof} Since $\mathbb{E}_{\mathcal{P}}$ is trace-preserving, it gives rise to a projection map $e_{\mathcal{P}}:L^2(\cM)\to L^2(\mathcal{P})$. Notice that since the traces decompose by Theorem \ref{decompositiontraces}, we must have that $L^2(\cM)=\int_X^{\oplus}L^2(\cM_x)\:d\mu(x)$ and $L^2(\mathcal{P})=\int_X^{\oplus}L^2(\mathcal{P}_x)\:d\mu(x)\subset L^2(\cM)$ are direct integrals of Hilbert spaces. To show $e_{\mathcal{P}}$ decomposes as a direct integral of, necessarily, projections from $L^2(\mathcal{M}_x)$ onto $L^2(\mathcal{P}_x)$ for almost every $x\in X$, it suffices to show $e_{\mathcal{P}}$ commutes with the diagonal algebra \cite[Corollary IV.8.16]{takesaki}. For that, it suffices to show $e_{\mathcal{P}}$ commutes with the diagonal algebra on the dense subspace $\cM\hat{1}\subset L^2(\cM)$. Let $m\in\cM$, $f\in L^{\infty}(X,\mu)$, and observe that since $L^{\infty}(X,\mu)\subset\mathcal{P}$ and $\mathbb{E}_{\mathcal{P}}(m)\hat{1}=e_{\mathcal{P}}\hat{m}$, we have

\begin{equation*}
    (e_{\mathcal{P}}f)\:m\hat{1}=e_{\mathcal{P}}(\hat{m})\cdot f=\mathbb{E}_{\mathcal{P}}(m)\hat{1}\cdot f=\mathbb{E}_{\mathcal{P}}(fm)\hat{1}=(fe_{\mathcal{P}})\:m\hat{1}.
\end{equation*}

\vspace{2mm}

Thus, there exists a field $x\mapsto e_{x}$ of bounded linear operators from $L^2(\cM_x)$ to $L^2(\mathcal{P}_x)$ for which $e_{\mathcal{P}}=\int_X^{\oplus}e_xd\mu(x)$. Define $\Phi_x:\cM_x\to\mathcal{P}_x$ to be the restriction of $e_x$ on $\cM_x$, and observe that 

\begin{equation*}
    \mathbb{E}_{\mathcal{P}}=\int_X^{\oplus}\Phi_xd\mu(x).
\end{equation*}

\vspace{2mm}

From \cite{dixmier,takesaki} we obtain that $\Phi_x$ is a trace-preserving, surjective, positive and $*$-linear map up to a measure zero set. Let $a_n(x)$ (respectively $p_m(x),q_l(x)$) be a weak$^*$ dense sequence of measurable fields of operators in $\cM$ (respectively $\mathcal{P}$). By equations \eqref{propertiesintegraldecomp}

\begin{equation*}
    \int_X^{\oplus}(p_m(x)\Phi_x(a_n(x))q_l(x))\:d\mu(x)=p_m\mathbb{E}_{\mathcal{P}}(a_n)q_l=\mathbb{E}_{\mathcal{P}}(p_ma_nq_l)=\int_X^{\oplus}\Phi_x(p_m(x)a_n(x)q_l(x))\:d\mu(x),
\end{equation*}

\vspace{2mm}

and hence, $p_m(x)\Phi_x(a_n(x))q_l(x)=\Phi_x(p_m(x)a_n(x)q_l(x))$ for almost every $x\in X$. By excising countably many measure zero sets, together with the fact that the sequences $a_n,p_m,q_l$ are weak$^*$ dense and the normality of the $\Phi_x$, we can assume that $p\Phi_x(a)q=\Phi_x(paq)$ holds for every $a\in\cM_x$, $p,q\in\mathcal{P}_x$ and every $x\in X$. Therefore, $\Phi_x$ is the conditional expectation from $\cM_x$ onto $\mathcal{P}_x$ for almost every $x\in X$.
\end{proof}

\begin{thm}\label{finiteindexfibers} Let $\mu$ be a Radon measure on a $\sigma$-compact locally compact Hausdorff space $X$, and let $(\mathcal{M}_x,\mathcal{H}_x)_x$ be a measurable field of tracial von Neumann algebras over $X$. Set $\mathcal{M}=\int_X^{\oplus}\cM_x\:d\mu(x)$. Let $L^{\infty}(X,\mu)\subset\int_X^{\oplus}\cM_xd\mu(x)=\cM$ be the inclusion of the center in the direct integral decomposition. Assume $L^{\infty}(X,\mu)\subset \mathcal{P}=\int_X^{\oplus}\mathcal{P}_x\subset \cM$ is the direct integral decomposition of an intermediate von Neumann subalgebra that has finite index in $\cM$. Then, the inclusion $\mathcal{P}_x\subset\mathcal{M}_{x}$ has finite index for almost every $x\in X$.  

\vspace{2mm}

In particular, if $\mathcal{P}=L^{\infty}(X,\mu)\:\overline{\otimes}\:\mathcal{Q}$, for $\mathcal{Q}\subset \cM$, then the von Neumann algebras $\cM_x$ have isomorphic copies of $\mathcal{Q}$ of finite index for almost every $x\in X$.
\end{thm}

\begin{proof} Since $\mathcal{P}\subset\cM$ is finite index, it admits a finite Pimsner-Popa basis, say $\{m_1,...,m_k\}\subset\cM$. Let $m_i=\int_Xm_{i,x}\:d\mu(x)$ be the integral decomposition of each of the basis elements and let $p_i=\int_Xp_{i,x}\:d\mu(x)$ be the decomposition of each projection. Note that $p_{i,x}$ is a projection for almost every $x\in X$ and every $1\leq i\leq k$. Since $\mathbb{E}_{\mathcal{P}}$ is a conditional expectation, by Proposition \ref{surjective*homdecomp} there exists conditional expectations $\mathbb{E}_x:\cM_x\to \mathcal{P}_{x}$ for almost every $x$, with $\mathbb{E}_{\mathcal{P}}=\int_X^{\oplus}\mathbb{E}_x\:d\mu(x)$. 

\vspace{2mm}

Thus, for a weak$^*$ dense sequence of measurable fields of operators $a_n(x)$ in $\cM$, we have

\begin{align*}
    \int_X^{\oplus}a_n(x)\:d\mu(x)=a_n=\sum_{i=1}^k\mathbb{E}_{\mathcal{P}}(a_nm_i^*)m_i=\int_X\sum_{i=1}^k\mathbb{E}_x(a_n(x)m_{i,x}^*)m_{i,x}\:d\mu(x),
\end{align*}

\vspace{2mm}

showing that $a_n(x)=\sum_{i=1}^k\mathbb{E}_x(a_n(x)m_{i,x}^*)m_{i,x}$ for almost every $x\in X$. By excising a set of measure zero, $a_n(x)=\sum_{i=1}^k\mathbb{E}_x(a_n(x)m_{i,x}^*)m_{i,x}$ holds for all $n$ and all $x\in X$. By the weak$^*$ density of $(a_n(x))_n$ in $\cM_x$ and the normality of $\mathbb{E}_x$, we obtain that $a=\sum_{i=1}^k\mathbb{E}_x(am_{i,x}^*)m_{i,x}$ for all $a\in\mathcal{M}_x$.

\vspace{2mm}

Similarly, since 

\begin{equation*}
    \int_X\delta_{i,j}p_{i,x}=\delta_{i,j}p_i=\mathbb{E}_{\mathcal{P}}(m_im_j^*)=\int_X\mathbb{E}_x(m_{i,x}m_{j,x}^*)\:d\mu(x),
\end{equation*}

\vspace{2mm}

we obtain $\delta_{i,j}p_{i,x}=\mathbb{E}_x(m_{i,x}m_{j,x}^*)$ for almost every $x\in X$. Hence, $\{m_{1,x},...,m_{k,x}\}$ is a finite Pimsner-Popa basis for $\mathcal{P}_{x}\subset\cM_x$ for almost every $x\in X$.

\vspace{2mm}

The last part follows by the first by observing that $\mathcal{P}=\int_X^{\oplus}\mathcal{Q}\:d\mu(x)$.
\end{proof}

The following result describes the fibers in the case of an inclusion of abelian von Neumann algebras.

\begin{lem}\label{completelyatomicfibers} Let $(X,\mu)$ and $(Y,\nu)$ be finite measure spaces such that $L^{\infty}(X,\mu)\subseteq L^{\infty}(Y,\nu)$. Assume there exists a countable family of orthogonal projections $\{r_i\}_i\subset L^{\infty}(Y,\nu)$ with $L^{\infty}(Y,\nu)=\oplus_iL^{\infty}(X,\mu)r_i$ and $\sum_ir_i=1$. Then, the fibers in the integral decomposition of $L^{\infty}(Y,\nu)$ over $X$ are completely atomic. 

\vspace{2mm}

Moreover, if $\mathcal{A}$ is a von Neumann algebra containing $L^{\infty}(Y,\nu)$ as a finite index subalgebra, the same holds.
\end{lem}

\begin{proof} Using the prior theorem, $L^{\infty}(Y,\nu)=\int_X^{\oplus}\mathcal{D}_xd\mu(x)$. For each $i$, let $r_i=\int_X^{\oplus}r_{i,x}d\mu(x)$ where $r_{i,x}$ are projections for almost every $x\in X$. By assumption, $L^{\infty}(Y)r_i=L^{\infty}(X)r_i$, and hence

\begin{equation*}
    L^{\infty}(Y)r_i=\int_X^{\oplus}\mathcal{D}_xr_{i,x}d\mu(x)=\int_X^{\oplus}\C r_{i,x}d\mu(x).
\end{equation*}

\vspace{2mm}

Therefore, $\mathcal{D}_xr_{i,x}=\C r_{i,x}$ for almost every $x\in X$. By removing countably many zero measure sets, we may assume $\mathcal{D}_xr_{i,x}=\C r_{i,x}$ for every $i\in I$ and every $x\in X$. Since $\sum_ir_{i,x}=1_{\mathcal{D}_x}$ for almost every $x\in X$, we obtain that the fibers $\mathcal{D}_x$ are completely atomic up to a measure zero set.

\vspace{2mm}

Using Theorem \ref{finiteindexfibers}, we know that the fibers $\mathcal{A}_x$ of the integral decomposition of $\mathcal{A}$ over $X$ contain $\mathcal{D}_x$ as finite index subalgebras. The moreover part follows.
\end{proof}

\section{Proof of \ref{theoremA}}

This section is devoted to the proof of \ref{theoremA}. To recall the set up, assume $G\in\mathcal{WR}(A,B\curvearrowright I)$ where $A$ is an abelian group, $B$ is an ICC subgroup of a hyperbolic group, and the action $B\curvearrowright I$ has amenable stabilizers. Let $\mathcal{Z}$ be an abelian von Neumann algebra, and assume $H$ is an arbitrary countable group such that $\cL(H)\cong\mathcal{Z}\:\overline{\otimes}\:\cL(G)=:\cM$.

\begin{lem}\label{diffusefibers}  Let $\mathcal M = L^\infty(X) \overline \otimes \mathcal N$ be a von Neumann algebra. Let $\mathcal{Q}\subset q\mathcal M q$ be a diffuse von Neumann subalgebra for some projection $q=\int_X^{\oplus}q_xd\mu(x)\in \cM$. Write $\mathcal{Q}=\int_X^{\oplus}\mathcal{Q}_xd\mu(x)\subseteq \int_X^{\oplus}q_x\cN q_xd\mu(x)$. If $\mathcal{Q}\not\prec (L^{\infty}(X)\otimes 1)q$, then $\mathcal{Q}_x$ is diffuse for almost every $x\in X$.
\end{lem}

\begin{proof} Let $\cM=\int^{\oplus}_X\cN d\mu(x)$ be the integral decomposition of $\cM$. Since $\mathcal{Q}\not\prec (L^{\infty}(X)\otimes 1)q$ there is a sequence of unitaries $(u_n)_n\in\mathcal{Q}$ so that $\|\mathbb{E}_{(L^{\infty}(X)\otimes 1)q}(vu_nw)\|_2\to 0$ for all $v,w\in q\cM q$. Writing $v=\int_X^{\oplus}v_xd\mu(x)$, $w=\int_X^{\oplus}w_xd\mu(x)$ and $u_n=\int_X^{\oplus}u_{n,x}d\mu(x)$ we get

\begin{align*}
    \|\mathbb{E}_{(L^{\infty}(X)\otimes 1)q}(vu_nw)\|_2^2&=\left\|\mathbb{E}_{(L^{\infty}(X)\otimes 1)q}\left(\int^\oplus_Xv_xu_{n,x}w_xd\mu(x)\right)\right\|_2^2=\int_X^{\oplus}|\tau_{q_x}(v_xu_{n,x}w_x)|^2d\mu(x),
\end{align*}

\vspace{2mm}

which converges to zero. Now let $\{t_i\}_{i\in\N}$ be a countable $\|\cdot\|_2$-dense subset of $q_x\cN q_x$. Letting $v$ and $w$ range over $\{1\otimes t_i:i\in\N\}$ we get a subsequence of $(u_n)_n$ such that for almost every $x\in X$ we have $|\tau_{q_x}(t_iu_{n,x}t_j)|\to 0$ for all $i,j\in\N$. Since $\{t_i\}_{i\in\N}$ is $\|\cdot\|_2$-dense in $q_x\cN q_x$, this means that $\mathcal{Q}_x$ is diffuse for almost every $x\in X$.
\end{proof}

For the proof of \ref{theoremA} we also need the following version of \cite[Theorem 4.1]{cios22}.

\begin{thm}\label{SOE} Let $D$ be an ICC group so that for every $1\neq d \in D$ its centralizer $C_D(d)$ is amenable. Let $D\curvearrowright^{\alpha}(X,\mu)=(Y^I,\nu^I)$ be a measure preserving action built over an action $D\curvearrowright I$, where $(Y,\nu)$ is a probability space. Assume $D\curvearrowright I$ has amenable stabilizers.  

\vspace{2mm}

Let $B$ be a countable property (T) group. Let $B\curvearrowright^{\beta} (X, \mu)$ be free, measure preserving action such that  $B\cdot x\subseteq  D\cdot x$, for almost every $x\in X$. Moreover, assume there exists a measurable set of positive measure $X_0 \subseteq X$ such that the restrictions of the induced equivalence relations $\mathscr R(B \curvearrowright ^\beta X  )_{|X_0\times X_0}\subseteq \mathscr R (D \curvearrowright^{\alpha} X)_{| X_0\times X_0}$ have finite index. 
\vspace{2mm}

Then one can find a finite index subgroup $D_1\leqslant D$,  a group isomorphism $\delta:B\rightarrow D_1$ and $\theta\in [\mathscr R( D\curvearrowright X)]$ such that
$\theta\circ\beta(h)=\alpha(\delta(h))\circ\theta$, for every $h\in B$.
\end{thm}

\begin{proof} First,  since all stabilizers of $D\curvearrowright I$ and all centralizers $C_D(g)$ for $g\neq 1$ are amenable groups and $B$ has property (T) and $B\curvearrowright^\beta X$ is a free action then for every $i \in I$  and $g\neq 1$ one can find a sequence $(h_m)\subset B$ such that  for every  $s,t\in  D$ we have 
 
\begin{equation}\label{measuretozero}\begin{split}
    &\lim_{m\ra \infty}\mu(\{x\in X_0\mid \beta(h_m)(x)\in \alpha(s\text{Stab}_D(i)t)(x)\})= 0,\text{ and }\\
    &\lim_{m\ra \infty}\mu(\{x\in X_0\mid \beta(h_m)(x)\in \alpha(sC_D(g)t)(x)\})= 0.\end{split}
\end{equation}

\vspace{2mm}

Next we argue that the action $B \curvearrowright ^\beta X$ is weak mixing. We will prove this using the von Neumann framework. 

\vspace{2mm}

Let $\cM=L^{\infty}(X)\rtimes D$ and $\cN=L^{\infty}(X)\rtimes B$.
Denote by $(u_g)_{g\in D}\subset \cM$ and $(v_h)_{h\in B}\subset \cN$ the canonical unitaries. For $h\in D$ and $g\in B$, let $A_h^g=\{x\in X\mid h^{-1}\cdot x=g^{-1}\cdot x\}$. Consider the $*$-homomorphism $\pi\colon \cN\rightarrow  \cM $ given by $\pi(a)=a$ and $\pi(v_h)=\sum_{g\in  D}\mathbbm{1}_{A_h^g} u_g$,  for every $a\in L^{\infty}(X_0)$ and $h\in B$. We view $\cN$ as a subalgebra of $\cM$ by identifying it with $\pi(\cN)$. 

\vspace{2mm}

Since $B$ has property (T), $D$ is ICC and $D\curvearrowright X$ is a generalized Bernoulli action, by \cite[Theorem 8.3]{ipv10} there exists a unitary $u\in\cM$ such that $u\cL(B)u^*\subseteq \cL(H)$. For every $h\in H$ let $\tilde h=uhu^*$. 

\vspace{2mm}

Assume by contradiction $B\curvearrowright X_0$ is not weak mixing. Then, one can find $C>0$ and $a_1,...,a_n,b_1,...,b_n\in L^{\infty}(X)\ominus\C 1$ such that for all $h\in B$, denoting $\tilde{h}=uhu^*$, we obtain

\begin{equation*}
     0<C\leq\sum_{i=1}^n|\tau(a_i\beta_h(b_i))|^2=\sum_{i=1}^n|\tau(a_ihb_ih^{-1})|^2=\sum_{i=1}^n|\tau((ua_iu^*) \tilde h (ub_iu^*)\tilde h^{-1})|^2.
\end{equation*}

\vspace{2mm}

Approximating $ua_iu^*$ and $ub_iu^*$, for $1\leq i\leq n$, we can find $a_i',b_i'\in L^{\infty}(X)$ supported on a finite nonempty set $F\subset I$, and $K\subset D$ finite such that

\begin{equation*}
    C\leq \sum_{i,\:k,l\in K}|\tau(a_i'\tilde h b_i'l\tilde h^{-1}k)|^2,\text{ for all }h\in B.
\end{equation*}

\vspace{2mm}

By letting $\tilde h=\sum_{g\in D}\tau( \tilde h g^{-1})g$, we obtain

\begin{align*}
    C&\leq \sum_{i,k,l}\left|\sum_{g,r\in D}\tau(\tilde h g^{-1})\overline{\tau(\tilde h r^{-1})}\tau(a_i'gb_i'lr^{-1}k)\right|^2=\sum_{i,k,l}\left|\sum_{g,r}\delta_{glr^{-1}k,e}\tau(\tilde h g^{-1})\overline{\tau(\tilde h r^{-1})}\tau(a_i'\alpha_g(b_i'))\right|^2\\
    &=\sum_{i,k,l}\left|\sum_{g\in D}\tau(\tilde h g^{-1})\overline{\tau(\tilde h l^{-1}g^{-1}k^{-1})}\tau(a_i'\alpha_g(\beta_i'))\right|^2.
\end{align*}

\vspace{2mm}

Let $N(F)=\{g\in D:g\cdot F=F\}$ and observe that $\tau(a_i'\alpha_g(b_i'))=0$ for $g\in D\setminus N(F)$. Thus, continuing from the prior equation, we obtain

\begin{align}\label{weakmix}
    0<C&\leq \sum_{r,k,l}\left|\sum_{g\in N(F)}\tau(\tilde h g^{-1})\overline{\tau(\tilde h l^{-1}g^{-1}k^{-1})}\tau(a_i'\alpha_g(b_i'))\right|^2\nonumber\\
    &\leq\sum_{i,k,l}\|a_i'\|^2_2\cdot\|b_i'\|^2_2\left(\sum_{g\in N(F)}|\tau(\tilde h g^{-1})|^2\right)\left(\sum_{g\in N(F)}|\tau(\tilde h l^{-1}g^{-1}k^{-1})|^2\right)\nonumber\\
    &\leq\sum_{i,k,l}\|a_i'\|^2_2\cdot\|b_i'\|^2_2\cdot\|\mathbb{P}_{N(F)}(\tilde h )\|_2^2\cdot\|\mathbb{P}_{lN(F)k}(\tilde h)\|_2^2.
\end{align}

\vspace{2mm}

Next observe that $N(F)\subseteq\bigcup_{i} g_i\text{Stab}_D(i)$, for finitely many $i$'s. Since $u\cL(B)u^*$ has property (T) and all stabilizers $\text{Stab}_D(i)$ are amenable, by Popa's intertwining techniques, there exists $(h_m)_m\subset B$ such that $\|\mathbb{P}_{N(F)}(\tilde h_m)\|_2,\:\|\mathbb{P}_{lN(F)k}(\tilde h_m)\|_2^2\to 0$ for every $l,k\in K\subset D$. This contradicts \eqref{weakmix}. Thus, the action $B\curvearrowright X_0$ is weak mixing.

\vspace{2mm}

Altogether, these show that the conditions of \cite[Theorem 4.1]{cios22} are satisfied and therefore one can find a subgroup $D_1\leqslant D$,  a group isomorphism $\delta:B\rightarrow D_1$ and $\theta\in [\mathscr R( D\curvearrowright X)]$ such that
 
$\theta\circ\beta(h)=\alpha(\delta(h))\circ\theta$, for every $h\in B$. 

\vspace{2mm}

Now let $p= \mathds 1_{X_0}$. In particular, the prior relation implies that there exists a unitary $u\in \mathscr N_{\mathcal M}(L^\infty(X))$ such that

\begin{equation}\label{equalalg}
    u (L^\infty(X)\rtimes_\beta B)u^*= L^\infty(X)\rtimes_\alpha \delta(B).
\end{equation} 

\vspace{2mm}

Let $q= upu^*\in L^{\infty}(X)$. Since from assumption $p (L^\infty(X)\rtimes B) p\subset p\mathcal M p$ has finite index then the inclusion $up (L^\infty(X)\rtimes B) pu^*\subset up\mathcal M pu^*$ also has finite index. Thus, using relation \eqref{equalalg}, we get that 

$$q (L^\infty(X)\rtimes_\alpha \delta(B)) q=up (L^\infty(X)\rtimes_\beta B) pu^*\subseteq up\mathcal M pu^*= q (L^\infty(X)\rtimes_\alpha D)q$$

\vspace{2mm}

also has finite index.  This however entails that $D_1=\delta(B)\leqslant D$ has finite index, as desired.\end{proof}

\begin{rem}\label{remark7.3} We note that the assumption $C_D(g)$ amenable for every $g\neq 1$ is only needed for the second equation in \eqref{measuretozero}.
\end{rem}

We are now ready to prove \ref{theoremA}. Given that our von Neumann algebra is a tensor between its center and $\cL(G)$ we will often work in the fibers of its integral decomposition. Within this context, we find ourselves working inside the factor $\cL(G)$.

\begin{proof}[\textbf{Proof of \ref{theoremA}}]  Let $\mathcal Z:= \mathscr Z(\mathcal M)$. Observe that $\cL(A^{(I)})\subseteq\cL(G)$ and $\mathcal{Z}\:\overline{\otimes}\:\cL(A^{(I)})\subseteq\cM$ are Cartan subalgebras.

\vspace{2mm}

First we will prove that $\cL(H^{hfc})\prec^s \mathcal Z$. Using \cite[Lemma 2.4]{dhi19} it suffices to show that $\cL(H^{hfc})z\prec \mathcal Z$, for every projection $z\in \mathcal Z $.
Fix $z \in \mathcal Z$ and assume by contradiction $\cL(H^{hfc})z\not\prec\mathcal{Z}$.

\vspace{2mm}

Since $H^{hfc}$ is amenable and normal, Claim \ref{step1} implies $\cL(H^{hfc})\prec_{\cM}^s\mathcal{Z}\:\overline{\otimes}\:\cL(A^{(I)})$. By Theorem \ref{theoremone}, $H^{hfc}$ is virtually abelian.

\begin{claim}\label{amenablerelcom} $(\cL(H^{hfc})'\cap\cM) z $ is amenable.
\end{claim}

\begin{subproof}[Proof of Claim \ref{amenablerelcom}] By the strong intertwining, for every $z_0\in\mathcal{P}((\cL(H^{hfc})'\cap\cM)z)$ there exists nonzero projections $p\in\cL(H^{hfc})z_0$, $q\in\mathcal{Z}\:\overline{\otimes}\:\cL(A^{(I)})$, a partial isometry $v\in q\cM p$ and a $*$-isomorphism onto its image $\Psi:p\cL(H^{hfc})p\to q(\mathcal{Z}\:\overline{\otimes}\:\cL(A^{(I)}))q$ such that $\Psi(x)v=vx$ for all $x\in p\cL(H^{hfc})p$. Denote $\mathcal{Q}:=\Psi(p\cL(H^{hfc})p)$. Since $\cL(H^{hfc})z\not\prec\mathcal{Z}$, then $\mathcal{Q}\not\prec\mathcal{Z}$ by \cite[Remark 3.8]{vaes08}. Moreover, $\mathcal{Q}vv^*=vp\cL(H^{hfc})pv^*$ and taking relative commutants we have, 

\begin{equation}\label{relativecommequal}
    vv^*(\mathcal{Q}'\cap q\cM q)vv^*=v(p\cL(H^{hfc})'p\cap p\cM p)v^*.
\end{equation}

\vspace{2mm}

In order to prove the claim, we first show that $\mathcal{Q}'\cap q\cM q$ is amenable. Let $\tilde{\mathcal{Q}}:=\mathcal{Q}\vee \mathcal{Z}q$ and observe that $\tilde{\mathcal{Q}}'\cap q\cM q=\mathcal{Q}'\cap q\cM q$. Since $\mathcal{Q}\not\prec\mathcal{Z}$, then also $\tilde{\mathcal{Q}}\not\prec\mathcal{Z}$. Moreover,

\begin{equation}\label{inclusionofQ}
    \mathcal{Z}q\subset\tilde{\mathcal{Q}}\subseteq(\mathcal{Z}\:\overline{\otimes}\:\cL(A^{(I)}))q\subseteq \tilde{\mathcal{Q}}'\cap q\cM q\subset q\cM q.
\end{equation}

\vspace{2mm}

Consider the integral decomposition of $\cM$ over its center $\mathcal{Z}=L^{\infty}(X)$; that is, $\cM=\int_X^{\oplus}\cL(G)d\mu(x)$, and decompose $q=\int_Xq_xd\mu(x)$, $\tilde{\mathcal{Q}}=\int_X^{\oplus}\mathcal{Q}_xd\mu(x)$. From \eqref{inclusionofQ} and \cite[Theorem IV.8.18]{takesaki} it follows that

\begin{equation*}
    \int_X^{\oplus}\mathcal{Q}_xd\mu(x)\subseteq \int_X\cL(A^{(I)})q_xd\mu(x)\subseteq \int_X^{\oplus}\mathcal{Q}_x'\cap q_x\cL(G)q_x\:d\mu(x),
\end{equation*}

\vspace{2mm}

or, $\mathcal{Q}_x\subseteq \cL(A^{(I)})q_x\subseteq \mathcal{Q}_x'\cap q_x\cL(G) q_x$ for almost every $x\in X$. 

\vspace{2mm}

By Lemma \ref{diffusefibers}, since $\tilde{\mathcal{Q}}\not\prec \mathcal{Z}$, we have that the fibers $\mathcal{Q}_x$ are diffuse for almost every $x\in X$. Hence, by \cite[Corollary 4.6]{cios22}, $\mathcal{Q}_x'\cap q_x\cL(G)q_x$ is amenable for almost every $x\in X$. Thus, \cite[Proposition 6.5]{connes76} implies that $\mathcal{Q}'\cap q\cM q$ is amenable. 

\vspace{2mm}

Therefore, using \eqref{relativecommequal} we obtain $v^*v(\cL(H^{hfc})'\cap \cM)v^*v$ is amenable; and if $z'$ denotes the central support of the projection $v^*v$ in $(\cL(H^{hfc})'\cap \cM)z$, then $(\cL(H^{hfc})'\cap \cM)z'$ is amenable. Note that $v^*v\leq p\leq z_0$, meaning that $z_0z'\neq 0$. In conclusion, we have proved that every corner of $(\cL(H^{hfc})'\cap\cM)z$ has a nontrivial amenable subcorner. Hence, $(\cL(H^{hfc})'\cap\cM)z$ is amenable.
\end{subproof}

Let $\mathcal{D}=\cL(H^{hfc})$, and $\mathcal{D}=\int_X^{\oplus}\mathcal{D}_xd\mu(x)$ be its integral decomposition. Let $z=\int_X^{\oplus}z_xd\mu(x)$ be the integral decomposition of $z$. By \cite[Theorem IV.8.18]{takesaki} we have $(\mathcal{D}'\cap \cM)z=\int_X^{\oplus}(\mathcal{D}_x'\cap \cL(G))z_xd\mu(x)$.

\vspace{3mm}

By \cite[Theorem 6.1 - Claim 3]{ioa10}, $\mathcal{D}_x'\cap\cL(G)$ is type I and there exists $v_x\in\mathscr{U}(\cL(G)z_x)$ such that $\mathcal{D}_xz_x\subseteq v_x\cL(A^{(I)})v_x^*z_x\subseteq (\mathcal{D}_x'\cap\cL(G))z_x$ for almost every $x\in X$.

\vspace{2mm}

Fix $h\in H$. As $h$ normalizes $\mathcal{D}z$ and $(\mathcal{D}'\cap\cM)z$, we obtain $h_x\mathcal{D}_xz_xh_x^{-1}=\mathcal{D}_xz_x$ and $h_x(\mathcal{D}_x'\cap \cL(G))z_xh_x^{-1}=(\mathcal{D}_x'\cap\cL(G))z_x$ for almost every $x\in X$. Now let $K_x:= H_x\cap (\mathcal D_x'\cap \cL (G))z_x$ and notice that $H^{hfc}_x\leqslant K_x\lhd H_x$ is a normal subgroup. Moreover, since $\mathcal D_x'\cap \cM$ is type I it follows that $K_x$ is virtually abelian. This implies there is a natural action  $H_x/K_x\curvearrowright^\alpha \mathcal D_x$ by conjugation, $\alpha_{s_x} ={\rm Ad}(s_x)$. Using \cite{ioa10}, one 
can find unitaries $w_{s,x} \in (\mathcal D_x'\cap \cL (G))z_x$ and an action $\beta_{s_x} =  {\rm Ad}( s_x w_{s,x})$ of $H_x/K_x$ on $v_x \cL(A^{(I)})v_x^*z_x$ which leaves $\mathcal D_xz_x$ invariant. Since $\cL(A^{(I)})\subset \cL(G)$ is a Cartan subalgebra, from the choice of $K_x$ it follows that $\beta$ is free.




\vspace{2mm}

Now, consider the wreath-like product free action $v_xBv_x^*z_x\curvearrowright^{\sigma_x}v_x\cL(A^{(I)})v_x^*z_x$. Define $\mathcal{P}_x=v_x\cL(A^{(I)})v_x^*z_x\vee\{s_xw_{s,x}:s_x\in H_x/K_x\}''\subseteq v_x\cL(G)v_x^*z_x=\cL(G)z_x$. Observe that $v_x\cL(A^{(I)})v_x^*z_x\subseteq \mathcal{P}_xz_x$ is a masa and $s_xw_{s,x}\in\mathscr{N}_{v_x\cL(G)v_x^*z_x}(v_x\cL(A^{(I)})v_x^*z_x)$ for all $s_x\in H_x/K_x$.

\vspace{0.5mm}

\begin{claim}\label{finiteindexequivrel2} For every $\delta>0$ there is a projection $p\in v_x \mathcal L(A^{(I)})v_x^*z_x$ with $\tau(p)\geq 1-\delta$ such that $p \mathcal P_x p \subseteq p v_x \cL(G)v_x^*z_x p$ has finite index. 
\end{claim}

\begin{subproof}[Proof of Claim \ref{finiteindexequivrel2}] Since $v_x\cL(A^{(I)})v_x^*z_x\subseteq (\mathcal{D}_x'\cap\cL(G))z_x$ is a Cartan subalgebra and $\mathcal{D}_x'\cap \cL(G)$ is type I, for every $\delta>0$ there exists a projection $z_1\in\mathscr{Z}((\mathcal{D}_x'\cap\cL(G))z_x)$ and $a_1,...,a_l\in\mathscr{U}((\mathcal{D}_x'\cap\cL(G))z_1)$ such that for all $t\in (\mathcal{D}_x'\cap\cL(G))z_x$ we have that

\begin{enumerate}
    \item[(1)] $\tau(z_1)\geq 1-\delta$, and
    \item[(2)] $tz_1=\sum_{i=1}^l\mathbb{E}_{v_x\cL(A^{(I)})v_x^*z_x}(tz_1a_i^*)a_i$. 
\end{enumerate}

Fix $\epsilon>0$ and $y\in(\cL(G)z_x)_1$. Since $\{h_x:h\in H\}''z_x=\cL(G)z_x$ and each $h_x$ can be written as $h_x=k_xs_x$ for $k_x\in K_x$ and $s_x\in H_x/K_x$, one can find finitely many $c_{k_xs_x}\in\C$ with

\begin{equation*}
    \left\|y-\sum_{k_xs_x\in F_y}c_{k_xs_x}s_xk_x\right\|_2\leq \epsilon,
\end{equation*}

\vspace{2mm}

where $F_y\subset H_x$ is a finite set. Now let $z_y^i=\sum_{k_xs_x\in F_y}c_{k_xs_x}s_xw_{s,x}\mathbb{E}_{v_x\cL(A^{(I)})v_x^*z_x}(w_{s,x}^*k_xz_1a_i^*)$. Observe that $z_y^i\in\mathcal{P}_x$ for every $1\leq i\leq l$. Using property (2) above with $t=w_{s,x}^*k_x$ for all $k_xs_x\in F_y$, we obtain

\begin{align*}
    \left\|yz_1-\sum_{i=1}^lz_y^ia_i\right\|_2&=\left\|yz_1-\sum_{i=1}^l\sum_{k_xs_x}c_{k_xs_x}s_xw_{s,x}\mathbb{E}_{v_x\cL(A^{(I)})v_x^*z_x}(w_{s,x}^*k_xz_1a_i^*)a_i\right\|_2\\
    &=\left\|yz_1-\sum_{k_xs_x}c_{k_xs_x}s_xk_xz_1\right\|_2\leq \epsilon.
\end{align*}

\vspace{2mm}

Thus, $\cL(G)z_1\subseteq\overline{\sum_{i=1}^l\mathcal{P}_xa_i}$ and hence $L^2(z_1\cL(G)z_1)$ is finitely generated as a left $z_1\mathcal{P}_xz_1$-module. By \cite[Theorem 2.2]{bccmsw23}, there exists $z_n\in z_1\cL(G)z_1\cap z_1\mathcal{P}_x'z_1\subseteq v_x\cL(A^{(I)})v_x^*z_1$ with $z_n\to z_1$ in SOT and $r_{n,1},...,r_{n,l}\in z_1\cL(G)z_1$, $z_1\mathcal{P}_xz_1$ orthogonal, with

\begin{equation*}
    z_ntz_n=\sum_{i=1}^l\mathbb{E}_{z_1\mathcal{P}_xz_1}(z_ntz_nr_{n,i}^*)r_{n,i}, \text{ for every }t\in z_1v_x\cL(G)v_x^*z_1.
\end{equation*}

\vspace{2mm}

Hence, $z_n\mathcal{P}_xz_n\subseteq z_nv_x\cL(G)v_x^*z_n$ has finite index for all $n$. The claim follows by letting $p=z_n$ for $n$ sufficiently large for which $\tau(z_n)\geq 1-\delta$.
\end{subproof}

Both von Neumann algebras $\mathcal{P}_x\subseteq v_x\cL(G)v_x^*z_x$ have the same Cartan subalgebra, $v_x\cL(A^{(I)})v_x^*z_x$. This fact combined with the prior claim shows that the inclusion of the equivalence relations $\mathscr{R}(H_x/K_x\curvearrowright v_x\cL(A^{(I)})v_x^*z_x)\subseteq \mathscr{R}(v_xBv_x^*z_x\curvearrowright v_x\cL(A^{(I)})v_x^*z_x)$ has many nontrivial restrictions of finite index. By \cite[Lemma 3.6]{dhi19} $H_x/K_x$ and $B$ are measure equivalent, and since $B$ has property (T), then so does $H_x/K_x$. Since the action $B\curvearrowright I$ has amenable stabilizers, $C_B(g)$ are infinite index in $B$, for $g\neq 1$, and using a similar argument as in \cite[Theorem 1.3 Step 3]{cios22} we get that \eqref{measuretozero} holds. Therefore, Theorem \ref{SOE} together with Remark \ref{remark7.3} show that there exists an injective group homomorphism $\delta_x:H_x/K_x\to B$ with finite index image, 
and $\theta_x\in[\mathscr{R}(v_xBv_x^*z_x\curvearrowright v_x\cL(A^{(I)})v_x^*z_x)]$ such that

\begin{equation}\label{equivactions}
    \sigma_{s_x}\circ\theta_x=\theta_x\circ \alpha_{\delta_x(s_x)},\text{ for all }s_x\in H_x/K_x\text{ and for almost every }x\in X.
\end{equation}

\vspace{2mm}

Let $u_x\in\mathscr{N}_{v_x\cL(G)v_x^*z_x}(v_x\cL(A^{(I)})v_x^*z_x)$ such that $u_xau_x^*=a\circ \theta_x^{-1}$ for all $a\in v_x\cL(A^{(I)})v_x^*z_x$. By \eqref{equivactions} we can find $(\zeta_{h,x})_x\in \mathscr{U}(v_x\cL(A^{(I)})v_x^*z_x)$ with

\begin{equation}\label{hxwxh=delta}
    u_x\widehat{s_x}w_{h,x}u_x^*=\zeta_{s,x}\delta_x(s_x), \text{ for all }s_x\in H_x/K_x.
\end{equation}

\vspace{2mm}

Fix $h\in H^{fc}$. Since $H^{fc}$ is FC there exists $h=:h_1,...,h_n\in H^{fc}$ and an action $H\curvearrowright\{1,...,n\}$ such that for every $g\in H$

\begin{equation}\label{torepeat}
    gh_ig^{-1}=h_{g\cdot i}.
\end{equation}

\vspace{2mm}

Thus, if we multiply the prior equation by $z$, we get $gh_izg^{-1}=h_{g\cdot i}z$, and considering the integral decomposition over the center, we obtain

\begin{equation*}
    g_x(h_iz)_xg_x^{-1}=(h_{g\cdot i}z)_x,
\end{equation*}

\vspace{2mm}

for almost every $x\in X$ and every $i=1,...,n$. Since $w_{s,x}\in (\mathcal{D}_x'\cap \cL(G))z_x\subset(\cL(H^{fc})_x'\cap\cL(G))z_x$ and using \eqref{hxwxh=delta} we get that for all $s_x\in H_x/K_x$

\begin{align*}
    u_x(h_{\widehat{s_x}\cdot i}z)_xu_x^*&=u_x\widehat{s_x}(h_iz)_x\widehat{s_x}^*u_x^*=u_x\widehat{s_x}w_{h,x}(h_iz)_xw_{h,x}^*\widehat{s_x}^*u_x^*\\
    &=\zeta_{s,x}\delta_x(s_x)u_x(h_iz)_xu_x^*\delta_x(s_x)^*\zeta_{s,x}^*.
\end{align*}

\vspace{2mm}

Moreover, since $u_x\in\mathscr{N}_{\cL(G)z_x}(v_x\cL(A^{(I)})v_x^*z_x)$, $v_x\cL(A^{(I)})v_x^*z_x$ is abelian and $\cL(H^{fc})_xz_x\subset\mathcal{D}_xz_x\subset v_x\cL(A^{(I)})v_x^*z_x$, we further get that

\begin{equation}\label{fixedpointsalphaij}
    u_x(h_{s_x\cdot i}z)_xu_x^*=\delta_x(s_x)u_x(h_iz)_xu_x^*\delta_x(s_x)^*.
\end{equation}

\vspace{2mm}

For every finite index subgroup $B_1\leqslant B$, the action $B_1\curvearrowright^{\alpha_{x,y}}\cL(A^{(I)})$ given by $\alpha_x(b)(c)=\text{Ad}(\delta_x^{-1}(b))(c)$, for $c\in \cL(A^{(I)})$, is weak mixing over $\cL(G)$. The relation \eqref{fixedpointsalphaij} shows that the action $\alpha_x$ has an invariant finite dimensional subspace, meaning $\text{Span}\{u_x(h_iz)_xu_x^*:i=1,...,n\}$. Since the action is weak mixing, $(h_iz)_x\in\C$ for almost every $x\in X$ and every $i=1,...,n$. This means $h_iz\in\mathscr{Z}$ for all $i$. In particular $hz\in\mathscr{Z}$. Since $h\in H^{fc}$ was arbitrary, $\cL(H^{fc})z\subset\mathscr{Z}$.

\vspace{2mm}

Now, consider the subgroup $H^{fc}\subset H_2^h$ of $H^{hfc}$. We are going to repeat the last part to show $\cL(H_2^h)z$ is contained in $\mathscr Z$. Fix $h\in H^h_2$. Since $H^h_2$ is FC over $H^{fc}$, we can find finitely many elements $h=:h_1,...,h_n\in H^h_2$, an action $H\curvearrowright \{1,...,n\}$ and a map $\alpha:\{1,...,n\}\times H\to H^{fc}$ such that for all $g\in H$

\begin{equation*}
    gh_ig^{-1}=\alpha(i,g)h_{g\cdot i}.
\end{equation*}

\vspace{2mm}

If we follow the same argument after equation \eqref{torepeat} and decompose the prior equation over $\mathscr{Z}$, we obtain 

\begin{equation*}
    g_x(h_iz)_xg_x^{-1}=(\alpha(i,g)z)_x(h_{g\cdot i}z)_x,
\end{equation*}

\vspace{2mm}

for almost every $x\in X$ and every $i$. Since $\cL(H^{fc})z\subset\mathscr{Z}$, $(\alpha(i,g)z)_x$ is a scalar, and using \eqref{hxwxh=delta} we obtain

\begin{equation*}
    \delta_x(s_x)(u_x(h_iz)_xu_x^*)\delta_x(s_x)^*=(\alpha(i,\widehat{s_x})z)_xu_x(h_{\widehat{s_x}\cdot i}z)_xu_x^*.
\end{equation*}

\vspace{2mm}

Since the action given by conjugation induced by $\delta_x$ is weak mixing, $hz\in\mathscr{Z}$. Therefore, $\cL(H^h_2)z\subset\mathscr{Z}$. By induction, $\cL(H^{hfc})z\subset\mathscr{Z}$ which contradicts the assumption. Therefore, $\cL(H^{hfc})\prec^s\mathscr{Z}$ as wanted. In particular, $H^{fc}$ is finite index in $H^{hfc}$ and $\cL(H^{fc})\prec^s\mathscr{Z}$.

\vspace{3mm}

By Theorem \ref{relativecommutatorgroup}, the subgroup $H^{fc}$ is BFC inside $H$. In particular, $H^{fc}$ is a BFC-group itself and by B. H. Neumann's result \cite{bhneumann} we have that the commutator $F:=[H^{fc},H^{fc}]$ is finite. Since $H^{fc}$ is normal in $H$ it follows that $F$ is also normal in $H$. Consider the central projection $p= |F|^{-1}\sum _{h \in F}u_h\in \mathcal Z$. Then clearly  $\cL(H^{fc})p \subseteq \mathcal M p$ is a regular, abelian von Neumann subalgebra such that $\mathcal Z p\subseteq \cL(H^{fc})p$. Therefore the remaining of the conclusion follows from Theorem \ref{tinyinclusion}.
\end{proof}


\subsection{Cocycle crossed product for \texorpdfstring{$\cL(H/F)$}{L(H/F)}}\label{sectionprop(T)} On the proceeding part, we will describe the structure of the group von Neumann algebra $\cL(H)$ in the context of \ref{theoremA}.

\vspace{2mm}

To simplify notation, we will use the symbol $\tau$ to represent the canonical trace on $\cL(H)$. The trace on the compression $p\cL(H)p$, where $p$ is a non-zero projection in $\cL(H)$, is denoted by $\tau_p$ and given by $\tau_p(pxp) = \tau(pxp) / \tau(p)$, for $x\in\cL(H)$. Additionally, we will denote the restriction of $\tau$ on $\cL(H^{hfc})$ as $\tilde{\tau}$, and the trace on the compression $q\cL(H^{hfc})q$, for $q$ a non-zero projection in $\cL(H^{hfc})$, will be represented as $\tilde{\tau}_q$.

\begin{prop}\label{cutdown} Let $F\ngroup H$ be groups and assume that $F$ is finite. Let $p=\frac{1}{|F|}\sum_{f\in F}u_f$. Then, $p$ is a projection in $\mathscr{Z}(\cL(H))\cap \cL(F)$ and $\cL(H)p\cong \cL(H/F)$.
\end{prop}

\begin{proof} The operator $p=\frac{1}{|F|}\sum_{g\in F}u_g$ is a central projection. Indeed, $p=p^*$,

\begin{equation*}
    p^2=\frac{1}{|F|^2}\sum_{g,h\in F}u_{gh}=\frac{1}{|F|^2}\cdot|F|\sum_{g\in F}u_g=p,
\end{equation*}

\vspace{2mm}

and $p$ is central because $F$ is normal in $H$. The trace of the projection $p$ is given by

\begin{equation*}
    \tau(p)=\frac{1}{|F|}\sum_{f\in F}\langle u_f\delta_e,\delta_e\rangle=\frac{1}{|F|}.
\end{equation*}

\vspace{2mm}

Fix a set of coset representatives of $H/F$. Notice that if $g_1,g_2\in H$ with $g_1F=g_2F$, then $g_2=g_1f$ for some $f\in F$. This implies that $u_{g_2}p=u_{g_1}u_fp=u_{g_1}p$. Here we have used that

\begin{equation*}
    u_fp=\frac{1}{|F|}\sum_{g\in F}u_{fg}=\frac{1}{|F|}\sum_{g\in F}u_g=p.
\end{equation*}

\vspace{2mm}

Thus, the subset $\{u_gp: g\text{ a coset representative of } H/F\}$ of $\cL(H)p$ is a group under multiplication and it does not depend on the choice of representatives of $H/F$. This shows that the map $\cL(H)p\to \cL(H/F)$ given by $u_gp\mapsto u_{l_g}$, where $g=l_gf_g$ for $l_g$ in a set of representatives of $H/F$ and $f_g\in F$, is well-defined. It is also surjective and trace-preserving. Indeed, 

\begin{equation*}
    \tau_p(u_gp)=\frac{\tau(u_gp)}{\tau(p)}=|F|\cdot\tau(u_gp)=\sum_{f\in F}\tau(u_{gf})=\delta_{g,F}=\delta_{l_g,e}.
\end{equation*}

\vspace{2mm}

Hence, we conclude that $\cL(H)p\cong\cL(H/F)$.
\end{proof}

\vspace{2mm}

Let $H/H^{hfc}\curvearrowright^{\alpha,c}H^{hfc}$ be a cocycle action, and consider the induced cocycle action $H/H^{hfc}\curvearrowright^{\beta,\tilde{c}}\cL(H^{hfc})p$ where $\beta$ is induced from $\alpha$ and $\tilde{c}(g_1,g_2)=u_{c(g_1,g_2)}p$ for $g_1,g_2\in H/H^{hfc}$.

\vspace{2mm}

\begin{prop}\label{cocyclecrossprodisomorphism} Let $H$ be a countable group. Assume $F\subset H^{hfc}$ is a finite normal subgroup of $H$. Then, 
$\cL(H)p\cong\cL(H^{hfc}/F)\rtimes_{\beta,\tilde{c}}H/H^{hfc}$ for the cocycle action $(\beta,\tilde{c})$ defined above.
\end{prop}

\begin{proof} 
Define $\Phi:(\cL(H^{hfc})p\rtimes_{\tilde{c}}H/H^{hfc},\tilde{\tau}_p)\to (\cL(H)p,\tau_p)$ by $u_gpu_{r}\mapsto u_{gr}p$, and observe it is well-defined and trace preserving. Indeed, for $g\in H^{hfc}$ and $r\in H/H^{hfc}$ we have

\begin{equation*}
    \tau(u_{gr}p)=\langle u_{gr}p\delta_e,\delta_e\rangle=\frac{1}{|F|}\sum_{f\in F}\langle\delta_f,\delta_{r^{-1}g^{-1}}\rangle=\frac{1}{|F|}\delta_{gr,F}=\frac{1}{|F|}\delta_{r,e}\delta_{g,F}.
\end{equation*}

\vspace{2mm}

This tells us that

\begin{equation*}
    \tilde{\tau}_p(u_gpu_{r})=\tilde{\tau}_p(u_gp)\delta_{r,e}=\delta_{g,F}\delta_{r,e}=\tau_p(u_{gr}p).
\end{equation*}

\vspace{2mm}

It is also an algebra homomorphism because 

\begin{align*}
    \Phi(u_{g_1}pu_{r_1}u_{g_2}pu_{r_2})&=\Phi(u_{g_1\alpha_{r_1}(g_2)c(r_1,r_2)}pu_{r_1r_2})=u_{g_1\alpha_{r_1}(g_2)c(r_1,r_2)r_1r_2}p\\
    &=u_{g_1r_1}pu_{g_2r_2}p=\Phi(u_{g_1}pu_{r_1})\Phi(u_{g_2}pu_{r_2}).
\end{align*}

\vspace{2mm}

Since the map is surjective, we have $\cL(H)p\cong\cL(H^{hfc})p\rtimes_{\tilde{c}}H/H^{hfc}$. The fact that $\cL(H^{hfc})p\cong\cL(H^{hfc}/F)$ follows from Proposition \ref{cutdown}.
\end{proof}

\begin{lem}\label{trivialabelianization} Let $K$ be a countable discrete ICC group, and let $W$ be a countable discrete group with property (T) that satisfies Jones' automorphism conjecture \cite[Problem 8]{jones} (see also \cite[Section 3]{popa07}). Assume $W$ has trivial abelianization and $\text{Out}(W)$ is torsion free. Assume also $\cL(W)\cong\cL_{c}(K)$ for a $\mathbb{T}$-valued 2-cocycle $c$. Then, $K$ has trivial abelianization.
\end{lem}

\begin{proof} Let $\Psi$ be the $\ast$-isomorphism between $\cL(W)$ and $\cL_{c}(K)$. First note that $K$ has property (T) by \cite{wata,connesjones}. Assume, by contradiction, $K$ does not have trivial abelianization. Then, there exists a non-trivial multiplicative character $K/[K,K]\to\mathbb{T}$. Let $\xi$ denote the character of $K$ obtained by the composition of the projection map and the character on the quotient. Let $\Theta_{\xi}:\cL_{c}(K)\to\cL_{c}(K)$ be given by $\Theta_{\xi}(v_h)=\xi(h)v_h$ and extend linearly. It is not hard to see that $\Theta_{\xi}$ is an automorphism. This means that $\Psi^{-1}\circ\Theta_{\xi}\circ\Psi$ belongs to $\text{Out}(\cL(W))$, which is equivalent to $\text{Out}(W)$ by our assumptions on $W$. Moreover, if we let $[K:[K,K]]=:n<\infty$, for $n\in\N$, we have that $\Theta_{\xi}^n=1$. This implies that $\Psi^{-1}\circ\Theta_{\xi}\circ\Psi\in\text{Out}(W)$ is an outer automorphism of finite order. This contradicts one of our assumptions on $W$, and hence, $K$ has trivial abelianization.
\end{proof}

\section{Rigidity results for twisted group von Neumann algebras}

In this section we present several rigidity results for twisted group factors (\ref{notrivialcompressions} and Theorem \ref{untwistcocycle}) which extend some results in the literature. While these generalizations are used in an essential way in the proofs of the main results of the paper, they can also be viewed as results of independent interest. 

\subsection{Strong W\texorpdfstring{$^*$}{*}-rigidity results for twisted wreath-like product group von Neumann algebras} In this subsection we establish a couple $W^*$-strong rigidity results which generalize some of the recent rigidity results for wreath-like product groups factors \cite[Theorem 7.3]{cios22b} to their twisted versions (see Theorem \ref{untwistcocycle}). To properly introduce them we first need to establish some notations along with some useful facts.

\vspace{2mm}

Let $B\curvearrowright I$ be a faithful action with infinite orbits. Any group $G\in\mathcal{WR}(A,B\curvearrowright I)$ and any $2$-cocycle $c:G\times G\to\mathbb{T}$ naturally give rise to an action $G \curvearrowright^{\sigma} \cL(A^{(I)})$ which is malleable in the sense of Popa and weak mixing as follows. Assume $A$ is an abelian free group and let $c:G\times G\to\mathbb{T}$ be a 2-cocycle. Note since $A$ is free, we may assume that $c|_{A^{(I)}\times A^{(I)}}=1$. Let $g\in G$ and $\otimes_{i\in I}a_i\in \cL(A^{(I)})$, and define $\sigma_g(\otimes_{i\in I}a_i)=c(g,(a_i)_i)c(g(a_i)_i,g^{-1})\otimes_{i\in I}a_{g^{-1}\cdot i}$. When $A$ is a free abelian group, the action $\sigma$ naturally extends to an action of the quotient group $B=G/A^{(I)}\curvearrowright\cL(A^{(I)})$.

\begin{lem}\label{weakmixingaction} Assume $A$ is an abelian free group. Then 

\begin{enumerate}
    \item[(i)] The action $\sigma|_B$ is weak mixing.
    \item[(ii)] The action $\sigma$, as defined above, is weak mixing and malleable in the sense of Popa (see for instance \cite{popa06,popashifts,popastrongrigII}).
    \item[(iii)] (\cite[Lemma 2.5]{patchell}) The action  $\sigma$ is free.  
\end{enumerate}
\end{lem}

\begin{proof} Assume $A=\Z$ for simplicity, and let $a_1$ be the generator of $\cL(A)$. First we show the action is weak mixing. Since the proof for weak mixingness is similar in both cases (i) and (ii), we only show it for (ii); that is, we show there exists an infinite sequence $(g_n)_n\subset G$ such that $\lim_n\tau(\sigma_{g_n}(x)y)=0$, for all $x,y$. It suffices to show this on a set of generators. Let $(a_1^{k_i})_i,(a_1^{l_i})_i\in \cL(A^{(I)})$. Note that,

\begin{equation*}
    \tau(\sigma_{g_n}((a_1^{k_i})_i)(a_1^{l_i })_i)=c(g_n,(a_{1}^{k_i})_i)c(g_n(a_{1}^{k_i})_i,g_n^{-1})\prod_i\tau(a_{1}^{k_{g_n^{-1}.i}}a_1^{l_i})
\end{equation*}

\vspace{2mm}

However, since $k_i,l_i$ are zero for all but finitely many $i$ and the action $G\curvearrowright I$ has infinite orbits, we can find a sequence $(g_n)_n$ that escapes all $\text{Stab}_G(i)$ for all $i$. Thus, 

\begin{equation*}
    \lim_n\tau(a_1^{k_{g_n^{-1}.i}}a_1^{l_i})=0.
\end{equation*}

\vspace{2mm}

Next, we show $\sigma$ is malleable. Let $a_1$ and $a_2$ be generating Haar unitaries for $\cL(A)$, and define $\alpha_t^0\in \text{Aut}(\cL(A)\otimes\cL(A))$ by $\alpha_t^0(a_1)=a_1^{1-t}a_2^t$, $\:\alpha_t^0(a_2)=a_2^{1+t}a_1^{-t}$, for $t\in\R$. Note that $a_1^{1-t}a_2^t$ and $a_2^{1+t}a_1^{-t}$ are also generating Haar unitaries for $\cL(A)$. Also, $\alpha_0^0=\text{id}$, $\alpha_1^0(a_1)=a_2$ and hence $\alpha_1^0(\cL(A)\otimes\C)=\C\otimes\cL(A)$. Define $\alpha_t:=\otimes_{i\in I}\alpha_t^0\in\text{Aut}(\cL(A^{(I)})\overline{\otimes}\cL(A^{(I)}))$. Note that $\alpha_t(\cL(A^{(I)})\otimes \C)=\C\otimes\cL(A^{(I)})$ and $\alpha_t\to\text{id}$ pointwise.

\vspace{2mm}

Next, we check that $\alpha_t\circ(\sigma_g\otimes\sigma_g)=(\sigma_g\otimes\sigma_g)\circ\alpha_t$ for $\in\R$ and $g\in G$. It again suffices to check the previous equality on generators of $A^{(I)}\times A^{(I)}$. Let $a_{1,i}^{k_1}, a_{2,i}^{k_2}\in A^{(I)}$ be the elements that at position $i$ they have $a_1^{k_1}$ (respectively $a_2^{k_2}$) and 1 everywhere else. Then,

\begin{align*}
    ((\sigma_g\otimes\sigma_g)\circ\alpha_t)(a_{1,i}^{k_1}a_{2,i}^{k_2})&=(\sigma_g\otimes\sigma_g)(a_{1,i}^{k_1(1-t)}a_{2,i}^{k_1t}a_{2,i}^{k_2(1+t)}a_{1,i}^{-tk_2})=(\sigma_g\otimes\sigma_g)(a_{1,i}^{k_1-k_1t-k_2t}a_{2,i}^{k_2+k_1t+k_2t})\\
    &=\sigma_g(a_{1,i}^{k_1-k_1t-k_2t})\otimes\sigma_g(a_{2,i}^{k_2+k_1t+k_2t}).
\end{align*}

\vspace{2mm}

Note that $\sigma_g(a_{1,i})=c(g,a_{1,i})c(ga_{1,i},g^{-1})a_{1,g^{-1}\cdot i}=:\lambda(g,i)a_{1,g^{-1}\cdot i}$. Also, for $k\in\N$, $ga^k_{1,i}g^{-1}=(ga_{1,i}g^{-1})^k$, because $c|_{A^{(I)}\times A^{(I)}}=1$. Thus, 

\begin{equation*}
    c(g,a_{1,i}^{k})c(ga_{1,i}^{k},g^{-1})=(c(g,a_{1,i})c(ga_{1,i},g^{-1}))^{k}=\lambda(g,i)^{k},
\end{equation*}

\vspace{2mm}

meaning that $\sigma_g(a_{1,i}^{k})=\lambda(g,i)^k\cdot a_{1,g^{-1}\cdot i}^k$, for any $k\in\Z$. Using Borel functional calculus, $\sigma_g(a_{1,i}^{t})=\lambda(g,i)^t\cdot a_{1,g^{-1}\cdot i}^t$ for all $t\in\R$. Therefore,

\begin{align*}
    ((\sigma_g\otimes\sigma_g)\circ\alpha_t)(a_{1,i}^{k_1}\:a_{2,i}^{k_2})&=\lambda(g,i)^{k_1-k_1t-k_2t}\:a_{1,g^{-1}\cdot i}^{k_1-k_1t-k_2t}\:\lambda(g,i)^{k_2+k_1t+k_2t}\:a_{2,g^{-1}\cdot i}^{k_2+k_1t+k_2t}\\
    &=\lambda(g,i)^{k_1+k_2}\:a_{1,g^{-1}\cdot i}^{k_1-k_1t-k_2t}\:a_{2,g^{-1}\cdot i}^{k_2+k_2t+k_1t},
\end{align*}

\vspace{2mm}

while on the other hand,

\begin{align*}
    (\alpha_t\circ(\sigma_g\otimes\sigma_g))(a_{1,i}^{k_1}\:a_{2,i}^{k_2})&=\alpha_t(\lambda(g,i)^{k_1}\:a_{1,g^{-1}\cdot i}^{k_1}\:\lambda(g,i)^{k_2}\:a_{2,g^{-1}\cdot i}^{k_2})\\
    &=\lambda(g,i)^{k_1+k_2}\:a_{1,g^{-1}\cdot i}^{k_1-k_1t-k_2t}\:a_{2,g^{-1}\cdot i}^{k_2+k_2t+k_1t}.
\end{align*}

\vspace{2mm}

Lastly, we prove $\sigma$ is a free action. By the faithfulness of the action $B\curvearrowright I$, for every $g\in G\setminus \{1\}$, the set $\{ i\,|\, g\cdot i\neq i\}$ is infinite \cite[Lemma 4.4]{cios22b}. Since $\cL(A)$ is diffuse, let $(u_n)_n\subset\mathscr{U}(\cL(A))$ be a sequence converging weakly to zero. Fix $g\in G$, and let $k,j\in I$ with $g\cdot k=j\neq k$. Let $y_n^k=u_n$ and $y_n^l=1$ for $k\neq l$, and $y_n=\otimes_{i\in I}y_n^i$. We claim that for all $a\in \cL(A^{(I)})$, $\langle ay_n,\sigma_g(y_n)a\rangle\to 0$. 

\vspace{2mm}

Fix $x=\otimes_ix_i$ (respectively  $z=\otimes_iz_i$) where all but finitely many $x_i$ (respectively $z_i$) are 1. Then, $xy_n=\otimes_{i\neq k}x_i\otimes x_ku_k$ and $\sigma_g(y_n)z=\otimes_{i\neq j}z_i\otimes z_j\sigma_g(y_n^k)$. Then, 

\begin{equation*}
    \tau(xy_nz^*\sigma_g(y_n)^*)=\tau(x_ku_nz_k)\tau(x_jz_j^*\sigma_g(y_n^k)^*)\prod_{i\neq k,j}\tau(x_iz_i^*)\to 0,
\end{equation*}

\vspace{2mm}

as $n\to\infty$ since the $u_n$'s converge weakly to zero.
\end{proof}

\vspace{2mm}

A map $\omega:G\to\mathscr{U}(\cL(A^{(I)}))$ satisfying the condition $\omega_g\sigma_g(\omega_h)=\omega_{gh}$ for every $g,h\in G$ is called a \textit{$1$-cocycle for $\sigma$}. Such $1$-cocycle for $\sigma$ is \textit{trivial} (or it is a \textit{coboundary}) if there exists a unitary element $v\in\mathscr{U}(\cL(A^{(I)}))$ such that $\omega_g=v^*\sigma_g(v)$, for all $g$. The map $\omega$ is a \textit{weak $1$-cocycle} if it satisfies the previous relation modulo the scalars, i.e., $\omega_g\sigma_g(\omega_h)\omega_{gh}^*\in\mathbb{T}1$, for all $g,h\in G$. It is \textit{weakly trivial} (or a \textit{weak coboundary}) if it is trivial modulo the scalars.

\vspace{2mm}

One can see that in \cite[Theorem 3.1]{sasyk} the only facts used about $G$-Bernoulli actions are malleability and weak mixingness. Since the action defined above satisfies these two properties, and the conclusion of \cite[Lemma 3.2]{sasyk}, we obtain the same statement for our action $\sigma$.

\begin{cor}\label{weakcocycleisweakcoboundary} Let $B \curvearrowright I$ be a faithful action with infinite orbits. Let $G\in\mathcal{WR}(A,B\curvearrowright I)$ be a property (T) group, where $A$ is a free abelian group and $B$ is an ICC subgroup of a hyperbolic group. Given any weak 1-cocycle $\omega$ for the action $\sigma$ defined above, $\omega$ is a
weak coboundary.
\end{cor}

\begin{thm} \label{untwistcocycle} Let $B \curvearrowright I$ be a faithful action with infinite orbits. Let $G\in \mathcal{WR}(A,B\curvearrowright I)$ be a property (T) wreath-like product group, where $A$ is a free abelian group, $B$ is an ICC subgroup of a hyperbolic group. Let $H\in \mathcal{WR}(C,D\curvearrowright J)$ be any property (T) wreath-like product group where $C$ is abelian, $D$ is an ICC subgroup of a hyperbolic group, and $D\curvearrowright J$ is faithful with infinite orbits. Let $c: G\times G\ra \mathbb T$ be  any $2$-cocycle. 

\vspace{2mm}

Assume that  $\Phi:\cL_c(G)\to \cL(H)$ is a $*$-isomorphism. Then, $c$ is a trivial 2-cocycle and $G\cong H$. Moreover, there exist a group isomorphism $\delta:G\to H$, a character $\eta:G\times G\to\mathbb{T}$ and a unitary $w\in \cL(H)$ such that $\Phi(u_g)=\eta(g)wv_{\delta(g)}w^*$ for every $g\in G$, where $(u_g)_{g\in G}$ and $(v_h)_{h\in H}$ denote the canonical generating unitaries of $\cL(G)$ and $\cL(H)$, respectively.
\end{thm}

\begin{proof} Let $\Phi:\cL_c(G)\to \cL(H)$ be a $*$-isomorphism, and denote by $\mathcal{P}:=\Phi(\cL_c(A^{(I)}))\subset\cL(H)=:\mathcal{M}\supset \cL(C^{(J)})=:\mathcal{Q}$. Then, $\mathcal{P},\mathcal{Q}\subset\cM$ are Cartan subalgebras. Since $D$ is an ICC subgroup of a hyperbolic group, by \cite[Lemma 3.7]{cios22} we have that $\cM$ has a unique Cartan subalgebra, and therefore, there exists $u\in\mathscr{U}(\cM)$ with $u\mathcal{P}u^*=\mathcal{Q}$. By replacing $\Phi$ with $\text{Ad}(u)\circ\Phi$, we can assume that $\mathcal{P}=\mathcal{Q}$.

\vspace{2mm}

Denote by $\sigma_h=\text{Ad}\:\Phi(u_{\hat{h}})\in\text{Aut}(\mathcal{P})$, for $h\in B$, the action built from $G\curvearrowright A^{(I)}$ and the cocycle $c:G \times G \to \mathbb T$ given at the beginning of the subsection. Then, $\sigma=(\sigma_h)_{h\in B}$ defines a free and weak mixing action of $B$ on $\mathcal{P}$ by Lemma \ref{weakmixingaction}. Now, consider the action $\alpha=(\alpha_h)_{h\in D}$ of $D$ on $\mathcal{Q}$ given by $\alpha_h:=\text{Ad}(v_{\hat{h}})$. Let $(Y,\mu)$ be the dual of $C$ with its Haar measure. Denote $(X,\nu)=(Y^J,\mu^J)$, and identify $\mathcal{Q}=L^{\infty}(X,\nu)$. Under the identification $\mathcal{P}=\mathcal{Q}$, we also identify $\mathcal{P}=L^{\infty}(X,\nu)$. The actions $\sigma$ and $\alpha$ induce actions $B\curvearrowright (X,\nu)$ and $D\curvearrowright (X,\nu)$, denoted by $\sigma$ and $\alpha$, respectively. 

\vspace{2mm}

Using the fact $\mathcal{Q}\subset\cM$ is a Cartan subalgebra and the p.m.p. equivalence relation associated to the inclusion $\mathcal{Q}\subset \cM$ is equal to $\mathscr{R}(D\curvearrowright^{\alpha}X)$ we get that

\begin{equation*}
    \sigma(B)\cdot x=\alpha(D)\cdot x\quad\nu\text{-a.e. }x\in X.
\end{equation*}

\vspace{2mm}

Fix $j\in J$ and $g\in D\setminus\{1\}$. Since the orbits of the action $D \curvearrowright J$ are infinite, $\text{Stab}_D(j)$ has infinite index in $D$. Since $D$ is ICC the same holds for $C_D(g)$. Then, using the same type of argument as in the proof of \cite[Theorem 1.3 Step 3]{cios22} there exists a sequence $(h_m)\subset B$ with

\begin{align*}
    \mu&(\{x\in X:\sigma_{h_m}(x)\in\alpha(s\text{Stab}_D(j)t)(x)\})\to 0 \text{ and }\mu(\{x\in X:\sigma_{h_m}(x)\in\alpha(s\text{C}_D(g)t)(x)\})\to 0,
\end{align*}

\vspace{2mm}

for all $s,t\in D, g\in D\setminus \{1\}$ and $ j\in J$. By \cite[Theorem 4.1]{cios22}, there is a group isomorphism $\delta:B\to D_1\leqslant D$ and $\varphi\in [\mathscr R(D\curvearrowright^{\alpha}X)]$ such that

\begin{equation*}
    \varphi\circ\sigma(h)=\alpha(\delta(h))\circ\varphi,\quad\text{for all } h\in B.
\end{equation*}

\vspace{2mm}

Let $u\in\mathscr{N}_{\cM}(\mathcal{Q})$ with $uau^*=a\circ\varphi^{-1}$, for all $a\in\mathcal{Q}$. Since $\sigma_h=\text{Ad}(\Phi(u_{\hat{h}}))$, there exists $\xi_h\in\mathscr{U}(\mathcal{Q})$ such that

\begin{equation*}
    u\Phi(u_{\hat{h}})u^*=\xi_h v_{\widehat{\delta(h)}},\quad\text{for all } h\in B.
\end{equation*}

\vspace{2mm}

Replacing $\Phi$ with $\text{Ad}(u)\circ\Phi$, we have that

\begin{equation*}
    \Phi(u_{\hat{h}})=\xi_hv_{\widehat{\delta(h)}},\quad\text{for all } h\in B.
\end{equation*}

\vspace{2mm}

Let $\Delta:\cL(H)\to\cL(H)\:\overline{\otimes}\:\cL(H)$ be the diagonal embedding. For $g\in G$, let $h=\epsilon(g)\in B$ and $a=g\widehat{h}^{-1}\in A^{(I)}$. Then,

\begin{align*}
    \Delta(\Phi(u_g))&=\Delta(\Phi(u_{a\widehat{h}}))=c(a,\hat{h})^{-1}\Delta(\Phi(u_{a}u_{\widehat{h}}))=c(a,\hat{h})^{-1}\Delta(\Phi(u_a))\Delta(\Phi(u_{\widehat{h}}))\\
    &=c(a,\hat{h})^{-1}\Delta(\Phi(u_{a}))\Delta(\xi_{h})(v_{\widehat{\delta(h)}}\otimes v_{\widehat{\delta(h)}})\\
    &=c(a,\hat{h})^{-1}\Delta(\Phi(u_{a})\xi_{h})(\xi_{\epsilon(g)}^*\otimes\xi_{\epsilon(g)}^*)(\Phi(u_g)\otimes \Phi(u_g))=\tilde{\eta}(g)(\Phi(u_g)\otimes \Phi(u_g)),
\end{align*}

\vspace{2mm}

Here we have denoted by $\tilde{\eta}(g):= c(a,\hat{h})^{-1}\Delta(\Phi(u_{a})\xi_{h})(\xi_{\epsilon(g)}^*\otimes\xi_{\epsilon(g)}^*)$.

\vspace{2mm}

Observe that

\begin{equation*}
    \Delta(\Phi(u_g))\Delta(\Phi(u_h))=c(g,h)\Delta(\Phi(u_{gh}))=c(g,h)\tilde{\eta}(gh)(\Phi(u_{gh})\otimes\Phi(u_{gh})),
\end{equation*}

\vspace{2mm}

while, on the other hand,

\begin{equation*}
    \Delta(\Phi(u_g))\Delta(\Phi(u_h))=\tilde{\eta}(g)(\sigma\otimes\sigma)_g(\tilde{\eta}(h))c(g,h)^2(\Phi(u_{gh})\otimes\Phi(u_{gh})).
\end{equation*}

\vspace{2mm}

Hence, $\tilde{\eta}(g)(\sigma\otimes\sigma)_g(\tilde{\eta}(h))c(g,h)=\tilde{\eta}(gh)$ for all $g,h\in G$. Therefore, $\tilde{\eta}$ is a weak 1-cocycle for the diagonal action $G\curvearrowright^{\sigma\otimes\sigma} \cL(A^{(I)})\overline{\otimes}\cL(A^{(I)})$. By Corollary \ref{weakcocycleisweakcoboundary}, $\tilde{\eta}$ is a weak coboundary; i.e. $\tilde{\eta}(g)=\lambda_gu^*(\sigma\otimes\sigma)_g(u)$ for some $u\in\mathscr{U}(\cL(A^{(I)})\overline{\otimes}\cL(A^{(I)}))$ and $\lambda_g\in\mathbb{T}$. It follows that $c$ is coboundary  $c(g,h)=\lambda_g\lambda_h\lambda_{gh}^{-1}$, and thus the map $ u_g \ra \lambda_g u_g$  implements a $\ast$-isomorphism, $\cL(G)\cong \cL_c(G)$. The rest of the statement follows using the conclusion of \cite[Theorem 7.3]{cios22b}.
\end{proof}

We end pointing out a result that is an extension of \cite[Theorem 10.1]{ciosv1} (see pages 82-84) and we reproduce it here just for the reader's convenience. 

\begin{thm}\label{trivialamplificationembedding}(\cite[Theorem 10.1]{ciosv1}) Let $G\in \mathcal{WR}(A,B\curvearrowright I)$ be a property (T) wreath-like product group, where $A$ is a free abelian group, $B$ is an ICC subgroup of a hyperbolic group, and the group $G$ satisfies one of the following conditions: 


\begin{enumerate}
    \item [1a)] $C_B(b)$ is amenable, for all $b\in B\setminus\{1\}$ and the action $B\curvearrowright I$ has amenable stabilizers; or
    \item [2a)] the action $B\curvearrowright I$ is faithful and has infinite orbits.
\end{enumerate}

\vspace{1mm}

Let $H\in \mathcal{WR}(C,D\curvearrowright J)$ be any property (T) wreath-like product group where $C$ is abelian, $D$ is an ICC subgroup of a hyperbolic group, and $H$ satisfies one of the following conditions: 


\begin{enumerate}
    \item [1b)] $C_D(d)$ is amenable for all $d\in D\setminus\{1\}$ and the action $D\curvearrowright J$ has amenable stabilizers; or
    \item[2b)] the action $D\curvearrowright J$ is faithful and has infinite orbits.
\end{enumerate} 

\vspace{1mm}

Let $c: G\times G\to \mathbb T$ be any $2$-cocycle. Assume $0<t\leq 1$ is a scalar for which there exists a $\ast$-embedding $\Phi:\cL_c(G)\to \cL(H)^t$. 

\vspace{2mm}

If we are under the assumptions 1a) and 1b) then $t=1$. If we are under the assumptions 2a) and 2b) and the image $\Phi(\mathcal L_c(G))\subseteq \mathcal L(H)^t$ has finite index then $t=1$.
\end{thm}

\begin{proof} Let $p\in\cL(H)$ be a non-zero projection with $\tau(p)=t$. Denote $\cM=\cL_c(G)$, $\mathcal{P}=\cL_c(A^{(I)})$, $\cN=\cL(H)$ and $\mathcal{Q}=\cL(C^{(J)})$. We notice that since $A$ is free abelian, $c$ untwists on the core $A^{(I)}$ and hence it follows that $\mathcal P \subset \mathcal M$ is a Cartan subalgebra. 

\vspace{2mm}

Let $\mathcal{R}:=\Phi(\mathcal{P})'\cap p\cL(H)p$. 

\vspace{2mm}

Suppose first the assumptions 1a) and 1b) are satisfied. Since $\Phi(\cL_c(G))\subset\mathscr{N}_{p\cL(H)p}(\Phi(\mathcal{P}))''$ has property (T), \cite[Theorem 3.10]{cios22} implies $\Phi(\mathcal{P})\prec_{\cL(H)}^s\mathcal{Q}$. By \cite[Corollary 4.7]{cios22}, $\mathcal{R}$ is amenable.

\vspace{2mm}

Now assume 2a) and 2b).  Since $\Phi(\mathcal M )\subseteq p\mathcal N p$ has finite index then so is $\Phi(\mathcal P )= \Phi(\mathcal P)'\cap \Phi(M) \subseteq \mathcal R$. Since $\mathcal P$ is abelian it follows, once again, that $\mathcal R$ is amenable.

\vspace{2mm}

Notice that $\Phi(\cL_c(G))\subset\mathscr{N}_{p\cL(H)p}(\mathcal{R})''$ has property (T). Thus, using \cite[Theorem 3.10]{cios22} 

\begin{equation}\label{intertwiningagain}
    \mathcal{R}\prec_{\cL(H)}^s\mathcal{Q}.
\end{equation}

\vspace{2mm}

Since $\mathcal{Q}\subset\cL(H)$ is a Cartan subalgebra, combining the intertwining \eqref{intertwiningagain} together with \cite[Lemma 3.7]{cios22} we obtain that after replacing $\Phi$ with $\text{Ad}(u)\circ\Phi$, for some $u\in \mathscr{U}(\cL(H))$, we may assume that $p\in \mathcal{Q}$ and $\Phi(\mathcal{P})\subset \mathcal{Q}p\subset\mathcal{R}$.

\vspace{2mm}

If $g\in B$, denote by $\sigma_g=\text{Ad}(\Phi(u_{\hat{g}}))$ the action built from $G\curvearrowright A^{(I)}$ and the cocycle $c:G\times G\to\mathbb{T}$ given at the beginning of the subsection. Then, $\sigma=(\sigma_g)_{g\in B}$ defines an action of $B$ on $\mathcal{R}$ which leaves $\Phi(\mathcal{P})$ invariant. By Lemma \ref{weakmixingaction}(iii) the restriction of $\sigma$ to $\Phi(\mathcal{P})$ is free, and hence, \cite[Lemma 3.8]{cios22} implies the existence of an action $\beta=(\beta_g)_{g\in B}$ of $B$ on $\mathcal{R}$ such that

\vspace{1mm}

\begin{enumerate}
    \item[(a)] for every $g\in B$ we have $\beta_g=\sigma_g\circ\text{Ad}(w_g)=\text{Ad}(\Phi(u_{\hat{g}})w_g)$, for some $w_g\in\mathscr{U}(\mathcal{R})$; and
    \item[(b)] $\mathcal{Q}p$ is $\beta(B)$-invariant and the restriction of $\beta$ to $\mathcal{Q}p$ is free. 
\end{enumerate}

\vspace{2mm}

Let $(Y,\nu)$ be the dual of $C$ endowed with its Haar measure. Denote $(X,\mu)=(Y^{J},\nu^J)$. Consider the action $(\alpha_d)_{d\in D}$ of $D$ on $\mathcal{Q}$ given by $\alpha_d=\text{Ad}(u_{\hat{d}})$, and its induced action $D\curvearrowright^{\alpha}(X,\mu)$. Identify $\mathcal{Q}=L^{\infty}(X,\mu)$ and let $X_0\subset X$ be a measurable set such that $p=\mathbbm{1}_{X_0}$. Since $\mathcal{Q}p=L^{\infty}(X_0,\mu|_{X_0})$, we get a measure preserving action $B\curvearrowright^{\beta}(X_0,\mu|_{X_0})$.

\vspace{2mm}

Since the restriction of $\beta$ to $\mathcal{Q}p$ is implemented by unitaries in $p\cL(H)p$ and we have $\mathscr{R}(\mathcal{Q}\subset\cL(H))=\mathscr{R}(D\curvearrowright^{\alpha}X)$, we deduce that

\begin{equation*}
    \beta(B)\cdot x\subset\alpha(D)\cdot x,\text{ for almost every }x\in X_0.
\end{equation*}

\vspace{2mm}

Since $B$ has property (T), using \cite[Lemma 4.4]{cios22} and a maximality argument, we can partition $X_0=\cup_{m=1}^lX_m$ into non-null, measurable sets, and we can find finite index subgroups $S_m\leqslant B$ such that $X_m$ is $\beta(S_m)$ invariant and the restriction $\beta|_{S_m}$ to $X_m$ is weakly mixing, for all $1\leq m\leq l$.

\vspace{2mm}

Let $1\leq m\leq l$ and fix $j\in J$ and $g\in D\setminus\{1\}$. Next we argue that in both cases that  one can find $(h_k)\subset S_m$  such that $\mu(\{x\in X_0:\beta_{h_k}(x)\in\alpha(s {\rm Stab}_D(j)t)(x)\})\to 0$  and $\mu(\{x\in X_0:\beta_{h_k}(x)\in\alpha(sC_D(g)t)(x)\})\to 0$ as $k\to\infty$, for every $s,t\in D$.

\vspace{2mm}
Under the assumptions 1a) and 1b) this follows because $C_D(g)$ and ${\rm Stab}_D(j)$ are amenable, $\beta$ is free and $S_m$ has property (T). If we assume 2a) and 2b) this  follows using that $\beta$ is free and a similar argument as in the proof of \cite[Theorem 1.3 Step 3]{cios22} together with the fact the subgroups $C_D(g)$, ${\rm Stab}_D(j)< D$ have infinite index, $S_m \leqslant B$ has finite index, and the von Neumann algebra $\Phi(\mathcal M)\subseteq  p\mathcal N p$ also has finite index.

\vspace{2mm}

By \cite[Theorem 4.1]{cios22}, we can find an injective group homomorphism $\epsilon_m:S_m\to D$ and $\varphi_m\in[\mathscr{R}(D\curvearrowright X)]$ such that $\varphi_m(X_m)=X\times\{m\}\equiv X$ and $\varphi_m\circ\beta_h|_{X_m}=\alpha_{\epsilon_m(h)}\circ\varphi_m|_{X_m}$ for all $h\in S_m$. In particular, $\mu(X_m)=1$. This shows that $\mu(X_0)=l\in\N$, and so $l=1$. Therefore, $p=1$.
\end{proof}

\subsection{A W\texorpdfstring{$^*$}{*}-superrigidity result for twisted wreath-like product group factors} Some of the W$^*$-rigidity results from the prior section can be upgraded to W$^*$-superrigidity results, somewhat similar with the very recent results of Donvil and Vaes \cite{DV24}. Specifically, using the prior methods from \cite{cios22,ioa10} in combination with the height criterion for W$^*$-superrigidity of twisted group factors from the recent work \cite[Theorem 4.1]{DV24}, Theorem \ref{untwistcocycle} can be enhanced significantly by removing all initial assumptions on the group $H$; see the statement of \ref{notrivialcompressions}. 

\vspace{2mm}

We remark that even in its stronger form, our result is still not as impressive and general as \cite[Theorem A]{DV24}, which involves twisted group factors both on the source and target along with arbitrary amplifications. However, the methods from the current proof of \ref{notrivialcompressions} and the argument from the proof of Theorem \ref{untwistcocycle} can be further extended and used in combination with \cite[Theorem 4.1]{DV24} and other techniques from the this paper to establish a version of \ref{notrivialcompressions} which is completely analogous to \cite[Theorem A]{DV24}. Since, this is rather technical and not needed for deriving our main results, we decided to include it in a forthcoming paper. 

\vspace{2mm}

In preparation for the main result and its proof we recall some notations and some preliminary results. Let $G$ and $H$ be groups, and let $c: H \times H \to \mathbb T$ be a $2$-cocycle. Assume that  $\cL(G)=\cL_{c}(H)$. Next we recall the notion of triple commultiplication from \cite[Lemma 10.9]{ioa10} and also \cite[Section 1]{DV24}. This is the $\ast$-homomorphism $\Delta:\cL_{c}(H)\to \cL_{c}(H)\:\overline{\otimes}\:\cL_{c}(H)^{\rm op}\:\overline{\otimes}\:\cL_{c}(H)$ given by 

\begin{equation}\label{3commul}
    \Delta(v_h)= \overline{c(h^{-1},h)}\: v_h\otimes \overline v_{h^{-1}}\otimes v_h\text{ for }h\in H.
\end{equation} 

\vspace{2mm}

As $\cL(G)=\cL_{c}(H)$, we can view $\Delta$ as a $\ast$-embedding $\Delta:\cL(G)\to \cL(G)\:\overline{\otimes}\:\cL(G)^{\text{op}}\:\overline{\otimes}\:\cL(G)$.

\vspace{2mm}

One can check the triple commultiplication satisfies

\begin{equation}\label{commult2}        
    (\Delta\otimes\text{id}\otimes\text{id})\circ \Delta=(\text{id}\otimes\text{id}\otimes\Delta)\circ \Delta.
\end{equation}

\vspace{2mm}

Now consider the  star-flip map $F: \mathcal L(G)\:\overline{\otimes}\:\mathcal L(G)^{\rm op} \rightarrow \cL(G)\:\overline{\otimes}\: \cL(G)^{\rm op}$ given by $F(\sum_{g,h} \lambda_{g,h} (u_g \otimes \overline{u}_{h}))= \sum_{g,h} \overline{\lambda_{g,h}} (u_{h^{-1}}\otimes \overline{u}_{g^{-1}})$. This map is well-defined, conjugate linear and multiplicative.  Moreover, for every $x, y\in \mathcal L(G)$, 

\begin{equation}\label{sflip}F(x\otimes \overline y)= y^*\otimes \overline{x^*}.\end{equation}

\vspace{2mm}

For further use (see \ref{notrivialcompressions}), we also note the following equation holds:  

\begin{equation}\label{commult2F}
F_{1,2}\circ(\Delta\otimes\text{id}\otimes\text{id})\circ \Delta=F_{3,4}\circ(\text{id}\otimes\text{id}\otimes\Delta)\circ \Delta,
\end{equation} 

\vspace{2mm}

where we denoted by $F_{1,2}=F\otimes \text{id}\otimes \text{id}\otimes \text{id}$ and $F_{3,4}=\text{id}\otimes \text{id}\otimes F\otimes \text{id}$.

\begin{lem}\label{notintertwiningforinfiniteindex} (\cite{cios22,ipv10}) Assume $\cL(G)=\cL_c(H)$ and denote by $\cM:=\cL(G)$. Let $\Delta:\cM\to \cM\:\overline{\otimes}\:\cM^{\text{op}}\:\overline{\otimes}\cM$ be the three-folded commultiplication given by \eqref{3commul}. Then, the following hold:

\begin{enumerate}
    \item[(i)] $\Delta(\cM)\not\prec \mathcal{Q}\:\overline{\otimes}\:\cM^{\text{op}}\:\overline{\otimes}\:\cM$ and $\Delta(\cM)\not\prec \cM\:\overline{\otimes}\:\cM^{\text{op}}\:\overline{\otimes}\:\mathcal{Q}$ for any von Neumann subalgebra $\mathcal{Q}\subset \cM$ such that $\cM \nprec_\cM  \mathcal Q$.     
    \vspace{1mm}
    \item[(ii)] $\Delta(\cM)\not\prec\cM\:\overline{\otimes}\:\mathcal{Q}\:\overline{\otimes}\:\cM$ for any von Neumann subalgebra $\mathcal{Q}\subset \cM^{\text{op}}$ with $\cM^{\rm op}\nprec_{\cM^{\rm op}} \mathcal Q$.
    \item[(iii)] $\Delta(\mathcal{Q})\not\prec\cM\:\overline{\otimes}\:1\:\overline{\otimes}\:\cM$, $\Delta(\mathcal{Q})\not\prec\cM\:\overline{\otimes}\:\cM^{\text{op}}\:\overline{\otimes}\:1$ and $\Delta(\mathcal{Q})\not\prec 1\:\overline{\otimes}\:\cM^{\text{op}}\:\overline{\otimes}\:\cM$, for any diffuse von Neumann subalgebra $\mathcal{Q}\subset\cM$.
\end{enumerate}
\end{lem}

\begin{proof} We only prove the second statement, as the first one follows similarly and the third one is \cite[Lemma 7.2(1)]{ipv10}. We include a proof based on the same arguments from \cite{ipv10}. Assume $\cM^{\text{op}} \prec_{\cM^{\text{op}}}  \mathcal Q$. By Theorem \ref{intertwining} one can find $k\in \mathbb N$, $\overline{s_i},\overline{t_i}\in \cM^{\rm op}$ $1\leq i\leq k$ and $d>0$ such that 

\begin{equation}\label{belowbound1}\sum_{i=1}^k\|\mathbb{E}_{\mathcal{Q}}(\overline{s_i}\:\overline{v}_{h^{-1}}\:\overline{t_i})\|_2^2\geq d >0\text{, for all }h\in H.\end{equation} 

Fix $g\in G$. Letting $u_g=\sum_h\tau(u_gv_h^*)v_h$ be its Fourier expansion, and using the definitions together with \eqref{belowbound1} we have

\begin{align*}
     &\sum_{i=1}^k\|\mathbb{E}_{\cM\overline{\otimes}\mathcal{Q}\overline{\otimes}\cM}(\overline{s_i}\Delta(u_g)\overline{t_i})\|_2^2 =\sum_{i=1}^k\left\|\sum_h\tau(u_gv_h^*)\overline{c(h^{-1},h)}(v_h\otimes \mathbb{E}_{\mathcal{Q}}(\overline{s_i}\:\overline{v}_{h^{-1}}\overline{t_i})\otimes v_h)\right\|_2^2\\
    &=\sum_{i=1}^k\sum_h|\tau(u_gv_h^*)|^2\|\mathbb{E}_{\mathcal{Q}}(\overline{s_i}\:\overline{v}_{h^{-1}}\overline{t_i})\|_2^2= \sum_h|\tau(u_gv_h^*)|^2\left( \sum_{i=1}^k\|\mathbb{E}_{\mathcal{Q}}(\overline{s_i}\:\overline{v}_{h^{-1}}\overline{t_i})\|_2^2\right)\geq d.
\end{align*}

Since this holds for all $g \in G$, Theorem \ref{intertwining} implies that  $\Delta(\cM)\prec\cM\:\overline{\otimes}\:\mathcal{Q}\:\overline{\otimes}\:\cM$.
\end{proof}

Before the next lemma, we recall the notion of height first defined in \cite[Section 4]{ioa10} and \cite[Section 3]{ipv10} and adapted to twisted group von Neumann algebras in the preprint \cite[Section 4]{DV24}. Given a countable group $H$, a 2-cocycle $c:H\times H\to\mathbb{T}$ and a subgroup $\mathcal{G}\leq \cL_c(H)$, define 

\begin{equation*}
    h_H(\mathcal{G})=\inf_{w\in \mathcal{G}}\left(\max_{h\in H}|\tau(v_{h^{-1}}w)|\right).
\end{equation*}

\vspace{2mm}

The proof of \ref{notrivialcompressions} relies on the following elementary result, which is essentially contained in the proof of \cite[Proposition 6.12]{DV24}. For reader's convenience include a complete proof below.

\begin{lem}\label{positiveheight} Let $G$, $H$ be countable groups and $c:H\times H\to\mathbb{T}$ be a $2$-cocycle. Assume $\cM:=\cL(G)=\cL_c(H)$ and let $\Delta$ as above. If there exist $0\neq x,y\in \cM\:\overline{\otimes}\:\cM^{\text{op}}\:\overline{\otimes}\:\cM$ such that $\Delta(u_g)x=y(u_g\otimes \overline {u}_{g^{-1}}\otimes u_g)$ for all $g\in G$, then $h_H(G)>0$.
\end{lem}

\begin{proof} Fix $\epsilon>0$ and let $F_{\epsilon}=F_1\times F_2\times F_3$, $K_{\epsilon}=K_1\times K_2\times K_3\subset H\times H\times H$ finite sets and $x_{\epsilon}$, $y_{\epsilon}\in\cM\:\overline{\otimes}\:\cM^{\text{op}}\:\overline{\otimes}\:\cM$ supported on $F_{\epsilon}$ and $K_{\epsilon}$, respectively, with $\|x-x_{\epsilon}\|_2<\epsilon$ and $\|y-y_{\epsilon}\|_2<\epsilon$. Denote by $\tilde{\tau}$ the trace on $\cM\:\overline{\otimes}\:\cM^{\text{op}}\:\overline{\otimes}\:\cM$. Using these estimates

\begin{align*}
    &\|y\|_2^2=|\langle \Delta(u_g)x(u_{g^{-1}}\otimes \overline u_{g}\otimes u_{g^{-1}}),y\rangle|\leq \epsilon\|y\|_2+\epsilon\|x_{\epsilon}\|_2+|\langle\Delta(u_g),y_{\epsilon}(u_{g}\otimes \overline u_{g^{-1}}\otimes u_{g})  x^*_{\epsilon}\rangle|\\
    &\leq \epsilon(\|y\|_2+\|x\|_2+\epsilon)+\\
    &\quad\quad+\left|\sum_{\substack{(f_1,f_2,f_3)=f\in F_{\epsilon}\\(k_1,k_2,k_3)=k\in K_{\epsilon}\\ h, h_1,h_2,h_3\in H}}\tilde{\tau}(x^*_\epsilon v_f) \tilde{\tau}(y_\epsilon v_k)\tau(u_gv_{h^{-1}})\overline{c(h^{-1},h)} \tau(u_g v_{h_1^{-1}})\overline \tau( \overline u_{g^{-1}}\overline v_{h_2})\tau( u_g v_{h_3^{-1}})\right.\\
    &\quad\quad\quad\quad\quad\quad\left.\delta_{h,k_1 h_1 f_1^{-1}}\delta_{h^{-1},f_2^{-1}h_2k_2} \delta_{h,k_3 h_3 f_3^{-1}}\right|\\
    &\leq \epsilon(\|y\|_2+\|x\|_2+\epsilon)+
    \sum_{\substack{(f_1,f_2,f_3)\in F_{\epsilon}\\(k_1,k_2,k_3)\in K_{\epsilon}\\ h\in H}}\|x^*_\epsilon\|_2 \|y_\epsilon\|_2 |\tau(u_gv_{h^{-1}}) \tau(u_g v_{f_1^{-1}h^{-1}k_1})\overline \tau( \overline u_{g^{-1}}\overline v_{f_2h^{-1}k_2^{-1}})\tau(u_g v_{f_3^{-1}h^{-1}k_3})|\\ 
    &\leq \epsilon(\|y\|_2+\|x\|_2+\epsilon)+(\|x\|_2+\epsilon)(\|y\|_2+\epsilon)\cdot |F_{\epsilon}|\cdot |K_{\epsilon}|\cdot h_H(u_g).
\end{align*}

\vspace{2mm}

Thus, for all $g\in G$ we have

\begin{equation*}
    h_H(u_g)\geq \frac{\|y\|_2^2-\epsilon(\|y\|_2+\|x\|_2+\epsilon)}{(\|x\|_2+\epsilon)(\|y\|_2+\epsilon)\cdot |F_{\epsilon}|\cdot |K_{\epsilon}|}.
\end{equation*}

\vspace{2mm}

Picking $\epsilon>0$ small enough, we obtain $h_H(G)>0$.
\end{proof}

With these preparations at hand we now introduce our main rigidity result. Our approach is similar to the proof of \cite[Theorem 1.3]{cios22}. However for the sake of completeness we include all details.

\vspace{2mm}

\begin{proof}[\textbf{Proof of \ref{notrivialcompressions}}] Using Theorem \ref{notrivialcompressions2} we get $p=1$. Identify $\cL(G)$ and $\cL_{c}(H)$ under the $\ast$-isomorphism $\Phi$. Let 

\begin{equation*}
    \Delta:\cL(G)\to \cL(G)\:\overline{\otimes}\:\cL(G)^{\text{op}}\:\overline{\otimes}\:\cL(G)
\end{equation*}

\vspace{2mm}

be the unital $\ast$-embedding described at the beginning of this section. Let 

\begin{align*}
    \cM:=\cL(G),\quad \tilde{\cM}:=\cM\:\overline{\otimes}\:\cM^{\text{op}}\:\overline{\otimes}\:\cM,&\quad \tilde{\mathcal{Q}}:=\cL(A^{(I)})\:\overline{\otimes}\:\cL(A^{(I)})^{\text{op}}\:\overline{\otimes}\:\cL(A^{(I)}),\quad \mathcal{P}:=\Delta(\cL(A^{(I)})),\quad\\
    &\mathcal{R}:=\mathcal{P}'\cap \tilde{\cM}\subset \tilde{\cM}.
\end{align*}

\begin{claim}\label{intertwiningcios} $\mathcal{R}\prec_{\tilde{\cM}}^s\cL(A^{(\tilde{I})})=\tilde{\mathcal{Q}}$.
\end{claim}

\begin{subproof}[Proof of Claim \ref{intertwiningcios}] Since $\tilde{\cM}=\cM\:\overline{\otimes}\:\cM^{\text{op}}\:\overline{\otimes}\:(\cL(A^{(I)})\rtimes B)=(\cM\:\overline{\otimes}\:\cM^{\rm op}\:\overline{\otimes}\:\cL(A^{(I_3)}))\rtimes B$ and $B$ is a subgroup of a hyperbolic group, by \cite[Theorem 3.10]{cios22}, and since $\cM$ has property (T), we have that $\mathcal{P}\prec^s\cM\:\overline{\otimes}\:\cM^{\text{op}}\:\overline{\otimes}\:\cL(A^{(I)})$. Doing the same argument in each tensor, we obtain $\mathcal{P}\prec^s\cL(A^{(I)})\:\overline{\otimes}\:\cM^{\text{op}}\:\overline{\otimes}\:\cM$ and $\mathcal{P}\prec^s\cM\:\overline{\otimes}\:\cL(A^{(I)})^{\text{op}}\:\overline{\otimes}\:\cM$. By \cite[Lemma 2.8(2)]{dhi19} we have $\mathcal{P}\prec^s \cL(A^{(I)})\:\overline{\otimes}\:\cL(A^{(I)})^{\text{op}}\:\overline{\otimes}\:\cL(A^{(I)})=\tilde{\mathcal{Q}}$. 

\vspace{2mm}

Using Lemma \ref{notintertwiningforinfiniteindex} we also have $\mathcal{P}\not\prec \cL(A^{(I)})\:\overline{\otimes}\:\cL(A^{(I)})^{\text{op}}\:\overline{\otimes}\:1$, $\cL(A^{(I)})\:\overline{\otimes}\:1\:\overline{\otimes}\:\cL(A^{(I)})$ nor $1\:\overline{\otimes}\:\cL(A^{(I)})^{\text{op}}\:\overline{\otimes}\cL(A^{(I)})$. By \cite[Corollary 4.7]{cios22}, $\mathcal{R}$ is amenable. Moreover, since $\mathcal{R}$ is normalized by $(\Delta(u_{\hat{g}}))_{g\in B}$, repeating the first part of the proof we get $\mathcal{R}\prec^s\tilde{\mathcal{Q}}$.
\end{subproof}

Since $\tilde{\cQ}\subset\tilde{\cM}$ is a Cartan subalgebra, combining Claim \ref{intertwiningcios} together with \cite[Lemma 3.7]{cios22} (see also \cite{ioa10}), after replacing $\Delta$ with $\text{Ad}(u)\circ\Delta$, for some $u\in\tilde{\cM}$, we may assume that 

\begin{equation*}
    \mathcal{P}\subset \tilde{\mathcal{Q}}\subset\mathcal{R}.
\end{equation*}

\vspace{2mm}

Throughout the proof, we will use of the following notation: for every $g\in B$, let $\hat{g}\in G$ so that $\epsilon(\hat{g})=g$, where $\varepsilon:G\to B$ is the canonical quotient map of the wreath-like product $G\in\mathcal{WR}(A,B\curvearrowright I)$.

\vspace{2mm}

Fixing $g\in B$ notice that $\Delta(u_{\hat{g}})$ normalizes $\mathcal{P}$ and $\mathcal{R}$. Denote by $\sigma_g:=\text{Ad}(\Delta(u_{\hat{g}}))\in\text{Aut}(\mathcal{R})$. Then, $\sigma=(\sigma_g)_{g\in B}$ defines an action of $B$ on $\mathcal{R}$ which leaves $\mathcal{P}$ invariant. Since the restriction of $\sigma$ to $\mathcal{P}$ is conjugate to an action $B\curvearrowright\cL(A)^I$ built over $B\curvearrowright I$, it is free and weakly mixing. 

\vspace{2mm}

By \cite[Theorem 1.3 - Step 2]{cios22}, the action $B\curvearrowright^{\sigma}\mathcal{R}$ is weakly mixing. This implies $\mathcal{R}$ is a type I$_k$ - algebra, for some $k\in\N$. Thus, there exists a decomposition $\mathcal{R}=\mathscr{Z}(\mathcal{R})\:\overline{\otimes}\:\mathbb{M}_k(\C)$ such that $\tilde{\mathcal{Q}}=\mathscr{Z}(\mathcal{R})\:\overline{\otimes}\:\mathbb{D}_k(\C)$. Therefore, $(\mathcal{R})_1\subset\sum_{i=1}^k(\tilde{\mathcal{Q}})_1x_i$, for some $x_1,...,x_k\in\mathcal{R}$. Moreover, by \cite[Lemma 3.8]{cios22}, there exists an action $\beta=(\beta_g)_{g\in B}$ of $B$ on $\mathcal{R}$ such that

\vspace{1mm}

\begin{enumerate}
    \item[(a)] for every $g\in B$ we have $\beta_g=\sigma_g\circ\text{Ad}(w_g)=\text{Ad}(\Delta(u_{\hat{g}})w_g)$, for some $w_g\in \mathscr{U}(\mathcal{R})$; and
    \item[(b)] $\tilde{\mathcal{Q}}$ is $\beta(B)$-invariant and the restriction of $\beta$ to $\tilde{\mathcal{Q}}$ is free. 
\end{enumerate}

\vspace{2mm}

Let $\tilde{B}=B\times B\times B$ and consider the action $\alpha:=(\alpha_h)_{h\in\tilde{B}}$ of $\tilde{B}$ on $\tilde{\mathcal{Q}}$ given by 

\begin{equation*}
    \alpha_{(h_1,h_2,h_3)}=\text{Ad}(u_{\widehat{h_1}}\otimes u_{\widehat{h_2^{-1}}}\otimes u_{\widehat{h_3}}),
\end{equation*}

\vspace{2mm}

for $h_1,h_2,h_3\in B$. Let $(Y,\nu)$ be the dual of $A$ with its respective Haar measure, and let $(X,\mu):=(Y^{I}\times Y^I\times Y^I,\nu^I\times\nu^I\times \nu^I)$. Identify $\tilde{Q}=L^{\infty}(X,\mu)$, and denote by $\alpha$ still the corresponding p.m.p. action $\tilde{B}\curvearrowright(X,\mu)$.

\vspace{2mm}

Notice $\tilde{\mathcal{Q}}\subset \tilde{\cM}$ is a Cartan subalgebra, and the equivalence relation associated to the inclusion is equal to $\mathscr{R}(\tilde{B}\curvearrowright X)$. Since the restriction of $\beta$ to $\tilde{\mathcal{Q}}$ is implemented by unitaries in $\tilde{\cM}$ we deduce that $\beta(B)\cdot x\subset \alpha(\tilde{B})\cdot x$ for almost every $x\in X$. Since $B$ has property (T) and $\alpha$ is an action built over $\tilde{B}\curvearrowright I$, using \cite[Lemma 4.4]{cios22} and a maximality argument, we can partition $X=\cup_{k=1}^lX_k$ into non-null measurable sets and we can find finite index subgroups $S_k\leqslant B$ such that $X_k$ is $\beta(S_k)$ invariant, the restriction $\beta|_{S_k}$ to $X_k$ is weakly mixing, for all $1\leq k\leq l$, and

\begin{equation}\label{actionsonXm}
    \beta(S_k)\cdot x\subset \alpha(\tilde{B})\cdot x \text{ for almost every }x\in X_k.
\end{equation}

\vspace{2mm}

Finally, to apply \cite[Theorem 4.1]{cios22} we prove the following claim.

\begin{claim}\label{setstozero} Let $B_1=B_2=B_3=B$. For every subgroup of infinite index $K_i\leqslant B_i$ there exists a sequence $(h_m)_m\subset B$ such that for every $s,t\in B$ and $1\leq i\leq 3$ we have

\begin{equation*}
    \mu(\{x\in X:\beta_{h_m}(x)\in\alpha(B_{\hat{i}}\times sK_it)(x)\})\to 0\text{ as } m\to\infty.
\end{equation*}
\end{claim}

\begin{subproof}[Proof of Claim \ref{setstozero}] Let $G_i:=\epsilon^{-1}(K_i)$, and notice that $G_i$ is an infinite index subgroup of $G$. By Lemma \ref{notintertwiningforinfiniteindex}, we can find a sequence $(g_m)_m\subset G$ with

\begin{equation*}
    \|\mathbb{E}_{\cM\overline{\otimes}\cM^{\text{op}}\overline{\otimes}\cL(G_3)}(x\Delta(u_{g_m})y)\|_2,\:\|\mathbb{E}_{\cM\overline{\otimes}\cL(G_2)^{\text{op}}\overline{\otimes}\cM}(x\Delta(u_{g_m})y)\|_2,\:\|\mathbb{E}_{\cL(G_1)\overline{\otimes}\cM^{\text{op}}\overline{\otimes}\cM}(x\Delta(u_{g_m})y)\|_2
\end{equation*}

\vspace{2mm}

converge to zero for all $x,y\in \tilde{\cM}$. Let $h_m:=\epsilon(g_m)\in B$. We claim $(h_m)_m$ is the sequence that satisfies the statement. Since $g_m^{-1}\widehat{h_m}\in A^{(I)}$ and $w_{h_m}\in\mathscr{U}(\mathcal{R})$ we get $\Delta(u_{\widehat{h_m}})w_{h_m}\in\Delta(u_{g_m})\mathscr{U}(\mathcal{R})$. Thus, for every $m$, $\Delta(u_{\widehat{h_m}})w_{h_m}\in \sum_{i=1}^t\Delta(u_{g_m})(\tilde{\mathcal{Q}})_1x_i$ for $x_1,...,x_t\in\mathcal{R}$. As $\tilde{\mathcal{Q}}$ is regular in $\tilde{\cM}$, this implies that for every $x,y\in\tilde{\mathcal{M}}$

\begin{align}\label{Deltaw_htozero}
    \|\mathbb{E}_{\cM\overline{\otimes}\cM^{\text{op}}\overline{\otimes}\cL(G_3)}(x\Delta&(u_{\widehat{h_m}})w_{h_m}y)\|_2,\:\|\mathbb{E}_{\cM\overline{\otimes}\cL(G_2)^{\text{op}}\overline{\otimes}\cM}(x\Delta(u_{\widehat{h_m}})w_{h_m}y)\|_2,\nonumber\\
    &\|\mathbb{E}_{\cL(G_1)\overline{\otimes}\cM^{\text{op}}\overline{\otimes}\cM}(x\Delta(u_{\widehat{h_m}})w_{h_m}y)\|_2
\end{align}

\vspace{2mm}

converge to zero. On the other hand, $\beta_{h_m}=\text{Ad}(\Delta(u_{\widehat{h_m}})w_{h_m})$ and $\alpha(g_1,g_2,g_3)=\text{Ad}(u_{(\widehat{g_1},\widehat{g_2^{-1}},\widehat{g_3})})$ for all $g_1,g_2,g_3\in B$. Altogether, these imply each of the following 

\begin{align*}
    \mu(\{x\in X:\beta_{h_m}(x)\in\alpha(B_{\hat 3}\times sK_3t)(x)\})&=\|\mathbb{E}_{\cM\overline{\otimes}\cM^{\text{op}}\overline{\otimes}\cL(G_3)}((1\otimes 1\otimes u_{\hat{s}}^*)\Delta(u_{\widehat{h_m}})w_{h_m}(1\otimes 1 \otimes u_{\hat{t}}^*))\|_2^2\\
    \mu(\{x\in X:\beta_{h_m}(x)\in\alpha(B_{\hat 2}\times sK_2t)(x)\})&=\|\mathbb{E}_{\cM\overline{\otimes}\cL(G_2)^{\text{op}}\overline{\otimes}\cM}((1\otimes \overline{u_{\hat{t}}^*}\otimes 1)\Delta(u_{\widehat{h_m}})w_{h_m}(1\otimes \overline{u_{\hat{s}}^*}\otimes 1))\|_2^2\\
    \mu(\{x\in X:\beta_{h_m}(x)\in\alpha( B_{\hat 1}\times sK_1t)(x)\})&=\|\mathbb{E}_{\cL(G_1)\overline{\otimes}\cM^{\text{op}}\overline{\otimes}\cM}((u_{\hat{s}}^*\otimes 1\otimes 1)\Delta(u_{\widehat{h_m}})w_{h_m}(u_{\hat{t}}^*\otimes 1\otimes 1))\|_2^2
\end{align*}

\vspace{2mm}

converges to zero.
\end{subproof}

Fix $j\in I$, so that $j\in B_i$ for some $1\leq i\leq 3$, and let $g=(g_1,g_2,g_3)\in \tilde{B}\setminus\{1\}$. Then, $\text{Stab}_{\tilde{B}}(j)=B_{\hat{i}}\times \text{Stab}_{B_i}(j)$ and $C_{\tilde{B}}(g)=C_{B_1}(g_1)\times C_{B_2}(g_2)\times C_{B_3}(g_3)$. Since $g\neq 1$, there exists $1\leq i\leq 3$ for which $g_i\neq 1$, and hence, $C_{\tilde{B}}(g)\leqslant B_{\hat{i}}\times C_{B_i}(g_i)$, where $C_{B_i}(g_i)$ is an infinite index subgroup of $B_i$. 

\vspace{2mm}

Let $1\leq k\leq l$. Since $\beta|_{S_k}$ is a free and weakly mixing action on $X_k$, $S_k$ has property (T), equation \eqref{actionsonXm}, the previous paragraph and Claim \ref{setstozero} show that the conditions of the moreover assertion of \cite[Theorem 4.1]{cios22} are satisfied by $\beta|_{S_k}$ on $X_k$ and $\alpha$. Thus, we can find an injective group homomorphism $\delta_k:S_k\to\tilde{B}$ and $\varphi_k\in \left[\mathscr{R}(\tilde{B}\curvearrowright X)\right]$ such that $\varphi_k(X_k)=X\times\{k\}\equiv X$ and $\varphi_k\circ\beta_h|_{X_k}=\alpha_{\delta_k(h)}\circ\varphi_k|_{X_k}$ for all $h\in S_k$. In particular, $\mu(X_k)=1$. This shows that $k=1$ and $S_1=B$.

\vspace{3mm}

Thus, we obtain a map $\delta=(\delta_1,\delta_2,\delta_3):B\to B\times B\times B$ and $\varphi\in\left[\mathscr{R}(\tilde{B}\curvearrowright X)\right]$ such that $\varphi\circ\beta_h=\alpha_{\delta(h)}\circ\varphi$ for all $h\in B$. 

\vspace{2mm}

Now, let $u\in\mathcal N_{\tilde{\mathcal M}}(\tilde{\mathcal{Q})}$ such that $uau^*=a\circ\phi^{-1}$, for every $a\in\tilde{\mathcal{Q}}$. The last relation implies that we can find $(\zeta_{h})_{h\in B}\subset\mathscr U(\tilde{\mathcal{Q}})$ such that

\begin{equation}\label{conjugat}
    u\Delta(u_{\widehat{h}})w_hu^*=\zeta_{h}u_{(\widehat{\delta_{1}(h)},\widehat{\delta_{2}(h)^{-1}},\widehat{\delta_{3}(h)})},\quad \text{for every $h\in B$.}
\end{equation}

\vspace{2mm}

After replacing $\Delta$ with $\text{Ad}(u)\circ\Delta$ we have that $(\xi_h)_{h\in B}\subset\mathscr{U}(\tilde{\mathcal{Q}})$ and \eqref{conjugat} rewrites as

\begin{equation}\label{conjug}
    \Delta(u_{\widehat{h}})w_h=\zeta_{h}u_{(\widehat{\delta_{1}(h)},\widehat{\delta_{2}(h)^{-1}},\widehat{\delta_{3}(h)})},\quad \text{for every $h\in B$.}
\end{equation}

\vspace{1mm}

\begin{claim}\label{RcontainedtildeQ}
    $\mathcal R \subseteq \tilde{\mathcal Q}$.
\end{claim} 

\begin{subproof}[Proof of Claim \ref{RcontainedtildeQ}] Since $(\Delta(u_{\hat{h}})w_h)_{h\in B}\subset\mathscr{U}(\tilde{\mathcal{M}})$ normalizes $\mathcal{R}$ and $(\zeta_h)_{h\in B}\subset\mathscr{U}(\tilde{\mathcal{Q}})$, equation \eqref{conjug} shows that $u_{(\widehat{\delta_{1}(h)},\widehat{\delta_{2}(h)^{-1}},\widehat{\delta_{3}(h)})}$ normalizes $\mathcal{R}$. Thus, to prove the claim it suffices to argue that for every $x,y\in\tilde \cM$ and $z\in \tilde \cM\ominus \tilde{\mathcal{Q}}$ we have

\begin{equation*}
    \|\mathbb{E}_{\tilde{\mathcal{Q}}}(xu_{(\widehat{\delta_{1}(h_m)},\widehat{\delta_{2}(h_m)^{-1}},\widehat{\delta_{3}(h_m)})}zu_{(\widehat{\delta_{1}(h_m)},\widehat{\delta_{2}(h_m)^{-1}},\widehat{\delta_{3}(h_m)})}^*y)\|\to 0.
\end{equation*}

\vspace{2mm}

Assume $x=u_a$, $y=u_b$ and $z=u_g$ for $a,b,g\in G\times G\times G$ and $g\not\in A^{(I)}\times A^{(I)}\times A^{(I)}$. Write $(\pi\times\pi\times\pi)(a)=(a_1,a_2,a_3)$, $(\pi\times\pi\times\pi)(b)=(b_1,b_2,b_3)$ and $(\pi\times\pi\times\pi)(g)=(g_1,g_2,g_3)$. Since $(g_1,g_2,g_3)\neq (e,e,e)$, we have that $g_i\neq e$ for some $i=1,2,3$.

\vspace{3mm}

Suppose $g_1\neq e$. For $h\in B$ denote $s_h=\|\mathbb{E}_{\tilde{\mathcal{Q}}}(xu_{(\widehat{\delta_{1}(h)},\widehat{\delta_{2}(h)^{-1}},\widehat{\delta_{3}(h)})}zu_{(\widehat{\delta_{1}(h)},\widehat{\delta_{2}(h)^{-1}},\widehat{\delta_{3}(h)})}^*y)\|$. If $s_h=0$ for all $h\in B$, then the assertion follows. Otherwise, if $s_h\neq 0$ for $h\in B$, then $a_1\delta_1(h)g_1\delta_1(h)^{-1}b_1=e$. In particular, there exists $l\in B$ with $a_1lg_1l^{-1}b_1=e$. If $s_{h_m}\neq 0$ for some $m\in\N$, then $\delta_1(h_m)g_1\delta_1(h_m)^{-1}=a_1^{-1}b_1^{-1}=lg_1l^{-1}$ and therefore, $\delta_1(h_m)\in lB_0$ for $B_0=C_B(g_1)$. 

\vspace{3mm}

Let $G_0=\pi^{-1}(B_0)$. By \eqref{conjug} we obtain that $\Delta(u_{h_m})w_{h_m}\in (u_{\hat{l}}\otimes 1\otimes 1)(\cL(G_0)\:\overline{\otimes}\:\cL(G)^{\text{op}}\:\overline{\otimes}\:\cL(G))$ for any such $m$. Since $g_1\neq e$ and $B$ is ICC, $B_0\leqslant B$ is infinite index. Thus, \eqref{Deltaw_htozero} implies that $\{m\in\N:s_{h_m}\neq 0\}$ is finite, proving that $s_{h_m}\to 0$.

\vspace{2mm}

The claim follows similarly if $g_3\neq e$. Now, assume $g_2\neq e$. Suppose $s_{h_m}\neq 0$ for some $m\in\N$. Then, $b_2\delta_2(h_m)g_2\delta_2(h_m)^{-1}a_2=e$. In particular, there exists $f\in B$ with $b_2fg_2f^{-1}a_2=e$, and hence, $\delta_2(h_m)g_2\delta_2(h_m)^{-1}=b_2^{-1}a_2^{-1}=fg_2f^{-1}$, and therefore, $\delta_2(h_m)\in fB_0$ for $B_0=C_B(g_2)$. Let $G_0=\pi^{-1}(B_0)$. By \eqref{conjug} we obtain that $\Delta(u_{\widehat{h_m}})w_{h_m}\in (\cL(G)\:\overline{\otimes}\:\cL(G_0)^{\text{op}}\overline{\otimes}\:\cL(G))(1\otimes u_{\widehat{f}}\otimes 1)$ for any such $m$. Thus, \eqref{Deltaw_htozero} implies that $\{m\in\N:s_{h_m}\neq 0\}$ is finite, proving that $s_{h_m}\to 0$.
\end{subproof}

Next, the prior claim implies $w_h\in\tilde{\mathcal{Q}}$ for every $h\in B$. Thus, $\eta_h:=\zeta_h\text{Ad}(u_{(\widehat{\delta_{1}(h)},\widehat{\delta_{2}(h)^{-1}},\widehat{\delta_{3}(h)})})(w_h^*)\in\tilde{\mathcal{Q}}$ and also 

\begin{equation*}
    \Delta(u_{\hat{h}})=\eta_h(u_{\widehat{\delta_1(h)}}\otimes \overline u_{\widehat{\delta_2(h)^{-1}}}\otimes u_{\widehat{\delta_3(h)}}),\quad\text{for every } h\in B.
\end{equation*}

\vspace{1mm}

\begin{claim}\label{assumemapsareidentity} We may assume that $\delta_1=\delta_2=\delta_3={\rm Id}_B$.
\end{claim}

\begin{subproof}[Proof of Claim \ref{assumemapsareidentity}] First, we argue that we may assume that $\delta_1=\delta_3=\text{Id}_B$. Using the flip automorphism $x\otimes y\otimes z\mapsto z\otimes y \otimes x$ and the same argument form the first part of Step 6 in the proof of \cite[Theorem 1.3]{cios22} we obtain that $\delta_1$ is conjugated to $\delta_3$. Hence, after replacing $\eta_h$ by an appropriate unitary, we may assume that $\delta_1=\delta_3=:\delta$ and there is $x,y\in \mathcal N_{\tilde \cM}(\tilde Q)$ such that 

\begin{equation*}
    \Delta(u_{\hat{h}})x=\eta_hy(u_{\widehat{\delta(h)}}\otimes \overline u_{\widehat{\delta_2(h)^{-1}}}\otimes u_{\widehat{\delta(h)}}),\quad\text{for all } h\in B.
\end{equation*}

\vspace{2mm}

Letting $x_1 = (\Delta \otimes {\rm id}\otimes {\rm id})(x)(x\otimes  1\otimes 1)$ and $\tilde \zeta_h =(\Delta\otimes\text{id}\otimes\text{id})(\eta_hy)(\eta_{\delta(h)}y\otimes 1\otimes 1)$ we get 

\begin{align*}
    (\Delta\otimes\text{id}\otimes\text{id})(\Delta(u_{\hat{h}}))x_1&=(\Delta\otimes\text{id}\otimes\text{id})(\eta_hy)(\Delta\otimes\text{id}\otimes\text{id})(u_{\widehat{\delta(h)}}\otimes \overline u_{\widehat{\delta_2(h)^{-1}}}\otimes u_{\widehat{\delta(h)}})(x\otimes 1\otimes 1)\\
    &=(\Delta\otimes\text{id}\otimes\text{id})(\eta_hy)(\Delta(u_{\widehat{\delta(h)}})x \otimes \overline u_{\widehat{\delta_2(h)^{-1}}}\otimes u_{\widehat{\delta(h)}})\\
    &=\tilde \zeta_h(u_{\widehat{\delta(\delta(h))}}\otimes \overline u_{\widehat{\delta_2(\delta(h))^{-1}}}\otimes u_{\widehat{\delta(\delta(h))}}\otimes \overline u_{\widehat{\delta_2(h)^{-1}}}\otimes u_{\widehat{\delta(h)}}).
\end{align*}

\vspace{2mm}

Similarly, letting $x_2 = ({\rm id}\otimes {\rm id}\otimes \Delta)(x)(1\otimes 1\otimes x)$ and $\tilde \xi_h =(\text{id}\otimes\text{id}\otimes\Delta)(\eta_hy)(1\otimes 1\otimes \eta_{\delta(h)}y)$ yields 

\begin{align*}
    (\text{id}\otimes\text{id}\otimes\Delta)(\Delta(u_{\hat{h}}))x_2 &=(\text{id}\otimes\text{id}\otimes\Delta)(\eta_hy)(u_{\widehat{\delta(h)}}\otimes \overline u_{\widehat{\delta_2(h)^{-1}}}\otimes \Delta(u_{\widehat{\delta(h)}})x)\\
    &=\tilde \xi_h(u_{\widehat{\delta(h)}}\otimes \overline u_{\widehat{\delta_2(h)^{-1}}}\otimes u_{\widehat{\delta(\delta(h))}}\otimes \overline u_{\widehat{\delta_2(\delta(h))^{-1}}}\otimes u_{\widehat{\delta(\delta(h))}}).
\end{align*}

\vspace{2mm}

Using \eqref{commult2} in combination with the two relations above, for every $h\in B$ we have

\begin{equation}\label{equaltensors1}
    (u_{\widehat{\delta(\delta(h))}}\otimes \overline u_{\widehat{\delta_2(\delta(h))^{-1}}}\otimes u_{\widehat{\delta(\delta(h))}}\otimes \overline u_{\widehat{\delta_2(h)^{-1}}}\otimes u_{\widehat{\delta(h)}}) x_1^*x_2= \tilde \zeta_h^*\tilde \xi_h(u_{\widehat{\delta(h)}}\otimes \overline u_{\widehat{\delta_2(h)^{-1}}}\otimes u_{\widehat{\delta(\delta(h))}}\otimes \overline u_{\widehat{\delta_2(\delta(h))^{-1}}}\otimes u_{\widehat{\delta(\delta(h))}}).
\end{equation} 

\vspace{2mm}

Denote by $\mathscr M :=\cM\:\overline{\otimes}\:\cM^{\text{op}}\:\overline{\otimes}\:\cM\:\overline{\otimes}\cM^{\text{op}}\:\overline{\otimes}\cM$ and $\mathscr Q:=\cL(A^{(I)})\:\overline{\otimes}\:\cL(A^{(I)})^{\text{op}}\:\overline{\otimes}\:\cL(A^{(I)})\:\overline{\otimes}\:\cL(A^{(I)})^{\text{op}}\:\overline{\otimes}\:\cL(A^{(I)})$. Also let $\tilde{\tilde{B}}:=B \times B \times B\times B \times B$. 

\vspace{2mm}

For a set $F=F_1\times \cdots\times F_5\subset \tilde{\tilde{B}}$, let $\mathcal{H}_{F}$ be the $\|\cdot\|_2$-closure of the linear span of 

\begin{equation*}
    \{u_{g_1}\otimes \overline{u}_{g_2}\otimes u_{g_3}\otimes \overline{u}_{g_4}\otimes u_{g_5}:g_i\in\pi^{-1}(F_i), \:1\leq i \leq 5\}
\end{equation*}

\vspace{2mm}

and $P_{F}$ be the orthogonal projection from $L^2(\mathscr M)$ onto $\mathcal{H}_{F}$. For later use, notice that for every pair of subsets $G=G_1\times \cdots \times G_5, \: F =F_1\times \cdots \times F_5\subset \tilde{\tilde{B}}$ the following hold:

\begin{enumerate}
    \item\label{orthogonalproj} If $F_j\cap G_j =\emptyset$ for some $1\leq j\leq 5$ then  we have $P_F P_G =0$.
    \item\label{secondproperty} For every $(g_1, ...,g_5),\: (h_1,..., h_n)\in \tilde{\tilde{B}}$ we have  
    
    \begin{equation*}
        (u_{g_1}\otimes \overline u_{g_2} \otimes u_{g_3} \otimes \overline u_{g_4} \otimes u_{g_5} )\mathcal H_F (u_{h_1}\otimes \overline u_{h_2} \otimes u_{h_3} \otimes \overline u_{h_4} \otimes u_{h_5}) = \mathcal H_{g_1F_1h_1\times h_2F_2g_2\times g_3F_3h_3\times h_4F_4g_4\times g_5F_5h_5}.
    \end{equation*}
\end{enumerate}

\vspace{1mm}

Next, since $\mathcal{H}_{F}$ is a $\mathscr Q$-bimodule, using the definitions of $\eta_h$, $\tilde{\xi_h}$, $\tilde{\zeta_h}$, $x_1$ and $x_2$ one can find a finite set $F=F_1\times \cdots\times F_5\subset \tilde{\tilde{B}}$ such that 

\begin{equation*}
    \| x_1x_2^*-P_F(x_1x_2^*)\|_2<1/4, \text{ and } \|\tilde{\zeta_h}^*\tilde{\xi_h}-P_{F}(\tilde{\zeta_h}^*\tilde{\xi_h})\|_2<1/4,\text{ for every }h\in B.
\end{equation*} 

\vspace{2mm}

These inequalities combined with \eqref{equaltensors1} show that for all $h\in B$ we have 

\begin{equation*}
    \scriptstyle{\langle P_{F}(\tilde{\zeta_h}^*\tilde{\xi_h})(u_{\widehat{\delta(h)}}\otimes \overline u_{\widehat{\delta_2(h)^{-1}}}\otimes u_{\widehat{\delta(\delta(h))}}\otimes \overline u_{\widehat{\delta_2(\delta(h))^{-1}}}\otimes u_{\widehat{\delta(\delta(h))}}), (u_{\widehat{\delta(\delta(h))}}\otimes \overline u_{\widehat{\delta_2(\delta(h))^{-1}}}\otimes u_{\widehat{\delta(\delta(h))}}\otimes \overline u_{\widehat{\delta_2(h)^{-1}}}\otimes u_{\widehat{\delta(h)}})P_{F}(x_1^*x_2)\rangle \geq1/2.}
\end{equation*}

\vspace{2mm}

This relation together with properties \eqref{orthogonalproj} and \eqref{secondproperty} above imply that $\delta(\delta(h))F_1\cap  F_1\delta(h)\neq \emptyset$, for every $h\in B$. Since $\delta(B)\leqslant B$ has finite index and $B$ is ICC, by \cite[Lemma 7.1]{bv13} one can find $g\in B$ such that $\delta(h)=ghg^{-1}$, for every $h\in B$. Thus, after replacing $\eta_h$ with an appropriate unitary, we may assume $\delta_1=\delta_3=\delta=\text{Id}_B$ and there is $z,w\in\mathcal{N}_{\tilde{\cM}}(\tilde Q)$ such that 

\begin{equation*}
    \Delta(u_{\hat{h}})z=\eta_hw(u_{\hat{h}}\otimes \overline u_{\widehat{\delta_2(h)^{-1}}}\otimes u_{\hat{h}}),\quad\text{for all }h\in B.
\end{equation*}

\vspace{2mm}

To show the map $\delta_2$ can be conjugated to the identity map, we will make use of relation \eqref{commult2F}. First, letting $z_1=F_{1,2}\circ (\Delta\otimes \text{id}\otimes\text{id})(z)F_{1,2} (z\otimes 1\otimes 1)$, $\zeta_h=F_{1,2}\circ (\Delta\otimes\text{id}\otimes\text{id})(\eta_hw)F_{1,2}(\eta_{h}w\otimes 1\otimes 1)$, and using multiplicativity of $F_{1,2}$ and \eqref{sflip} we get

\begin{align*}
    F_{1,2}\circ (\Delta\otimes\text{id}\otimes\text{id})\Delta(u_{\hat{h}})z_1&=\zeta_hF_{1,2}(u_{\hat{h}}\otimes \overline u_{\widehat{\delta_2(h)^{-1}}}\otimes u_{\hat{h}}\otimes \overline u_{\widehat{\delta_2(h)^{-1}}}\otimes u_{\hat{h}})\\
    &=\zeta_h(u_{\widehat{\delta_2(h)^{-1}}}^*\otimes \overline {u_{\hat{h}}^*}\otimes u_{\hat{h}}\otimes \overline u_{{\widehat{\delta_2(h)^{-1}}}}\otimes u_{\hat{h}}),
\end{align*}

\vspace{2mm}

Similarly, letting $z_2=F_{3,4}\circ (\text{id}\otimes \text{id}\otimes \Delta)(z)F_{3,4}(1\otimes 1\otimes z)$ and $\xi_h=F_{3,4}\circ (\text{id}\otimes\text{id}\otimes\Delta)(\eta_hw)F_{3,4}(1\otimes 1\otimes \eta_{h}w)$ we get

\begin{align*}
    F_{3,4}\circ (\text{id}\otimes\text{id}\otimes\Delta)\Delta(u_{\hat{h}})z_2&=\xi_h(u_{\hat{h}}\otimes \overline u_{\widehat{\delta_2(h)^{-1}}}\otimes u_{\widehat{\delta_2(h)^{-1}}}^*\otimes \overline{u_{\hat{h}}^*}\otimes u_{\hat{h}}).
\end{align*}

\vspace{2mm}

Using these relations in combination with \eqref{commult2F} we get  for all $h\in B$ we have 

\begin{equation}\label{equaltensor3}
    (u_{\widehat{\delta_2(h)^{-1}}}^*\otimes \overline {u_{\hat{h}}^*}\otimes u_{\hat{h}}\otimes \overline u_{{\widehat{\delta_2(h)^{-1}}}}\otimes u_{\hat{h}})z_1^*z_2 = \zeta^*_h \xi_h (u_{\hat{h}}\otimes \overline u_{\widehat{\delta_2(h)^{-1}}}\otimes u_{\widehat{\delta_2(h)^{-1}}}^*\otimes \overline{u_{\hat{h}}^*}\otimes u_{\hat{h}}).
\end{equation}

\vspace{2mm}

Let $F=F_1\times \cdots\times F_5\subset \tilde{\tilde{B}}$ be a finite set for which 

\begin{align*}
    \|z_1^*z_2-P_F(z_1^*z_2)\|_2<1/2,\text{ and } \|\zeta^*_h \xi_h-P_F(\zeta^*_h \xi_h)\|_2<1/2\text{, for every } h\in B.
\end{align*}

\vspace{2mm}

Just as in the prior case, we obtain

\begin{equation*}
    \langle P_{F}(\zeta_h^*\xi_h)(u_{\hat{h}}\otimes \overline u_{\widehat{\delta_2(h)^{-1}}}\otimes u_{\widehat{\delta_2(h)^{-1}}}^*\otimes \overline u_{\hat{h}}^*\otimes u_{\hat{h}}), (u_{\widehat{\delta_2(h)^{-1}}}^*\otimes \overline u_{\hat{h}}^*\otimes u_{\hat{h}}\otimes \overline u_{\widehat{\delta_2(h)^{-1}}}\otimes u_{\hat{h}})P_{F}(z_1^*z_2)\rangle \geq1/2.
\end{equation*}

\vspace{2mm}

Therefore,  using the prior equation and properties \eqref{orthogonalproj} and \eqref{secondproperty}, we obtain that $F_1h\cap \delta_2(h)F_1\neq\emptyset$, for every $h\in B$. Since the map $B \ni h\rightarrow \delta_2(h)\in B$ is a group homomorphism whose image has finite index in $B$, an ICC group, by \cite[Lemma 7.1]{bv13} one can find $l\in B$ such that $\delta_2(h)=lhl^{-1}$, or $\delta_2(h)^{-1}=lh^{-1}l^{-1}$, for every $h\in B$. 

\vspace{2mm}

Hence, after replacing $\Delta$ and $\eta_h$ with the appropriate conjugate, we may assume $\delta_2=\text{Id}_B$, and 

\begin{equation*}
    \Delta(u_{\hat{h}})=\eta_h(u_{\hat{h}}\otimes \overline u_{\widehat{h^{-1}}}\otimes u_{\hat{h}}),\quad\text{for all }h\in B.
\end{equation*}\end{subproof}

To finish the proof,
let $g\in G$. Let $h=\epsilon(g)\in B$ and $a=g\widehat{h}^{-1}\in A^{(I)}$.
Then $\Delta(u_g)=\Delta(u_a)\Delta(u_{\widehat{h}})=\Delta(u_a)\eta_h(u_{\widehat{h}}\otimes \overline u_{\widehat{h^{-1}}}\otimes u_{\widehat{h}})=\Delta(u_a)\eta_h(u_{a^{-1}}\otimes \overline u_{gc(h^{-1},h)^{-1}g^{-1}a}\otimes u_{a^{-1}})(u_g\otimes \overline u_{g^{-1}}\otimes u_g)$. Thus, if we denote $r_g=\Delta(u_a)\eta_h(u_{a^{-1}}\otimes \overline u_{gc(h^{-1},h)^{-1}g^{-1}a}\otimes u_{a^{-1}})$, then $r_g\in \mathscr U(\tilde{\mathcal Q})$ and 

\begin{equation}\label{w_g'}
    \Delta(u_g)=r_g(u_g\otimes \overline u_{g^{-1}}\otimes u_g), \quad \text{for every } g\in G.
\end{equation}

\vspace{2mm}

Consider the action $G\curvearrowright^{\gamma} \tilde{\mathcal Q}$ given by $\gamma_g=\text{Ad}(u_g\otimes \overline u_{g^{-1}} \otimes u_g)$, for $g\in G$. By \cite[Lemma 3.4]{cios22} $\gamma$ is isomorphic to an action $G\curvearrowright (Y^J,\nu^J)$ built over $G\curvearrowright J=I\times\{1,2,3\}$ given by $g\cdot (i,j)=(\epsilon(g)\cdot i,j)$, for every $g\in G$ and $(i,j)\in J$. Since the action $G\curvearrowright J$ has infinite orbits, $\gamma$ is weakly mixing. Moreover, \eqref{w_g'} gives that $r_{gh}=r_g\gamma_g(r_h)$, for every $g,h\in G$. Therefore, $(r_g)_{g\in G}$ is a $1$-cocycle for $\gamma$.

\vspace{2mm}

Since $G$ has property (T), \cite[Theorem 3.6]{cios22} gives $s\in \mathscr U(\tilde{\mathcal Q})$ and a homomorphism $\xi\colon G\to \mathbb T$ such that $r_g= s^*\gamma_g(s)$, for every $g\in G$. Thus, after replacing $\Delta$ by $\text{Ad}(s)\circ\Delta$, \eqref{w_g'} rewrites as   

\begin{equation}\label{xi_g}
    \Delta(u_g)=\xi_gu_g\otimes \overline u_{g^{-1}}\otimes u_g,\quad\text{ for every } g\in G.
\end{equation}

\vspace{2mm}

Lemma \ref{positiveheight} shows that the height $h_H(G)>0$. Since $G$ is ICC, the unitary representation $(\text{Ad}(u_g))_{g\in G}$ of $G$ on $L^2(\mathcal M)\ominus\mathbb C1$ is weakly mixing, and since $H$ is ICC, $\cL_c(H)\not\prec\cL_c(C_H(k))$ for $k\neq e$. By applying \cite[Theorem 4.1]{DV24} we conclude that there exists a unitary $w\in\mathcal M$ and an isomorphism $\rho\colon G\rightarrow H$ such that $u_g=\xi_gwv_{\rho(g)}w^*$, for every $g\in G$.
\end{proof}

\vspace{0.5mm}

\section{Proofs of \ref{theoremB} and Corollary D}

In the context given in \ref{theoremA}, we demonstrate that our group $H$ splits up to a finite normal subgroup when the wreath-like product group $W$ is perfect. Moreover, when the conditions specified in \ref{theoremB} are met, we establish that this splitting is total. 

\begin{proof}[\textbf{Proof of \ref{theoremB}}] By \ref{theoremA} we have that $H\cong H^{hfc}\rtimes_{\beta, c} H/H^{hfc}$ and $F:=[H^{fc},H^{fc}]\leqslant H^{hfc}$ is finite normal subgroup of $H$. 

\vspace{2mm}

To prove our result, we must establish that $H/H^{hfc}$ is isomorphic to $W$, $Z(H)=H^{hfc}$, and that $H$ can be decomposed into a direct product of these groups. We achieve this through the following three successive claims.

\vspace{2mm}

Throughout the proof we will denote by  $\tilde c \,:\, (H/H^{hfc})\times (H/H^{hfc})\to \mathscr U (\cL(H^{hfc}))$ the $2$-cocycle induced by the group cocycle $c$; that is, $\tilde c(g,h)=u_{c(g,h)}$ for all $g,h \in H/H^{hfc}$.

\vspace{0.5mm}

\begin{claim}\label{quotientiswreathlikeprod} $H/H^{hfc}\cong W$.
\end{claim}

\begin{subproof}[Proof of Claim \ref{quotientiswreathlikeprod}] Observe that $H/F=(H^{hfc}/F)\rtimes_{\alpha,d}H/H^{hfc}$ is a group extension. Letting $z=|F|^{-1}\sum_{f\in F}u_f\in\mathscr{P}(\mathscr{Z}(\cL(H)))$, Proposition \ref{cutdown} further yields the following  cocycle crossed product decomposition $\cL(H)z\cong \cL(H^{hfc})z\rtimes_{\beta,\tilde{d}}H/H^{hfc}$. In addition, we have $\cL(H^{hfc})z\cong \cL(H^{hfc}/F)$.

\vspace{2mm}

From \ref{theoremA} one can find orthogonal projections $s_i\in \cL(H^{fc})z$ with $\sum_i s_i=z$ such that $\cL(H^{fc}/F)=\oplus_{i=1}^{\infty}\mathscr{Z}(\cL(H/F))s_i$. 
Let $\cL(H^{hfc})z=\int_X^{\oplus}\mathcal{D}_xd\mu(x)$ be the integral decomposition over the center. Since $H^{fc}\leqslant H^{hfc}$ is finite index by \ref{theoremA}, by Lemma \ref{completelyatomicfibers} we obtain that the fibers $\mathcal{D}_x$ are completely atomic for almost every $x\in X$.

\vspace{2mm}

Using Propositions \ref{integraldecompcocycles} and \ref{isomcocycleintegral}, we further get

\begin{equation*}
    \cL(H)z\cong\int_X^{\oplus}\mathcal{D}_x\rtimes_{\beta_x,\tilde{d}_x}H/H^{hfc}\:d\mu(x).
\end{equation*}

\vspace{2mm}

Moreover, by assumption, $\cL(H)=\cL(A)\:\overline{\otimes}\:\cL(W)$ and thus,

\begin{equation*}
    \int_X^{\oplus}\cL(W)d\mu(x)=\cL(A)z\:\overline{\otimes}\:\cL(W)\cong\int_X^{\oplus}\mathcal{D}_x\rtimes_{\beta_x,\tilde{d}_x}H/H^{hfc}\:d\mu(x).
\end{equation*}

\vspace{2mm}

By \cite[Theorem 2.1.14]{spaas}, the fibers are isomorphic; i.e. $\cL(W)\cong\mathcal{D}_x\rtimes_{\beta_x,\tilde{d}_x}H/H^{hfc}$ for almost every $x\in X$. In particular, $\mathcal{D}_x\rtimes_{\beta_x,\tilde{d}_x}H/H^{hfc}$ is a II$_1$ factor, and hence, by Lemma \ref{finitedimfromcompatomic}, $\mathcal{D}_x$ is finite dimensional. From Proposition \ref{amplificationoffinitebyiccvn}, we find a subgroup $H_x\leqslant H/H^{hfc}$ of index $n_x\in\mathbb{N}$, $l_x\in\mathbb{N}$ and a 2-cocycle $\eta_x:H_x\times H_x\to\mathbb{T}$ satisfying $\mathcal{D}_x=\mathbb{M}_{l_x}\:\overline{\otimes}\:\mathbb{D}_{n_x}$ and $\mathcal{D}_x\rtimes_{\beta,\tilde{d}_x}H/H^{hfc}\cong\cL_{\eta_x}(H_x)^{n_xl_x}$. 

\vspace{2mm}

Therefore, $\cL(W)\cong\cL_{\eta_x}(H_x)^{n_xl_x}$. By \ref{notrivialcompressions}, $n_xl_x=1$, $\eta_x$ is a 2-coboundary and $W\cong H_x=H/H^{hfc}$. Moreover, since $n_x=1$ and $l_x=1$, $\mathcal{D}_x=\C$ and $\eta_x=\tilde{d}_x$ is a 2-coboundary. In particular, $\cL(H^{hfc}/F)=\mathscr{Z}(\cL(H/F))$ implies $H^{hfc}/F$ is the center of $H/F$.
\end{subproof}





\vspace{0.3mm}

\begin{claim}\label{splittingtildeH} $H/F\cong (H^{hfc}/F)\times W$.
\end{claim}

\begin{subproof}[Proof of Claim \ref{splittingtildeH}] By Claim \ref{quotientiswreathlikeprod} and Proposition \ref{integ2coboundaries}, $H/F=H^{hfc}/F\rtimes_{d}H/H^{hfc}$ is a central extension where the induced 2-cocycle $\tilde{d}:H/H^{hfc}\times H/H^{hfc}\to\mathscr{U}(\cL(H^{hfc}/F))$ is of the form $\tilde d(g,h)=\xi_g\xi_h\xi_{gh}^*$ for a function $\xi:H/H^{hfc}\to \mathscr{U}(\cL(H^{hfc}/F))$.

\vspace{2mm}

To prove our claim, it suffices to establish that $\xi_g\in \cL(H^{hfc}/F)$ is a unitary implemented by a group element of $H^{hfc}/F$, for all $g\in H/H^{hfc}$. Let $\Delta:\cL(H^{hfc}/F)\to\cL(H^{hfc}/F\times H^{hfc}/F)$ be the canonical comultiplication map; that is, the $*$-homomorphism given by $\Delta(u_g) = u_g \otimes u_g$, for $g \in H^{hfc}/F$. As $\tilde d(g,h)$ is a group unitary we see that

\begin{equation}\label{comultiplicationcocyclep}
    \Delta(\xi_g\xi_h\xi_{gh}^*)=\Delta(\tilde d(g,h))=\tilde d(g,h)\otimes \tilde d(g,h)=\xi_g\xi_h\xi_{gh}^*\otimes \xi_g\xi_h\xi_{gh}^*,\:\text{for all }g,h\in H/H^{hfc}.
\end{equation}

\vspace{2mm}

From \eqref{comultiplicationcocyclep}, we obtain

\begin{equation}\label{grhomp}
    \Delta(\xi_{gh}^*)(\xi_{gh}\otimes\xi_{gh})=\Delta(\xi_g^*)(\xi_g\otimes \xi_g)\:\Delta(\xi_h^*)(\xi_h\otimes\xi_h),\quad\text{for all }g,h\in H/H^{hfc}.
\end{equation}

\vspace{2mm}

Define $\eta:H/H^{hfc}\to\mathscr{U}(\cL(H^{hfc}/F\times H^{hfc}/F))$ by $\eta_g:=\Delta(\xi_g^*)(\xi_g\otimes \xi_g)$, and notice \eqref{grhomp} shows that $\eta$ is a group homomorphism. As $H/H^{hfc}$ has trivial abelianization, we obtain $\eta_g=1$ for all $g\in H/H^{hfc}$. Thus, $\Delta(\xi_g^*)(\xi_g\otimes\xi_g)=1$ and hence

\begin{equation*}
    \Delta ( \xi_g)=\xi_g \otimes \xi_g,\text{ for all }g\in H/H^{hfc}.
\end{equation*}

\vspace{2mm}

Using \cite[Lemma 7.1]{ipv10} we conclude that $\xi_g$ is a group unitary for every $g\in H/H^{hfc}$. This implies that $H/F\cong (H^{hfc}/F)\times (H/H^{hfc})\cong (H^{hfc}/F)\times W$. 
\end{subproof}

The prior claim implies the cocycle $\tilde{c}$ untwists on $\cL(H^{hfc}/F)\cong\cL(H^{hfc})p$, and therefore, $\tilde{c}$ is supported only on $\cL(F)$.

\vspace{1mm}

\begin{claim}\label{finiteconjiscenter} $H^{hfc}=Z(H)$.
\end{claim}

\begin{subproof}[Proof of Claim \ref{finiteconjiscenter}] Using Theorem \ref{centerLH} we can decompose 

\begin{equation*}
    \cL(H)\cong\bigoplus_{i=1}^k\mathbb{M}_{n_i}(\C)\otimes(\mathbb{D}_{m_i}\rtimes_{\alpha_i,c_i}(H^{hfc}/F\times H/H^{hfc})).
\end{equation*}

\vspace{2mm}

In addition, we write $\mathbb{D}_{m_i}=\bigoplus_{t=1}^{t_i}\mathbb{D}_{s_i}^t$ for which there exists finite index subgroups $A_t^i$ of $H^{hfc}/F$ that act transitively on $\mathbb{D}_{s_i}^t$. Moreover, we obtain unitaries $x_{t,a}^i\in\mathbb{D}_{s_i}^t$ and $A_0\leqslant H^{hfc}/F$ finite index for which 

\begin{equation*}
    \mathscr{Z}(\cL(H))\cong\left\{\left(\sum_ix_a^i\right)a:a\in A_0\right\}\subset \mathscr{Z}(\cL(F\rtimes_{\alpha,c}H^{hfc}/F))\cong \bigoplus_i\bigoplus_t\cL_{\eta_t^i}(A_t^i).
\end{equation*}

\vspace{2mm}

We let $z_i\in\mathscr{Z}(\cL(H))$ be the sum of all projections in an orbit of the action $H^{hfc}/F\times H/H^{hfc}\curvearrowright\mathbb D_{m_i}$. The prior relations imply that the von Neumann subalgebras $\cL(F\rtimes_{\alpha,c}H/H^{hfc})z_i$, $\mathscr{Z}(\cL(H))z_i$ commute and, in addition, the conditional expectation satisfies 

\begin{equation*}
    \mathbb{E}_{\cL(F\rtimes_{\alpha,c}H/H^{hfc})z_i}(\mathscr{Z}(\cL(H))z_i)=\C z_i.
\end{equation*}

\vspace{2mm}

Altogether, this means that $\mathscr{Z}(\cL(H))z_i$ and $\cL(F\rtimes_{\alpha,c}H/H^{hfc})z_i$ are in tensor position; that is, $\mathscr{Z}(\cL(H))z_i\vee \cL(F\rtimes_{\alpha,c}H/H^{hfc})z_i\cong \mathscr{Z}(\cL(H))z_i\:\overline{\otimes}\: \cL(F\rtimes_{\alpha,c}H/H^{hfc})z_i$.

\vspace{2mm}

Let $r_i\in\cL(A)$ be the image of $z_i\in\mathscr{Z}(\cL(H))$ under the inverse isomorphism between $\cL(G)\cong\cL(H)$. As $\tilde{c}$ is only supported on $\cL(F)$, we have the following canonical trace-preserving finite index inclusions of von Neumann algebras which preserve the centers of the compressions $\cL(A)r_i\cong\mathscr{Z}(\cL(H))z_i$:

\begin{equation}\label{isomofcorners}
    \cL(A)r_i\:\overline{\otimes}\:\cL(W)\cong\cL(A\times W)r_i\cong \cL(H)z_i\supseteq \mathscr{Z}(\cL(H))z_i\:\overline{\otimes}\:\cL(F\rtimes_{\alpha,c}H/H^{hfc})z_i.
\end{equation}

\vspace{2mm}

Using Theorem \ref{finiteindexfibers} we conclude $\cL(W)$ contains isomorphic copies of $\cL(F\rtimes_{\alpha,c}H/H^{hfc})z_i$ as finite index subalgebras. By Proposition \ref{amplificationoffinitebyiccvn}, since the von Neumann algebra $\cL(F\rtimes_{\alpha,c}H/H^{hfc})z_i$ is a factor, there exists $n,m\in\N$ with $\cL(F)z_i=\mathbb{M}_n(\C)\otimes\mathbb{D}_m$, a finite index subgroup $H_0\leqslant H/H^{hfc}$ and a scalar 2-cocycle $v:H_0\times H_0\to\mathbb{T}$ such that

\begin{equation*}
    \cL_v(H_0)\subset\cL(W)^{\frac{s_i}{nm}},
\end{equation*}

\vspace{2mm}

where $s_i$ denotes the trace of $z_i$ in $\cL(W)$. Now notice that since $H_0\leqslant H/H^{hfc}\cong W$ is a finite index subgroup, $H_0$ has the same wreath-like product structure as $W$. Thus, applying Theorem \ref{trivialamplificationembedding} we conclude that $s_i=n=m=1$. Hence, $F=\{1\}$ and so $Z(H)=H^{hfc}=H^{fc}$.
\end{subproof}

\vspace{0.2mm}

Altogether, Claims \ref{splittingtildeH} and \ref{finiteconjiscenter} yield $H\cong Z(H)\times H/Z(H)\cong Z(H)\times W$.

\vspace{3mm}

Finally, we illustrate the form of the isomorphism $\Theta$. Since

\begin{equation*}
    \int_X\cL(W)d\mu\cong\cL(A\times W)\cong\cL(H)=\int_X\cL(H/Z(H))d\mu,
\end{equation*}

\vspace{2mm}

for almost every $x\in X$ there exists a $*$-isomorphism $\Theta_x:\cL(W)\to\cL(H/Z(H))$. By the description of the isomorphism of wreath-like product groups in \cite{cios22b}, there exists $w_x\in \mathscr{U}(\cL(H/Z(H)))$, a group isomorphism $\delta_x:W\to H/Z(H)$ such that $\Theta_x(u_g)=w_xv_{\delta_x(g)}w_x^*$ for all $g\in G$ and for almost every $x\in X$. 

\vspace{2mm}

Since $W$ has property (T), there is at most countably many of these isomorphisms, denote these by $\delta_1,\delta_2,...$ Consider the set

\begin{align*}
    Y_n:&=\{(x,u)\in X\times\mathscr{B}(\ell^2(H/Z(H))): u\in\mathscr{U}(\cL(H/Z(H))_x)\text{ and }\Theta_x(u_g)u=uv_{\delta_n(g)}\text{ for all }g\in G\}\\
    &=\{(x,u):u\in\mathscr{U}(\cL(H/Z(H))_x)\}\cap\bigcap_{g\in G}\{(x,u):\Theta_x(u_g)u=uv_{\delta_n(g)}\}.
\end{align*}

\vspace{2mm}

Since $x\mapsto\Theta_x(u_g)$ and $x\mapsto v_{\delta_n(g)}$ are measurable fields of operators, and $(x,u)\mapsto\Theta_x(u_g)u$ and $(x,u)\mapsto uv_{\delta_n(g)}$ are Borel functions (see \cite[Corollary IV.7.8]{takesaki}), we obtain $Y_n$ is a Borel set. Observe that $\pi_1(Y_n)\subseteq X$, where $\pi_1$ denotes the projection onto the first coordinate. By \cite[Theorem A.16]{takesaki} there exists a measurable crossed section $s_n:\pi_1(Y_n)\to \mathscr{B}(\ell^2(H/Z(H)))$ such that $(x,s_n(x))\in Y_n$ for all $x\in\pi_1(Y_n)$. By the prior paragraph, we see that $\{\pi_1(Y_n)\}_n$ is a measurable partition of $X$. Hence, letting $e_n=\chi_{\pi_1(Y_n)}\in\mathscr{P}(\cL(Z(H)))$ and $w=\int_Xw_xd\mu(x)$ where $w_x=s_n(x)$ for $x\in \pi_1(Y_n)$, we obtain

\begin{equation*}
    \Theta(u_g)=\int_Xw_xv_{\delta_x(g)}w_x^*d\mu(x)=w\left(\int_Xv_{\delta_x(g)}d\mu(x)\right)w^*=w\left(\sum_ne_n\otimes v_{\delta_n(g)}\right)w^*.
\end{equation*}

\vspace{2mm}

Therefore,

\begin{equation*}
    \Theta(xu_g)=\Theta(x)\: w\left(\sum_{n\in\N}e_n\otimes v_{\delta_n(g)}\right)w^*, \quad\text{for all }x\in\cL(A)\text{ and }g\in W.
\end{equation*}\end{proof}

\begin{rem}\label{indepenedentproof} If in the statement of  \ref{theoremB} one assumes in addition that $W$ has torsion free outer automorphism group, then the theorem can be established via a slight variation of the proof above which does not appeal at all to the twisted W$^*$-superrigidity result \ref{notrivialcompressions}. Instead one can just use the regular W$^*$-superrigidity result, \cite{cios22} along with the strong rigidity results, Theorems \ref{untwistcocycle} and \ref{trivialamplificationembedding}, and other basic group properties for deducing splitting of central extensions (see Lemma \ref{splittinglemma}). In fact, this was our initial approach which implicitly is independent of the results in \cite{DV24}. We notice that groups $W$ with the aforementioned properties have been constructed in \cite{amcos23}. 
\end{rem}

Indeed, in the prior proof we use \ref{notrivialcompressions} only to derive Claim \ref{quotientiswreathlikeprod}. Below we explain how this claim can be achieved without it. In our arguments we use the same notations as before. We also provide an alternative argument to the method used in Claim \ref{splittingtildeH} that can be used to show the central extension $\tilde H$ splits over a finite subgroup in its center. 

\vspace{0.5mm}

\begin{claim}\label{quotientiswreathlikeprod'} $H/H^{hfc}\cong W$ and $H/F\cong H^{hfc}/F\times W$.
\end{claim}

\begin{subproof}[Proof of Claim \ref{quotientiswreathlikeprod'}] 


Denote $\tilde{H}:=H/F$ and $\tilde{B}=H^{hfc}/F$. Next we argue that one can find a finite subgroup $F'<\tilde B$ such that $\tilde{H}/F'\cong \tilde{B}/F'\times H/H^{hfc}$. From the first part of Claim \ref{quotientiswreathlikeprod} we obtain a finite index subgroup $H_0\leqslant H/H^{hfc}$ of index $n$ and a 2-cocycle $\eta:H_0\times H_0\to\mathbb{T}$ such that $\cL(W)\cong \cL_{\eta}(H_0)^{nl}$. Here, the natural number $nl$ also denotes the dimension of the fibers of the integral decomposition of $\cL(\tilde{B})$. By Theorem \ref{notrivialcompressions2}, $n=1$ meaning $H_0=H/H^{hfc}$, $\tilde{B}$ is the center of $\tilde{H}$ and $\cL(W)\cong\cL_{\eta}(H/H^{hfc})$. By Lemma \ref{trivialabelianization}, we conclude that $H/H^{hfc}$ has trivial abelianization and property (T); and therefore, Lemma \ref{splittinglemma} ensures the existence of such a splitting.

\vspace{2mm}

Let $p':=|F'|^{-1}\sum_{f\in F'}u_f\in\mathcal{P}(\mathscr{Z}(\cL(\tilde{H})))$. Following a similar reasoning as in the proofs outlined in Section \ref{sectionprop(T)}, we obtain  $\cL(A)p'\:\overline{\otimes}\:\cL(W)\cong\cL(\tilde{H})p'\cong\cL(\tilde{B}/F')\:\overline{\otimes}\:\cL(H/H^{hfc})$ via a trace-preserving isomorphism. Thus, by letting $\cL(A)p'=L^{\infty}(X',\nu)\cong\cL(\tilde{B}/F')$, we obtain

\begin{equation*}
    \int_{X'}\cL(W)d\nu\cong \cL(A)p'\:\overline{\otimes}\:\cL(W)\cong\cL(\tilde{B}/F')\:\overline{\otimes}\:\cL(H/H^{hfc})=\int_{X'}\cL(H/H^{hfc})d\nu.
\end{equation*}

\vspace{2mm}

By \cite[Theorem 2.1.14]{spaas} the fibers are isomorphic, $\mathcal L(W)\cong \mathcal L(H/H^{hfc})$, 
and using the W$^*$-superrigidity of $W$ \cite{cios22} we conclude $W\cong H/H^{hfc}$.

\vspace{2mm}

Notice that $\tilde H = \tilde{B} \rtimes_{\alpha,d} H/H^{hfc}$ is a group extension, where $d$ is the natural $2$-cocycle $d: H/H^{hfc}\times H/H^{hfc}\to \tilde{B}$ (supported on $F'$). To show $\tilde{H}$ splits as a direct product, we show $\tilde{B}=Z(\tilde{H})$ and the cocycle $d$ is trivial. 

\vspace{2mm}

Let $\tilde d: H/H^{hfc}\times H/H^{hfc}\to \mathscr U(\cL(Z(\tilde H)))$ be the $2$-cocycle induced by the group cocycle $d$. From the first paragraph, we obtain that the cocycles $\tilde d_x$ in the integral decomposition of $\tilde d$ satisfy $\cL(W)\cong\cL_{\tilde d_x}(H/H^{hfc})^{n_x}\cong\cL_{\tilde d_x}(W)^{n_x}$ for almost every $x\in X$, where $n_x$ are the dimensions of the fibers of the integral decomposition of $\cL(\tilde{B})$ over the center. From Theorem \ref{trivialamplificationembedding}, $n_x=1$ for almost every $x$; and from Theorem \ref{untwistcocycle}, as $C$ is free abelian, $\tilde d_x$ is trivial for for almost every $x\in X$. Therefore, by Proposition \ref{integ2coboundaries}, the cocycle $\tilde d$ is of the form $\tilde d(g,h)=\xi_g\xi_h\xi_{gh}^*$ for a function $\xi:H/H^{hfc}\to \mathscr{U}(\cL(\tilde{B}))$. To conclude $\tilde{H}$ splits one can just follow the rest of Claim \ref{splittingtildeH}.
\end{subproof}

\vspace{1mm}
 
Even though this case is less general, we believe it might be worth including this line of proof as it hints to a possibly more generic recipe for deciding splitting of nontrivial central extension groups with property (T) central quotient from W$^*$-equivalence. In fact, one only needs that the ICC central quotient is W$^*$-superrigid in the classical sense, its factor has no nontrivial symmetries of finite order and satisfies embedding properties as in Theorems \ref{untwistcocycle} and \ref{trivialamplificationembedding}. It is plausible such groups could exist outside the realm of wreath-like product groups considered here.   

\vspace{0.5mm} 


\vspace{0.5mm}

\begin{rem} We also note that the assumption that $W$ has trivial abelianization is not essential. In fact the proof of \ref{theoremB} can be extended without changing its substance (i.e.\ passing to finite index subgroups of $H/H^{hfc}$ in various places) to accommodate any group $W$ with finite abelianization (always the case as $W$ has property (T)). We leave the details to the reader.
\end{rem}

\vspace{0.5mm}

\begin{proof}[\textbf{Proof of \ref{C*superrigidity}}] First observe from \cite[Theorem 4.1]{bkko16} that the map $\Theta$ is trace-preserving on $C^*_r(G)$ and hence it can be extended to a von Neumann algebra tracial isomorphism $\Theta:\cL(A\times W)\to\cL(H)$.  Then, \ref{theoremB} gives us that $H\cong Z(H)\times H/Z(H)$ with $H/Z(H)\cong W$. Since $\Theta(C_r^*(A))= C^*_r(Z(H))$ and $A$ is an abstractly C$^*$-superrigid group \cite{scheinberg} we get that $Z(H)\cong A$. Therefore, $H\cong G$ as wanted. The rest of the statement follows from \ref{theoremB}.
\end{proof}

\appendix

\section{A compression result for twisted wreath-like product group factors}\label{appendixBB}

This appendix is devoted to proving compression result  of independent interest for wreath like product group factors. This extends an early result of Ioana \cite[Corollary 10.11]{ioa10} and it is obtained by solely combining the earlier methods in \cite[Section ]{ioa10} with the recent developments in \cite{cios22}. We notice that it is used in deriving parts of \ref{notrivialcompressions}. 

\vspace{2mm}

To introduce our result and its proof we recall some notations along with a few useful facts. Let $G$ and $H$ be groups, and let $c: H \times H \to \mathbb T$ be a $2$-cocycle. Assume that $p\in \cL(G)$ is a projection such that  $p\cL(G)p=\cL_{c}(H)$. Following \cite[Lemma 10.9]{ioa10} consider the $\ast$-homomorphism $\Delta:\cL_{c}(H)\to \cL_{c}(H)\:\overline{\otimes}\:\cL_{c}(H)\:\overline{\otimes}\:\cL_{\overline{c}}(H)$ given by $\Delta(v_h)=v_h\otimes v_h\otimes v_h$ for $h\in H$. Let $\Delta_1:=(\text{id}\otimes\text{id}\otimes\Psi_{c})\circ\Delta:\cL_{c}(H)\to \cL_{c}(H)\:\overline{\otimes}\:\cL_{c}(H)\:\overline{\otimes}\:\cL_{c}(H)$ where $\Psi_{c}:\cL_{\overline{c}}(H)\to \cL_{c}(H)$ is the natural $\ast$-isomorphism. Since $p\cL(G)p=\cL_{c}(H)$, we can view $\Delta_1$ as an embedding $\Delta_1:p\cL(G)p\to p\cL(G)p\:\overline{\otimes}\:p\cL(G)p\:\overline{\otimes}\:p\cL(G)p$. By amplifying $\Delta_1$ (see \cite[Remark 4.9]{cios22}) we get a unital $\ast$-embedding $\theta:\cL(G)\to (p\otimes p\otimes 1)(\cL(G)\:\overline{\otimes}\:\cL(G)\:\overline{\otimes}\:\cL(G))(p\otimes p\otimes 1)$ with $\theta(1)=p\otimes p\otimes 1=:q$.

\begin{lem}\label{notintertwiningforinfiniteindex2} (\cite[Lemma 4.10]{cios22}) If $\theta:\cM\to q(\cM\:\overline{\otimes}\:\cM\:\overline{\otimes}\:\cM)q$ is as before, with $\cM=\cL(G)$, then $\theta(\cM)\not\prec\cM\:\overline{\otimes}\:\cM\:\overline{\otimes}\:\mathcal{Q}$ for any infinite index subalgebra $\mathcal{Q}\subset \cM$.
\end{lem}

\begin{proof} Otherwise, we can find $s_i,t_i\in\cM\:\overline{\otimes}\:\cM\:\overline{\otimes}\:\cM$ and $d>0$ with

\begin{equation*}
    \sum_{i=1}^k\|\mathbb{E}_{\cM\overline{\otimes}\cM\overline{\otimes}\mathcal{Q}}(s_i\theta(u)t_i)\|_2^2\geq d>0, \quad\text{for all } u\in\mathscr{U}(\cM).
\end{equation*}

\vspace{2mm}

By approximating the elements $s_i,t_i$ and decreasing $d$, if necessary, we may assume $s_i,t_i\in 1\otimes 1\otimes\cM$. Now, for $u_g\in\cM$, $u_g=\sum_h\tau(u_gv_h^*)v_h$, $\theta(u_g)=\sum_h\tau(u_gv_h^*)(v_h\otimes v_h\otimes v_h)$, and

\begin{align*}
    d&\leq \sum_{i=1}^k\|\mathbb{E}_{\cM\overline{\otimes}\cM\overline{\otimes}\mathcal{Q}}(s_i\theta(u_g)t_i)\|_2^2=\sum_{i=1}^k\left\|\sum_h\tau(u_gv_h^*)(v_h\otimes v_h\otimes \mathbb{E}_{\mathcal{Q}}(s_iv_ht_i))\right\|_2^2\\
    &=\sum_{i=1}^k\sum_h|\tau(u_gv_h^*)|^2\|\mathbb{E}_{\mathcal{Q}}(s_iv_ht_i))\|_2^2,\quad\quad\quad\quad\text{for all } g\in G.
\end{align*}

\vspace{2mm}

This implies that $\sum_h|\tau(u_gv_h^*)|^2\|\mathbb{E}_{\mathcal{Q}}(s_iv_ht_i)\|_2^2\geq d_0>0$ for all $g\in G$ and for all $1\leq i\leq k$. Hence, after doing appropriate approximations, $\sum_{i=1}^k\|\mathbb{E}_{\mathcal{Q}}(s_iut_i)\|_2^2\geq d_0'>0$ for all $u\in\mathscr{U}(\cM)$. This means $\cM\prec \mathcal{Q}$, or $\mathcal{Q}\subset \cM$ is finite index.
\end{proof}

With these preparations at hand we now introduce our result. Our approach follows closely the proof of \cite[Theorem 1.3]{cios22}. However for the reader's convenience we will include all the details. 

\vspace{2mm}

\begin{thm}\label{notrivialcompressions2} Let $G\in\mathcal{WR}(A,B\curvearrowright I )$ be a property (T) group where $A$ is a non-trivial abelian group, $B$ is an ICC subgroup of a hyperbolic group and $B\curvearrowright I$ has amenable stabilizers. Let $H$ be any group and let $c:H\times H \to \mathbb T$ be any $2$-cocycle. If $0<t\leq 1$ be a scalar for which there is a $\ast$-isomorphism $\Phi:\cL(G)^t\to \cL_{c}(H)$, then $t=1$.
\end{thm}

\begin{proof} Assume $p\in \cL(G)$ is projection such that $\tau(p)=t$. Identify $p\cL(G)p$ and $\cL_{c}(H)$ under the $*$-isomorphism $\Phi$. Let 

\begin{equation*}
    \theta:\cL(G)\to (p\otimes p\otimes 1)(\cL(G)\:\overline{\otimes}\:\cL(G)\:\overline{\otimes}\:\cL(G))(p\otimes p\otimes 1)
\end{equation*}

\vspace{2mm}

with $\theta(1)=p\otimes p\otimes 1=:q$ be the unital $\ast$-embedding described at the beginning of this section. We claim that $q$, and hence $p$, is equal to $1$. Consider the group $\tilde{G}=G\times G\times G$. Here we view $\tilde{G}$ as a generalized wreath-like product in $\mathcal{WR}(A,\tilde{B}\curvearrowright \tilde{I})$ where $\tilde{B}=B_1\times B_2\times B_3$, with $B_i\cong B$, and $\tilde{I}=I_1\cup I_2\cup I_3$ with $I_i\cong I$. 

\vspace{2mm}

Now, let 

\begin{equation*}
    \cN:=\cL(G),\quad \tilde{\cN}:=\cL(\tilde{G}),\quad \tilde{\mathcal{Q}}:=\cL(A^{(\tilde{I})}),\quad \mathcal{P}:=\theta(\cL(A^{(I)})),\quad \cM:=\theta(\cL(G)).
\end{equation*}

\vspace{2mm}

Let $\mathcal{R}:=\mathcal{P}'\cap q\tilde{\cN}q\subset \tilde{\cN}$.

\vspace{2mm}

\begin{claim}\label{intertwiningcios2} $\mathcal{R}\prec_{\tilde{\cN}}^s\cL(A^{(\tilde{I})})=\tilde{\mathcal{Q}}$.
\end{claim}

\begin{subproof}[Proof of Claim \ref{intertwiningcios2}] Since $\tilde{\cN}=\cN\:\overline{\otimes}\:\cN\:\overline{\otimes}\:(A^{(I_3)}\rtimes B)=(\cN\:\overline{\otimes}\:\cN\:\overline{\otimes}\:\cL(A^{(I_3)}))\rtimes B$ and $B$ is a subgroup of a hyperbolic group, by \cite[Theorem 3.10]{cios22}, and since $\cM$ has property (T), we have that $\mathcal{P}\prec^s\cN\:\overline{\otimes}\:\cN\:\overline{\otimes}\:\cL(A^{(I_3)})$. Doing the same argument in each tensor, we obtain $\mathcal{P}\prec^s\cL(A^{(I_1)})\:\overline{\otimes}\:\cN\:\overline{\otimes}\:\cN$ and $\mathcal{P}\prec^s\cN\:\overline{\otimes}\:\cL(A^{(I_2)})\:\overline{\otimes}\:\cN$. By \cite[Lemma 2.8(2)]{dhi19} we have $\mathcal{P}\prec^s \cL(A^{(I_1)})\:\overline{\otimes}\:\cL(A^{(I_2)})\:\overline{\otimes}\:\cL(A^{(I_3)})=\cL(A^{(\tilde{I})})=\tilde{\mathcal{Q}}$. 

\vspace{2mm}

As $\theta$ can be seen as a diagonal embedding, we also have that $\mathcal{P}\not\prec\overline{\otimes}_{k\neq j}\cL(A^{(I_k)})$, for all $1\leq j\leq 3$. By \cite[Corollary 4.7]{cios22}, $\mathcal{R}$ is amenable. Moreover, since $\mathcal{R}$ is normalized by $(\theta(u_{\hat{g}}))_{g\in B}$, by the first part of the proof, $\mathcal{R}\prec^s\cL(A^{(\tilde{I})})=\tilde{\mathcal{Q}}$.
\end{subproof}

Since $\tilde{\cQ}\subset\tilde{\cN}$ is a Cartan subalgebra, combining Claim \ref{intertwiningcios2} together with \cite[Lemma 3.7]{cios22} (see also \cite{ioa10}) we obtain that, after replacing $\theta$ with $\text{Ad}(u)\circ\theta$, for some $u\in\tilde{\cN}$, we may assume that $q\in\tilde{\mathcal{Q}}$ and

\begin{equation*}
    \mathcal{P}\subset \tilde{\mathcal{Q}}q\subset\mathcal{R}.
\end{equation*}

\vspace{2mm}

Throughout the proof, we will use of the following notation: for every $g\in B$, let $\hat{g}\in G$ so that $\epsilon(\hat{g})=g$, where $\varepsilon:G\to B$ is the canonical quotient map of the wreath-like product $G\in\mathcal{WR}(A,B\curvearrowright I)$.

\vspace{2mm}

Fixing $g\in B$ notice that $\theta(u_{\hat{g}})$ normalizes $\mathcal{P}$ and $\mathcal{R}$. Denote by $\sigma_g:=\text{Ad}(\theta(u_{\hat{g}}))\in\text{Aut}(\mathcal{R})$. Then, $\sigma=(\sigma_g)_{g\in B}$ defines an action of $B$ on $\mathcal{R}$ which leaves $\mathcal{P}$ invariant. Note that the restriction of $\sigma$ to $\mathcal{P}$ is free because this action is conjugate to an action $B\curvearrowright\cL(A^{(I)})$. By \cite[Lemma 3.8]{cios22}, there exists an action $\beta=(\beta_g)_{g\in B}$ of $B$ on $\mathcal{R}$ such that

\vspace{1mm}

\begin{enumerate}
    \item[(a)] for every $g\in B$ we have $\beta_g=\sigma_g\circ\text{Ad}(w_g)=\text{Ad}(\theta(u_{\hat{g}})w_g)$, for some $w_g\in \mathscr{U}(\mathcal{R})$; and
    \item[(b)] $\tilde{\mathcal{Q}}q$ is $\beta(B)$-invariant and the restriction of $\beta$ to $\tilde{\mathcal{Q}}q$ is free. 
\end{enumerate}

\vspace{1mm}

Our next goal is to apply \cite[Theorem 4.1]{cios22}. Consider the action $\alpha:=(\alpha_h)_{h\in\tilde{B}}$ of $\tilde{B}$ on $\tilde{\mathcal{Q}}$ given by $\alpha_h=\text{Ad}(v_{\hat{h}})$, for $h\in\tilde{B}$. Let $(Y,\nu)$ be the dual of $A$ with its respective Haar measure, and let $(X,\mu):=(Y^{\tilde{I}},\nu^{\tilde{I}})$. Identify $\tilde{Q}=L^{\infty}(X,\mu)$, and denote by $\alpha$ still the corresponding p.m.p. action $\tilde{B}\curvearrowright(X,\mu)$. Let $X_0\subset X$ be a measurable set such that $q=\mathbbm{1}_{X_0}$. Since $\tilde{\mathcal{Q}}q=L^{\infty}(X_0)$ is $\beta(B)$-invariant, we get a measure preserving action $B\curvearrowright^{\beta}(X_0,\mu|_{X_0})$. 

\vspace{2mm}

Notice $\tilde{\mathcal{Q}}\subset \tilde{\cN}$ is a Cartan subalgebra, and the equivalence relation associated to the inclusion is equal to $\mathscr{R}(\tilde{B}\curvearrowright X)$. Since the restriction of $\beta$ to $\tilde{\mathcal{Q}}q$ is implemented by unitaries in $q\tilde{\cN}q$ we deduce that $\beta(B)\cdot x\subset \alpha(\tilde{B})\cdot x$ for almost every $x\in X_0$. Since $B$ has property (T) and $\alpha$ is an action built over $\tilde{B}\curvearrowright I$, using \cite[Lemma 4.4]{cios22} and a maximality argument, we can partition $X_0=\cup_{m=1}^lX_m$ into non-null measurable sets and we can find finite index subgroups $S_m\leqslant B$ such that $X_m$ is $\beta(S_m)$ invariant, the restriction $\beta|_{S_m}$ to $X_m$ is weakly mixing, for all $1\leq m\leq l$, and

\begin{equation}\label{actionsonXm2}
    \beta(S_m)\cdot x\subset \alpha(\tilde{B})\cdot x \text{ for almost every }x\in X_m.
\end{equation}

\vspace{2mm}

Finally, to apply \cite[Theorem 4.1]{cios22} we prove the following claim.

\begin{claim}\label{setstozero2} For every subgroup of infinite index $K_i\leqslant B_i$ there exists a sequence $(h_m)_m\subset B$ such that for every $s,t\in B$ and $1\leq i\leq 3$ we have

\begin{equation*}
    \mu(\{x\in X_0:\beta_{h_m}(x)\in\alpha(B_{\hat{i}}\times sK_it)(x)\})\to 0\text{ as } m\to\infty.
\end{equation*}
\end{claim}

\begin{subproof}[Proof of Claim \ref{setstozero2}] Let $G_i:=\epsilon^{-1}(K_i)$, and notice that $G_i$ is an infinite index subgroup of $G$. By Lemma \ref{notintertwiningforinfiniteindex2} above, as $(\theta(u_g))_{g\in G}$ generate $\cM$, we can find a sequence $(g_m)_m\subset G$ with

\begin{equation*}
    \|\mathbb{E}_{\cN_{\hat{i}}\overline{\otimes}\cL(G_i)}(x\theta(u_{g_m})y)\|_2\to 0,
\end{equation*}

\vspace{2mm}

for all $x,y\in \tilde{\cN}$ and every $1\leq i\leq 3$. Let $h_m:=\epsilon(g_m)\in B$. We claim $(h_m)_m$ is the sequence that satisfies the statement. Since $g_m^{-1}\widehat{h_m}\in A^{(I)}$ and $w_{h_m}\in\mathscr{U}(\mathcal{R})$ we get that $\theta(u_{\widehat{h_m}})w_{h_m}\in\theta(u_{g_m})\mathscr{U}(\mathcal{R})$. Thus, for every $m$, $\theta(u_{\widehat{h_m}})w_{h_m}\in \sum_{i=1}^t\theta(u_{g_m})(\mathcal{P})_1x_i$ for some fixed $x_1,...,x_t\in\mathcal{R}$ (see \cite[Theorem 1.3 Step 2]{cios22}). In particular this implies that

\begin{equation*}
    \|\mathbb{E}_{\cN_{\hat{i}}\overline{\otimes}\cL(G_i)}(x\theta(u_{\widehat{h_m}})w_{h_m}y)\|_2\to 0,
\end{equation*}

\vspace{2mm}

for every $x,y\in\tilde{\mathcal{N}}$ and all $1\leq i\leq 3$. On the other hand, $\beta_{h_m}=\text{Ad}(\theta(u_{\widehat{h_m}})w_{h_m})$ and $\alpha(g_1,g_2,g_3)=\text{Ad}(u_{(\widehat{g_1},\widehat{g_2},\widehat{g_3})})$ for all $(g_1,g_2,g_3)\in\tilde{B}$. These facts imply that $\mu(\{x\in X_0:\beta_{h_m}(x)\in\alpha(B_{\hat{i}}\times sK_it)(x)\})$ is equal to $\|\mathbb{E}_{\cN_{\hat{i}}\overline{\otimes}\cL(G_i)}((u_{\hat{s}}^*\otimes 1\otimes 1)\theta(u_{\widehat{h_m}})w_{h_m}(u_{\hat{t}}^*\otimes 1\otimes 1))\|_2^2$. 
\end{subproof}

Fix $j\in I$, so that $j\in B_i$ for some $1\leq i\leq 3$, and let $g=(g_1,g_2,g_3)\in \tilde{B}\setminus\{1\}$. Then, $\text{Stab}_{\tilde{B}}(j)=B_{\hat{i}}\times \text{Stab}_{B_i}(j)$ and $C_{\tilde{B}}(g)=C_{B_1}(g_1)\times C_{B_2}(g_2)\times C_{B_3}(g_3)$. Since $g\neq 1$, there exists $1\leq i\leq 3$ for which $g_i\neq 1$, and hence, $C_{\tilde{B}}(g)\leqslant B_{\hat{i}}\times C_{B_i}(g_i)$, where $C_{B_i}(g_i)$ is an infinite index subgroup of $B_i$. 

\vspace{2mm}

Let $1\leq m\leq l$. The fact that $\beta|_{S_m}$ is free and weakly mixing on $X_m$, $S_m$ has property (T), equation \eqref{actionsonXm2}, the previous paragraph and Claim \ref{setstozero2} show that the conditions of the moreover assertion of \cite[Theorem 4.1]{cios22} are satisfied by $\beta|_{S_m}$ on $X_m$ and $\alpha$. Thus, we can find an injective group homomorphism $\epsilon_m:S_m\to\tilde{B}$ and $\varphi_m\in \left[\mathscr{R}(\tilde{B}\curvearrowright X)\right]$ such that $\varphi_m(X_m)=X\times\{m\}\equiv X$ and $\varphi_m\circ\beta_h|_{X_m}=\alpha_{\epsilon_m(h)}\circ\varphi_m|_{X_m}$ for all $h\in S_m$. In particular, $\mu(X_m)=1$. This shows that $\mu(X_0)=l\in\N$, and so $l=1$. Therefore, $q=1$ and hence, $p=1$. 
\end{proof}

\vspace{2mm}

\printbibliography

\vspace{5mm}

\end{document}

%% file: preamble.tex
\usepackage{enumitem, amsmath, amsthm, amsfonts, amssymb, xy,  mathrsfs, graphicx, lscape}
\usepackage[usenames, dvipsnames]{color}
\usepackage[margin=1in]{geometry}
\usepackage{tikz}
\usepackage{pgfplots}
\usepackage{mathtools}
\usepackage{comment}
\usetikzlibrary{matrix,arrows, calc}
\pgfplotsset{compat=1.17}
\usepackage{tikz-cd}
\usepackage{xcolor}
\usepackage{stackengine}
\usepackage{graphicx}

\theoremstyle{plain}

\newtheorem{thm}{Theorem}[section]
\newtheorem*{thm*}{Theorem}
\newtheorem{cor}[thm]{Corollary}
\newtheorem{lem}[thm]{Lemma}
\newtheorem{claim}[thm]{Claim}
\newtheorem{prop}[thm]{Proposition}

\theoremstyle{definition}
\newtheorem*{ex}{Example}
\newtheorem{defn}[thm]{Definition}

\theoremstyle{remark}
\newtheorem{rem}[thm]{Remark}

\newtheorem{rmk}[equation]{Remark}

\newtheorem{eg}[equation]{Example}

\newtheorem{ques}[equation]{Question}

\newenvironment{subproof}[1][\proofname]{%
  \begin{proof}[#1]%
}{%
  \end{proof}%
}

\usepackage{bbm}

\usepackage[colorlinks=true, pdfstartview=FitV, linkcolor=black, citecolor=black, urlcolor=black]{hyperref}

\newcommand{\Z}{\mathbb{Z}}
\newcommand{\C}{\mathbb{C}}
\newcommand{\N}{\mathbb{N}}
\newcommand{\R}{\mathbb{R}}

\newcommand{\ngroup}{\trianglelefteq}
\newcommand{\cupm}{\bigcup\limits}

\renewcommand{\phi}{\varphi}
\renewcommand{\emptyset}{\varnothing}

\renewcommand{\tilde}[1]{\widetilde{#1}}

\makeatletter
\def\Ddots{\mathinner{\mkern1mu\raise\p@
\vbox{\kern7\p@\hbox{.}}\mkern2mu
\raise4\p@\hbox{.}\mkern2mu\raise7\p@\hbox{.}\mkern1mu}}
\makeatother

\numberwithin{equation}{section}